\documentclass{scrartcl}

\usepackage{stmaryrd}
\usepackage{amsmath}
\usepackage{amssymb}
\usepackage{amsthm}
\usepackage{xspace}
\usepackage{cleveref}
\theoremstyle{thmstyleone}
\newtheorem{theorem}{Theorem}
\newtheorem{assumption}{Assumption}
\newtheorem{proposition}{Proposition}
\newtheorem{lemma}{Lemma}
\newtheorem{corollary}{Corollary}
\usepackage{xcolor}
\usepackage[numbers]{natbib}

\theoremstyle{thmstyletwo}
\newtheorem{remark}{Remark}
\usepackage[inline]{enumitem}
\usepackage{caption}
\usepackage{subcaption}
\usepackage[normalem]{ulem}

\newcommand{\E}{\mathbb{E}}

\newcommand{\N}{\mathbb{N}} %
\newcommand{\PP}{\mathbb{P}} %
\newcommand{\R}{\mathbb{R}} %

\newcommand{\bb}{\mathbf{b}}

\newcommand{\eb}{\mathbf{e}}
\newcommand{\fb}{\mathbf{f}}
\newcommand{\gb}{\mathbf{g}}

\newcommand{\rb}{\mathbf{r}}

\newcommand{\ub}{\mathbf{u}}
\newcommand{\vb}{\mathbf{v}}
\newcommand{\wb}{\mathbf{w}}
\newcommand{\xb}{\mathbf{x}}
\newcommand{\yb}{\mathbf{y}}

\newcommand{\Ab}{\mathbf{A}}
\newcommand{\Bb}{\mathbf{B}}

\newcommand{\Db}{\mathbf{D}}

\newcommand{\Fb}{\mathbf{F}}
\newcommand{\Gb}{\mathbf{G}}

\newcommand{\Ib}{\mathbf{I}}
\newcommand{\Jb}{\mathbf{J}}

\newcommand{\Qb}{\mathbf{Q}}
\newcommand{\Rb}{\mathbf{R}}

\newcommand{\Ub}{\mathbf{U}}
\newcommand{\Vb}{\mathbf{V}}
\newcommand{\Wb}{\mathbf{W}}

\newcommand{\Zb}{\mathbf{Z}}

\newcommand{\Gcal}{\mathcal{G}}

\newcommand{\Ncal}{\mathcal{N}}

\newcommand{\Qcal}{\mathcal{Q}}

\newcommand{\zetab}{\boldsymbol{\zeta}}
\newcommand{\etab}{\boldsymbol{\eta}}
\newcommand{\thetab}{\boldsymbol{\theta}}

\newcommand{\mub}{\boldsymbol{\mu}}

\newcommand{\Gammab}{\boldsymbol{\Gamma}}

\newcommand{\Pib}{\boldsymbol{\Pi}}

\newcommand{\Sigmab}{\boldsymbol{\Sigma}}

\newcommand{\bfu}{\boldsymbol{u}}

\renewcommand{\le}{\leqslant}
\renewcommand{\ge}{\geqslant}
\renewcommand{\leq}{\leqslant}
\renewcommand{\geq}{\geqslant}

\newcommand*{\argmin}{\mathop{\mathrm{argmin}}}

\newcommand{\diag}{\mathop{\mathrm{diag}}}
\newcommand{\rank}{\mathop{\mathrm{rank}}}
\newcommand{\range}{\mathop{\mathrm{Range}}}

\newcommand{\nucnorm}[1]{{\left\vert\kern-0.25ex\left\vert\kern-0.25ex\left\vert #1 
    \right\vert\kern-0.25ex\right\vert\kern-0.25ex\right\vert}}

\newcommand{\wh}[1]{\widehat{#1}}
\newcommand{\wt}[1]{\widetilde{#1}}

\newcommand{\eg}{\emph{e.g.}\xspace}
\newcommand{\ie}{\emph{i.e.}\xspace}
\newcommand{\iid}{\emph{i.i.d.}\xspace}

\newcommand{\rbr}[1]{\left(#1\right)}
\newcommand{\sbr}[1]{\left[#1\right]}

\newcommand{\nbr}[1]{\left\|#1\right\|}

\newcommand{\ssbr}[1]{\left\llbracket #1 \right\rrbracket}
\newcommand{\nnbr}[1]{{\left\vert\kern-0.25ex\left\vert\kern-0.25ex\left\vert #1 \right\vert\kern-0.25ex\right\vert\kern-0.25ex\right\vert}}

\newcommand{\ssepp}[2]{\left[#1 ~\middle\vert~ #2 \right]}
\newcommand{\csepp}[2]{\left\{#1 ~\middle\vert~ #2 \right\}}

\definecolor{commentcolor}{RGB}{110,154,155}   %

\renewcommand{\b}{\textbf}
\renewcommand{\t}[1]{\text{#1}}

\newcommand{\tsvd}[2]{\ssbr{#1}_{#2}}

\newcommand{\nrank}[2]{\mathop{\mathrm{rank}_{#1}}\rbr{#2}}

\newcommand{\bftheta}{\boldsymbol{\theta}}

\newcommand{\bff}{\boldsymbol{f}}

\newcommand{\bfF}{\boldsymbol{F}}

\newcommand{\bfJ}{\Jb}
\newcommand{\bfeta}{\etab}

\newcommand{\Cvar}{C_{v}}
\newcommand{\Cbias}{C_{b}}
\newcommand{\Cdisc}{C_{d}}

\usepackage{graphicx}

\graphicspath{{PlotSources}}

\usepackage{authblk}

\title{Randomized time stepping of nonlinearly parametrized solutions of evolution problems}

\author[1]{Yijun Dong}
\author[1]{Paul Schwerdtner}
\author[1]{Benjamin Peherstorfer\thanks{Corresponding author}}

\affil[1]{Courant Institute of Mathematical Sciences, New York University,
251 Mercer St, New York, NY 10012, US}

\begin{document}

\maketitle

\begin{abstract}
The Dirac-Frenkel variational principle is a widely used building block for using nonlinear parametrizations in the context of model reduction and numerically solving partial differential equations; however, it typically leads to time-dependent least-squares problems that are poorly conditioned. This work introduces a randomized time stepping scheme that solves at each time step a low-dimensional, random projection of the parameter vector via sketching. The sketching has a regularization effect that leads to better conditioned least-squares problems and at the same time reduces the number of unknowns that need to be solved for at each time step. Numerical experiments with benchmark examples demonstrate that randomized time stepping via sketching achieves competitive accuracy and outperforms standard regularization in terms of runtime efficiency.
\end{abstract}

\vspace{0.5cm}

\vspace{0.5cm}

\section{Introduction}\label{sec:Intro}
We are interested in numerically solving time-dependent least-squares problems that can arise when nonlinear parametrizations such as neural networks are used for discretizing solution functions of partial differential equations (PDEs). The challenge that we address in this work is that the least-squares problems are poorly conditioned in many cases in this context. In this section, we review past literature and then overview our contribution, which is developing and analyzing randomized time stepping schemes that can cope with poorly conditioned least-squares problems. We close this section with an outline of the manuscript.

\subsection{Motivation}\label{sec:Intro:Motivation}
To motivate time-dependent least-squares problems, let us start by considering a time-dependent PDE
\begin{equation}\label{eq:Intro:PDEProblem}
\partial_t u(t, \xb) = f(u, \xb),\qquad \xb \in \Omega,
\end{equation}
over time $t \in [0, T]$ and spatial domain $\Omega \subseteq \R^d$. The solution function is $u: [0, T] \times \Omega \to \R$. The right-hand side function is $f$, and it can include partial derivatives of $u$ in space. Given boundary conditions and an initial condition $u_0: \Omega \to \R$, we assume well-posedness of the corresponding PDE problem over a suitable Hilbert space $\mathcal{U}$ of functions with norm $\|\cdot\|$. 

We parametrize the solution function $u$ via a sufficiently regular parametrization $\hat{u}: \R^p \times \Omega \to \R$ over a time-dependent finite-dimensional parameter vector $\thetab(t) \in \R^p$. 
Critically, we allow the parameter vector $\thetab(t)$ to enter nonlinearly in $\hat{u}$, which is different from many standard approaches in numerical analysis that rely on $\thetab(t)$ entering linearly. For example, standard finite-element methods typically parametrize solution functions $u$ via linear combinations of basis functions centered at grid points, which means that the parameter vector $\thetab(t)$ enters linearly because it corresponds to the coefficients in the linear combination \cite{FEMTheory}. 
In contrast, nonlinear parametrizations such as neural networks, Gaussian wave packets, and tensor networks allow the parameter vector $\thetab(t)$ to enter nonlinearly \cite{Goodfellow-et-al-2016,MAL-059,BECK20001,lubich2008quantum}. 
Using nonlinear parametrizations can be advantageous when numerically solving differential equations over high-dimensional spatial domains and for reduced modeling of transport-dominated problems; see, e.g., \cite{doi:10.1073/pnas.1718942115,P22AMS}.

There is a large class of schemes that uses matrix factorizations as nonlinear parametrizations. Examples are  dynamic low-rank approximations \cite{doi:10.1137/050639703,EINKEMMER2021110353,doi:10.1137/23M1547603,Peherstorfer15aDEIM,CPagAdapt,hesthaven_pagliantini_rozza_2022,doi:10.1137/09076578X,Arnold2014} and dynamic orthogonal decompositions \cite{SAPSIS20092347,NobileDO,MUSHARBASH2018135,doi:10.1137/16M1109394,doi:10.1137/21M1431229,doi:10.1137/22M1534948}. 
We focus in this work on schemes that allow more generic nonlinear parametrizations, which have been developed under the names Neural Galerkin schemes \cite{bruna2024neural,BERMAN2024389}, evolutional neural networks \cite{Du_2021}, nonlinear reduced-order solutions \cite{doi:10.1137/21M1415972}, and others \cite{UngerTransformModes2020,finzi2023a,KAST2024112986}. We refer to \cite{zhang2024sequentialintimetrainingnonlinearparametrizations} for a more detailed discussion about such schemes and additional references. 
All of these methods have in common that they build on the Dirac-Frenkel variational principle \cite{dirac1930note,frenkel1934wave,Kramer1981,lubich2008quantum} for solving for $\hat{u}(\thetab(t), \cdot)$, which leads to the equation
\begin{equation}\label{eq:Intro:DF}
\partial_t \hat{u}(\thetab(t), \cdot) = \mathrm P_{\hat{u}(\thetab(t), \cdot)}f(\hat{u}(\thetab(t), \cdot), \cdot),
\end{equation}
where $\mathrm P_{\hat{u}(\thetab(t), \cdot)}$ is the orthogonal projection of $f(\hat{u}(\thetab(t), \cdot), \cdot)$ onto the space spanned by the component functions of the gradient $\nabla_{\thetab}\hat{u}(\thetab(t), \cdot)$.   
Following collocation approaches used in \cite{bruna2024neural,Du_2021,finzi2023a,KAST2024112986,WEN2024134129}, discretizing \eqref{eq:Intro:DF} in space and using the chain rule $\partial_t \hat{u}(\thetab(t), \cdot) = \nabla_{\thetab}\hat{u}(\thetab(t), \cdot) \dot{\thetab}(t)$ leads to time-dependent least-squares problems of the form
\begin{equation}\label{eq:Intro:LSQDynSys}
    \min_{\etab(t) \in \R^p} \|\bfJ(\thetab(t))\etab(t) - \bff(\thetab(t))\|_2^2,
\end{equation}
where $\etab(t) = \dot{\thetab}(t)$ is the time derivative of $\thetab(t)$. 
The matrix $\bfJ(\thetab(t))$ is the batch gradient of $\hat{u}$ with respect to $\thetab(t)$, i.e., it is the matrix obtained by evaluating the gradient at collocation points $\xb_1, \dots, \xb_n \in \Omega$,
\begin{equation}\label{eq:Intro:BatchJac}
\bfJ(\thetab(t)) = \begin{bmatrix} - & \nabla_{\thetab}\hat{u}(\thetab(t), \xb_1)^{\top} & -\\
 & \vdots & \\
 - & \nabla_{\thetab}\hat{u}(\thetab(t), \xb_n)^{\top} & -
\end{bmatrix} \in \R^{n \times p}.
\end{equation}
The components of the vector-valued right-hand side function 
\[
\bff(\thetab(t)) = [f(\hat{u}(\thetab(t), \xb_1), \cdot), \dots, f(\hat{u}(\thetab(t), \cdot), \xb_n)]^{\top} \in \R^n
\]
are given by evaluating the function $f$ at the collocation points $\xb_1, \dots, \xb_n \in \Omega$.  

If a discretization via collocation is not suitable, one can also start with the system of ordinary differential equations arising from a numerically proper spatial discretization of \eqref{eq:Intro:PDEProblem},
\begin{equation}\label{eq:Intro:HighDimDynSys}
\frac{\mathrm d}{\mathrm dt}\bfu(t) = \bff(\bfu(t)),
\end{equation}
with the vector-valued state function $\bfu: [0, T] \to \R^n$ and the vector-valued right-hand side function $\bff: \R^n \to \R^n$.  
To avoid having to solve for the high-dimensional state function $\bfu$ over time, one can apply an approach motivated by model reduction that parametrizes $\bfu$ as $\hat{\bfu}(\thetab(t))$ with a low-dimensional $\thetab(t) \in \R^p$ with $p < n$; see, e.g., \cite{LEE2020108973,UngerTransformModes2020,berman2024colora,weder2024nonlinearmodelreductionneural}. Plugging the parametrization $\hat{\bfu}$ into \eqref{eq:Intro:HighDimDynSys} and solving for $\thetab(t)$ leads to the same least-squares problem as given in \eqref{eq:Intro:LSQDynSys}, except that the matrix $\Jb(\thetab)$ is different. In any case, the computational task becomes solving time-dependent least-squares problem that have the form \eqref{eq:Intro:LSQDynSys}.

\subsection{Poor conditioning and literature review}
A well know challenge of solving time-dependent least-squares problems \eqref{eq:Intro:LSQDynSys} that arise in the context of the Dirac-Frenkel variational principle and nonlinear parametrizations is that the component functions of the gradient can be numerically linearly dependent so that the corresponding matrix $\bfJ(\thetab(t))$ is poorly conditioned; see \cite{KAY1989165,10.1063/1.449204,PhysRevE.101.023313,feischl2024regularized,zhang2024sequentialintimetrainingnonlinearparametrizations}. 
Several approaches have been introduced in the literature to address this challenge. First, other principles than the Dirac-Frenkel variational principle can be used to find $\thetab(t)$. For example, following a discretize-then-optimize paradigm \cite{kvaal2023need,zhang2024sequentialintimetrainingnonlinearparametrizations} avoids the poorly conditioned least-squares problems altogether; however, it can be computationally expensive to solve the typically nonlinear optimization problems corresponding to discretize-then-optimize schemes \cite{zhang2024sequentialintimetrainingnonlinearparametrizations}.

Second, there is a wide range of works on time integration for specifically dynamic low-rank approximations, which means that the nonlinear parametrizations are given by matrix factorizations. Pioneering work in this area include \cite{lubich2014projector} and \cite{Ceruti2022}, which propose projector-splitting schemes. In particular, we note that the specific matrix factorizations used in dynamic low-rank approximation allow handling issues such as poor conditioning with parametrization-specific formulations and techniques. However, we are interested in schemes that can handle more generic nonlinear parametrizations than matrix factorizations. 

Third, there is an increasing body of literature on regularizing the time-dependent least-squares problem. To the best knowledge of the authors, the work \cite{feischl2024regularized} is the first one that systematically studies the impact of a Tikhonov regularizer on time-dependent least-squares problems in this context. The authors show that Tikhonov regularization helps to avoid poor conditioning and eases time-step size restrictions. In \cite{lubich2025regularizeddynamicalparametricapproximation}, a similar approach based on Tikhonov regularization is presented for implicit time integration schemes.

\subsection{Our approach: Randomized time stepping}
In this work, we follow a different route compared to the literature discussed in the previous section. We introduce and analyze a randomized time stepping scheme that can lead to better conditioned least-squares problems. 
To the best knowledge of the authors, the first randomized time integration scheme in the context of time-dependent least-squares problems was introduced in \cite{berman2024randomized}; however, only empirical results are presented and no analysis is provided. 
In the work \cite{lam2024randomizedlowrankrungekuttamethods}, randomized Runge-Kutta methods are introduced specifically in the context of dynamic low-rank approximations. The methods introduced in \cite{lam2024randomizedlowrankrungekuttamethods} replace the projection step of the Dirac-Frenkel variational principle with a randomized one, which can be interpreted as sketching $\bfJ(\thetab(t))$ from the left. This is starkly different from our approach that applies randomization to sketch the matrix $\bfJ(\thetab(t))$ from the right so that the number of unknowns at each time step is reduced. In other words, the least-squares problems are solved over a random low-dimensional projection of $\dot{\thetab}(t)$ with $m \ll p$ components. As we will show, this leads to speedups compared to updating all $p$ components at every time step. The work  \cite{lindsey2025mne} contributes randomized time stepping and follows a similar approach as \cite{lam2024randomizedlowrankrungekuttamethods}, where the batch gradient is sketched from the left. 
There is a range of other randomized time integration schemes such as \cite{Lie2022} for uncertainty quantification and \cite{Abdulle2020} for chaotic systems; however, in these schemes, randomization serves a different purpose than improving the conditioning of the system. 
Sketching from the right has been used in randomized Newton and trust-region methods to solve intermediate least-squares problems efficiently \cite{gower2019rsn,doi:10.1137/22M1530264,doi:10.1137/22M1524072,romanov2025newton}; however, our goal is to closely follow a trajectory given by a time-dependent PDE and so we cannot leverage that there is convergence to an optimum.

\subsection{Overview and outline}
In \Cref{sec:problem_setup}, we set the stage by stating the fully discrete system and the problem formulation in more detail. 
In \Cref{sec:ng_randomized_time_integration}, we introduce a generic randomized time stepping scheme that builds on randomized increment functions that have bounded variance. In particular, we show that under those conditions the mean-squared error of the estimator at each time step can be controlled via the number of samples that are taken at each time step and the variance of the increment function. If the increment function is unbiased, then in the limit of infinite samples, the mean-squared error vanishes. 
In \Cref{sec:RTSInc}, we derive a randomized time stepping scheme specifically for systems for which the increment function is given by the solution of a least-squares problem. We propose to randomly sketch (randomly project) the unknown $\dot{\thetab}(t)$. We show that the randomized increment function based on sketching (a)  satisfies the conditions required by our randomized time stepping scheme and (b) allows controlling the condition number of the sketched least-squares problem.
Additionally, we stress that the sketching reduces the number of unknowns in the least-squares problem, which means that lower costs per time step are incurred than when using regularization and updating all components of the parameter vector at each time step.
In \Cref{sec:NumExp}, numerical experiments with Neural Galerkin schemes \cite{bruna2024neural,BERMAN2024389} demonstrate that the randomized time stepping helps obtaining accurate approximations, while no regularization fails. 
Additionally, we show that taking larger networks with the same costs in sketching is useful and leads to speedups compared to other Tikhonov regularization that update all components.
Conclusions are drawn in \Cref{sec:Conc}.

\section{Problem setup}\label{sec:problem_setup}
We study the evolution of dynamical systems whose increments are obtained from poorly conditioned least-squares systems.  

\subsection{Dynamical systems}
Building on the setup described in the introduction, let us consider dynamical systems of the form
\begin{equation}\label{eq:dirac_frenkel_least_squares_dynamics}
\dot{\bftheta}(t) = \bfF(\bftheta(t))\,,
\end{equation}
over time $t \in [0, T]$ with increment function $\bfF: \mathbb{R}^p \to \mathbb{R}^p$ defined as
\begin{align}\label{eq:increment_function}
    \Fb(\thetab) = \argmin_{\etab \in \R^p} \nbr{\Jb(\thetab) \etab - \fb(\thetab)}_2^2\,.
\end{align}
The initial condition is $\bftheta(0) = \bftheta_0 \in \mathbb{R}^p$. Recall that in the case of using nonlinear parametrizations of PDE solution functions with collocation, the matrix $\bfJ(\bftheta)$ is the batch gradient defined in \eqref{eq:Intro:BatchJac} and $\fb(\thetab) \in \mathbb{R}^n$ is the right-hand side $f$ of \eqref{eq:Intro:PDEProblem} evaluated at collocation points $\xb_1, \dots, \xb_n$. 
To ensure that the least-squares problem \eqref{eq:increment_function} identifies a unique parameter function $\thetab(t)$, we assume throughout this work that $\rank(\Jb(\thetab)) = p$ for all $\thetab \in \R^p$. 
Furthermore, we assume that $\Fb$ is Lipschitz and grows at most linearly, which we state in the following assumption: 
\begin{assumption}[Lipschitzness and growth of $\Fb$]\label{asm:lipschitzness_Ainvf}
    The function $\Fb: \R^p \to \R^p$ in \eqref{eq:increment_function} is
    \begin{enumerate}[label=(\roman*)]
        \item $L_F$-Lipschitz, \ie there exists $L_F > 0$ such that
    \begin{align}\label{eq:asm:LipschitzOfF}
        \nbr{\Fb(\thetab) - \Fb(\thetab')}_2 \le L_F \nbr{\thetab - \thetab'}_2, \quad \text{ for all } \thetab, \thetab' \in \R^p,
    \end{align}
    \item $B_F$-linear growth, \ie $\nbr{\Fb(\thetab)}_2 \le B_F \nbr{\thetab}_2$ for all $\thetab \in \R^p$.
    \end{enumerate}
\end{assumption}

The linear growth assumption of $\Fb$ is a strong but convenient choice here. The main result given in \Cref{thm:randomized_constraints_convergence_biased} can be extended to affine and polynomial growth, which just leads to different constants in the bounds and which we leave for future work.  

\subsection{Discrete dynamical systems}
Applying explicit Euler time integration with time-step size $\delta t > 0$ to \eqref{eq:dirac_frenkel_least_squares_dynamics} yields a time-discrete dynamical system that evolves over the time interval $[0, T]$ with time-step size $\delta t = T/K$ for a fixed number of equidistant time steps $K \in \N$ over $0 = t_0 < t_1 < \dots < t_K = T$:
\begin{align}\label{eq:discrete_dynamics_setup}
    &\thetab^*_{k+1} = \thetab^*_{k} + \delta t \, \etab_k^*\,,\qquad \etab_k^* = \Fb(\thetab^*_{k}), \qquad k = 0, 1, \dots, K-1,
\end{align}
where the initial condition is $\thetab^*_0 = \thetab_0$ and the increment function $\Fb$ is given in \eqref{eq:increment_function}. Comparing the time-continuous and time-discrete trajectories, classical discretization error analysis for first-order explicit methods (see \eg, \cite[\S 11.3.2]{Quarterioni}) shows under \Cref{asm:lipschitzness_Ainvf} and $\thetab(t)$ twice continuously differentiable with $\Cdisc = \max_{t \in [0,T]} \frac{1}{2} \|\ddot{\thetab}(t)\|_2$ bounded that the error satisfies
\begin{align}\label{eq:discrete_error_bound}
    \nbr{\thetab^*_{k} - \thetab(k\delta t)}_2 \le \Cdisc  \delta t (\exp(L_F k  \delta t)-1)/L_F\,,\qquad k=1,\dots,K\,, 
\end{align}
where  
$L_F$ is the Lipschitz constant of $\Fb$ in \Cref{asm:lipschitzness_Ainvf}. The error \eqref{eq:discrete_error_bound} is typically referred to as the global error \cite{Quarterioni}.

\subsection{Increment functions based on poorly conditioned least-squares problems}\label{sec:ill_conditioned_Jacobian}
A numerical challenge of solving systems \eqref{eq:discrete_dynamics_setup} in the context of time-dependent nonlinear parametrizations and the Dirac-Frenkel variational principle is that the least-squares problem underlying the increment function $\Fb$ can be poorly conditioned,  
\begin{align*}
    \kappa(\Jb(\thetab)) = \sigma_1(\Jb(\thetab))/\sigma_p(\Jb(\thetab)) \gg 1\,,
\end{align*} 
where $\sigma_1(\Jb(\thetab)) \geq \cdots \geq \sigma_p(\Jb(\thetab))$ denote the singular values of $\Jb(\thetab)$; see, e.g., \cite{feischl2024regularized,zhang2024sequentialintimetrainingnonlinearparametrizations}; and also the earlier work \cite{KAY1989165,10.1063/1.449204,PhysRevE.101.023313}. 
In the following, we consider the situation that $\Jb(\thetab)$ is full rank but numerically low-rank with respect to a given precision $0 < \tau \ll 1$, 
\begin{align}\label{eq:numerical_rank}
    \nrank{\tau}{\Jb(\thetab)} = \min\csepp{m \in \sbr{p}}{\nbr{\Jb(\thetab) - \tsvd{\Jb(\thetab)}{m}}_2 < \tau \nbr{\Jb(\thetab)}_2}\,,
\end{align}
where $[p]$ denotes the set $[p] = \{1, \dots, p\}$ and $\tsvd{\Jb(\thetab)}{m}$ the best rank-$m$ approximation of $\Jb(\thetab)$ in the spectral operator 2-norm. 
If $\nrank{\tau}{\Jb(\thetab)} < p$, we refer to $\Jb(\thetab)$ as numerically low-rank. In other words, a matrix $\Jb(\thetab)$ with a low numerical rank is poorly conditioned in the sense that  $\kappa(\Jb(\thetab)) > \tau^{-1} \gg 1$.
This means that when time stepping  \eqref{eq:discrete_dynamics_setup}, even small (numerical) perturbations to $\Jb(\thetab)$ and $\fb(\thetab)$ can be amplified by large factors in the relative errors of the increments given by the least-squares solutions of \eqref{eq:increment_function}, which can accumulate over time.

\subsection{Regularizing dynamics}\label{sec:Problem:RegularizingDynamics}
Regularization can help to circumvent the poor conditioning of the least-squares problems in \eqref{eq:discrete_dynamics_setup}. Time stepping with Tikhonov regularization leads to the system 
\begin{align}\label{eq:discrete_time_det_tikhonov}
    \wh\thetab_{k+1} = \wh\thetab_k + \delta t\ \wh\etab_k,\quad 
    \wh\etab_k = \argmin_{\etab \in \R^p} \nbr{\Jb(\wh\thetab_k)\etab - \fb(\wh\thetab_k)}_2^2 + \alpha \nbr{\etab}_2^2\,,
\end{align}
for time steps $k = 0, 1, \dots, K-1$, $\wh{\thetab}_0 = \thetab_0^*$ and regularization parameter $\alpha \geq 0$. 
The increment $\wh\etab_k$ in \eqref{eq:discrete_time_det_tikhonov} has the analytical representation
\begin{align}\label{eq:discrete_time_det_tikhonov_solution}
    \wh\etab_k = \rbr{\Jb(\wh\thetab_k)^\top \Jb(\wh\thetab_k) + \alpha \Ib_p}^{-1} \Jb(\wh\thetab_k)^\top \fb(\wh\thetab_k).
\end{align}
The key numerical step to obtain $\wh\etab_k$ is the linear solve with the matrix  $\Jb(\wh\thetab_k)^\top \Jb(\wh\thetab_k) + \alpha \Ib_p$ whose condition number is controlled by the regularization parameter $\alpha$. See, e.g., \cite{feischl2024regularized,lubich2025regularizeddynamicalparametricapproximation}, where Tikhonov regularization is introduced and rigorously analyzed in the context of the Dirac-Frenkel variational principle.

Time stepping with truncated SVD regularization takes the form,
\begin{align}\label{eq:discrete_time_det_tsvd}
    \wh\thetab_{k+1} = \wh\thetab_k + \delta t\ \wh\etab_k,\quad 
    \wh\etab_k = \tsvd{\Jb(\wh\thetab_k)}{r}^{\dagger} \fb(\wh\thetab_k),\quad 
    k = 0, 1, \dots, K-1,
\end{align}
with initial condition  $\wh\thetab_0 = \thetab_0^*$. The rank at which the batch gradient $\Jb(\wh\thetab_k)$ is truncated is denoted as $r$. See, e.g., \cite{bruna2024neural} where regularization with truncated SVD is used in the context of the Dirac-Frenkel variational principle.

We stress that regularized systems such as the one obtained with Tikhonov regularization \eqref{eq:discrete_time_det_tikhonov} and truncated SVD \eqref{eq:discrete_time_det_tsvd} solve different dynamics. In particular, there is the trajectory $\wh\thetab_1, \dots, \wh\thetab_K$ corresponding to the regularized discretized system, the time-discrete trajectory $\thetab_1^*, \dots, \thetab_K^*$ of the unregularized but discrete system \eqref{eq:discrete_dynamics_setup}, and the time-continuous trajectory $\thetab(t_1), \dots, \bftheta(t_K)$ of the time-continuous system \eqref{eq:dirac_frenkel_least_squares_dynamics} at the discrete time points $t_1, \dots, t_K$. Using the triangle inequality and the standard error bound \eqref{eq:discrete_error_bound}, we obtain a bound for the global error of the regularized and discretized dynamics compared to the time-continuous trajectory:
\begin{equation}\label{eq:stable_discrete_error_bound}
\|\wh\thetab_k - \thetab(k \delta t)\|_2 \leq \|\wh\thetab_k - \thetab_k^*\|_2 + \Cdisc  \delta t (\exp(L_F k  \delta t)-1)/L_F\,,\quad k = 1, \dots, K\,.
\end{equation}
We refer to the first term on the right-hand side of \eqref{eq:stable_discrete_error_bound} as the regularization error because it is the error introduced by solving a regularized time-discrete system rather than directly the time-discrete system. Recall that solving directly the time-discrete system is challenging in our case because of the poor conditioning of the increment function. The second term in the bound \eqref{eq:stable_discrete_error_bound} is the error due to the time discretization given in \eqref{eq:discrete_error_bound}.

\section{Randomized time stepping}\label{sec:ng_randomized_time_integration}
We now introduce a randomized time stepping scheme that provides estimates of the parameter vectors over time. 
We stress that the main result in this section (\Cref{thm:randomized_constraints_convergence_biased}) applies to general increment functions $\Fb$ that are not necessarily given via least-squares regression problems as the increment function given in \eqref{eq:increment_function}, which motivated this work.

Consider a discrete dynamical system \eqref{eq:discrete_dynamics_setup} with increment function $\Fb$. 
Building on the increment function $\Fb$, we propose randomized time stepping that is initialized with $\wt\thetab_0 = \thetab^*_0$ and evolves a trajectory as
\begin{align}\label{eq:discrete_time_random}
    \wt\thetab_{k+1} = \wt\thetab_k + \delta t\ \wt\etab_k,\quad \wt\etab_k = \frac{1}{q} \sum_{i=1}^{q} \wt{\Fb}(\wt\thetab_k; \Gammab_{k,i}), \quad k = 0, 1, \dots, K-1,
\end{align}
where $\wt{\Fb}: \R^p \times \Qcal \to \R^p$ is a randomized increment function. The randomized increment function takes two inputs. First, the current estimator $\wt\thetab_k$ and, second, an independent random variable $\Gammab_{k,i} \in \Qcal$ that induces the randomness into $\wt{\Fb}$ and that is supported on $\Qcal$. 
The local sample size $q \in \N$ is the number of independent evaluations of $\wt{\Fb}$ over which each increment $\wt\etab_k$ is averaged at each time step $k$. 
We denote the set of random parameters associated with the $k$-th randomized increment $\wt\etab_k$ as $\Gcal_k = \csepp{\Gammab_{k,i}}{i \in [q]}$. 
All random parameters in $\Gcal = \bigcup_{k=0}^{K-1} \Gcal_k$ are drawn $\iid$ from some distribution $P_{\Qcal}$ supported on $\Qcal$. We make the following assumption on the randomized increment function:

\begin{assumption}[$\wt{\Fb}$ with bounded variance and bias]\label{asm:randomized_constraints_biased}
    Given a randomized increment function $\wt{\Fb}: \R^p \times \Qcal \to \R^p$, let 
    \begin{equation}\label{eq:Rand:MuDef}
    \mub(\thetab) = \E_{\Gammab \sim P_{\Qcal}}[\wt{\Fb}(\thetab; \Gammab)]\,,\qquad \thetab \in \R^p\,,
    \end{equation}
   be the mean increment function. Then there exist constants $\Cvar, \Cbias > 0$ such that 
    \begin{align}
        (i)\quad & \nbr{\mub(\thetab) - \Fb(\thetab)}_2^2 \le \Cbias \nbr{\Fb(\thetab)}_2^2,\text{ and} & \qquad\qquad \label{asm:i}\\
        (ii)\quad & \E_{\Gammab \sim P_{\Qcal}}\sbr{\nbr{\wt{\Fb}(\thetab; \Gammab) - \mub(\thetab)}_2^2} \le \Cvar \nbr{\Fb(\thetab)}_2^2, & \qquad\qquad\label{asm:ii}
    \end{align}
    for all $\thetab \in \R^p$.
\end{assumption}

Given a randomized increment function $\wt{\Fb}$ with bounded bias and variance as in \Cref{asm:randomized_constraints_biased}, we obtain the following error bounds for the solutions obtained with the randomized time stepping:
\begin{theorem}\label{thm:randomized_constraints_convergence_biased}
    Let the trajectory $\wt\thetab_1, \dots, \wt\thetab_K$ be given by the randomized time stepping  \eqref{eq:discrete_time_random} and let it be bounded as $\max_{0 \le k \le K}\|\wt\thetab_k\|_2 \leq B_\theta$ almost surely. 
    Under Assumptions~\ref{asm:lipschitzness_Ainvf} and \ref{asm:randomized_constraints_biased}, when taking $0 < \delta t \le 1/(4L_F)$, the error is bounded as 
    \begin{align}\label{eq:rand_convergence_step_biased}
    \begin{split}
        &\E_{\bigcup_{j=0}^{K-1} \Gcal_j}\sbr{\nbr{\wt\thetab_{K} - \thetab^*_{K}}_2} \le \rbr{\E_{\bigcup_{j=0}^{K-1} \Gcal_j}\sbr{\nbr{\wt\thetab_{K} - \thetab^*_{K}}_2^2}}^{1/2} \\
        &\le \sqrt{2}B_F B_\theta \Bigg(\underbrace{\sqrt{\frac{\Cvar}{q} T \rbr{1 + 3 L_F T  \mathrm e^{4 L_F T}} \delta t}}_{\t{vanishing variance}} + \underbrace{\sqrt{\Cbias T \mathrm e^{4 L_F T} \rbr{\frac{2}{L_F} + 3 \delta t}}}_{\t{cumulative + vanishing bias}} \Bigg) \,.
    \end{split}
    \end{align}
    
\end{theorem}
The first term in the brackets in  \eqref{eq:rand_convergence_step_biased} comes from the variance of the randomized increment function $\wt{\Fb}$ in \Cref{asm:randomized_constraints_biased}(ii), while the second term is the result of a combination of the bias of $\wt{\Fb}$ compared to the true increment function $\Fb$ in \Cref{asm:randomized_constraints_biased}(i) and the variance. The variance terms  vanish as $\delta t \to 0$, the second term, representing the accumulated error from the biased increment function $\wt\Fb$ over the entire trajectory, can remain positive regardless of the time-step size $\delta t$. This is analogous to the non-vanishing error introduced by the deterministic regularization approaches discussed in \Cref{sec:Problem:RegularizingDynamics}. 

\begin{proof}[Proof of \Cref{thm:randomized_constraints_convergence_biased}]
    First, recall from \eqref{eq:discrete_time_random} that 
    each randomized increment $\wt\etab_k$ satisfies
    \begin{align}\label{eq:pf_mean_random_increment}
        \E_{\Gcal_k}[\wt\etab_k \mid \wt\thetab_k] = \E_{\Gcal_k}\ssepp{\frac{1}{q} \sum_{i=1}^{q} \wt{\Fb}(\wt\thetab_k; \Gammab_{k,i})}{\wt\thetab_k}
        = \frac{1}{q} \sum_{i=1}^{q} \E_{\Gammab_{k,i}\sim P_{\Qcal}}[\wt{\Fb}(\wt\thetab_k; \Gammab_{k,i}) \mid \wt\thetab_k]
        = \mub(\wt\thetab_k)
    \end{align}
    where $\mub(\thetab)$ defined in \eqref{eq:Rand:MuDef} denotes the mean of the increment function for a fixed $\thetab$ and 
    $\Gcal_k$ denotes the set of $q$ \iid random variables $\Gammab_{k,i} \sim P_{\Qcal}$ at the $k$-th step. 
    We now make the following variance-bias decomposition through the mean increment function $\mub(\wt\thetab_k)$:
    \begin{align}\label{eq:variance_bias_decomposition}
    \begin{split}
        &\E_{\bigcup_{j=0}^{k-1} \Gcal_j}\sbr{\nbr{\sum_{j=0}^{k-1}\wt\etab_j - \etab^*_j}_2^2} 
        = \E_{\bigcup_{j=0}^{k-1} \Gcal_j}\sbr{\nbr{\sum_{j=0}^{k-1}\wt\etab_j - \mub(\wt\thetab_j) + \mub(\wt\thetab_j) - \etab^*_j}_2^2} \\
        \leq &2\underbrace{\E_{\bigcup_{j=0}^{k-1} \Gcal_j}\sbr{\nbr{\sum_{j=0}^{k-1}\wt\etab_j - \mub(\wt\thetab_j)}_2^2}}_{\t{variance}} + 2\underbrace{\E_{\bigcup_{j=0}^{k-2} \Gcal_j}\sbr{\nbr{\sum_{j=0}^{k-1}\mub(\wt\thetab_j) - \etab^*_j}_2^2}}_{\t{bias}}.
    \end{split}
    \end{align} 
    Notice that the bias term for index $k - 1$ in \eqref{eq:variance_bias_decomposition} is independent of $\Gcal_{k - 1}$ and thus it was dropped in the subscript of the expectation.

    \emph{Bounding the variance}: 
    By \Cref{asm:randomized_constraints_biased}, for each $k = 0, 1, \dots, K-1$, we have
    \begin{align*}
        \E_{\Gcal_k}\sbr{\nbr{\wt\etab_k - \mub(\wt\thetab_k)}_2^2}
        = &\E_{\Gcal_k}\sbr{\nbr{\frac{1}{q} \sum_{i=1}^{q} \wt{\Fb}(\wt\thetab_k; \Gammab_{k,i}) - \mub(\wt\thetab_k)}_2^2} \\
        = &\frac{1}{q} \E_{\Gammab_{k,i} \sim P_{\Qcal}}\sbr{\nbr{\wt{\Fb}(\wt\thetab_k; \Gammab_{k,i}) - \mub(\wt\thetab_k)}_2^2} \\
        \le &\frac{\Cvar}{q} \nbr{\Fb(\wt\thetab_k)}_2^2\,,
    \end{align*}
    where the second equality holds because $\wt{\Fb}(\wt\thetab_k; \Gammab_{k,i}) - \mub(\wt\thetab_k)$ for $i = 1, \dots q$ are independent and have zero mean. 
    Therefore, the linear growth of $\Fb$ (see \Cref{asm:lipschitzness_Ainvf}) and $\|\wt\thetab_k\|_2$ together implies that 
    \begin{align}\label{eq:pf_variance_one_step}
        \E_{\Gcal_k}\sbr{\nbr{\wt\etab_k - \mub(\wt\thetab_k)}_2^2} \le \frac{\Cvar}{q} \nbr{\Fb(\wt\thetab_k)}_2^2 \le \frac{\Cvar}{q} B_F^2 \nbr{\wt\thetab_k}_2^2 \le \frac{\Cvar}{q} B_F^2 B_\theta^2.
    \end{align}
    Then, we can bound the variance by induction. For the base case, \eqref{eq:pf_variance_one_step} implies that
    \begin{align*}
        \E_{\Gcal_0}\sbr{\nbr{\wt\etab_0 - \mub(\wt\thetab_0)}_2^2} \le \frac{\Cvar}{q} B_F^2 B_\theta^2.
    \end{align*}
    Now suppose $\E_{\bigcup_{j=0}^{k-2} \Gcal_j}\sbr{\nbr{\sum_{j=0}^{k-2}\wt\etab_j - \mub(\wt\thetab_j)}_2^2} \le \frac{\Cvar}{q} B_F^2 B_\theta^2 (k-1)$. For any $k=2,\cdots,K$,
    \begin{align*}
        &\E_{\bigcup_{j=0}^{k-1} \Gcal_j}\sbr{\nbr{\sum_{j=0}^{k-1}\wt\etab_j - \mub(\wt\thetab_j)}_2^2} \\
        = &\E_{\bigcup_{j=0}^{k-1} \Gcal_j}\sbr{\nbr{\rbr{\wt\etab_{k-1} - \mub(\wt\thetab_{k-1})} + \rbr{\sum_{j=0}^{k-2}\wt\etab_j - \mub(\wt\thetab_j)}}_2^2} \\
        = &\E_{\bigcup_{j=0}^{k-1} \Gcal_j}\sbr{\nbr{\wt\etab_{k-1} - \mub(\wt\thetab_{k-1})}_2^2} + \E_{\bigcup_{j=0}^{k-2} \Gcal_j}\sbr{\nbr{\sum_{j=0}^{k-2}\wt\etab_j - \mub(\wt\thetab_j)}_2^2},
    \end{align*}
    where the cross term vanishes because according to \eqref{eq:pf_mean_random_increment},
    \begin{align*}
        \E_{\bigcup_{j=0}^{k-1} \Gcal_j}\sbr{\wt\etab_{k-1} - \mub(\wt\thetab_{k-1})} = \E_{\bigcup_{j=0}^{k-2} \Gcal_j}\sbr{\E_{\Gcal_{k-1}}\ssepp{\wt\etab_{k-1}}{\wt\thetab_{k-1}} - \mub(\wt\thetab_{k-1})} = 0.
    \end{align*}
    Combining the induction hypothesis and \eqref{eq:pf_variance_one_step} then leads to an upper bound for the variance term in \eqref{eq:variance_bias_decomposition}:
    \begin{align}\label{eq:pf_variance_bound}
        \E_{\bigcup_{j=0}^{k-1} \Gcal_j}\sbr{\nbr{\sum_{j=0}^{k-1}\wt\etab_j - \mub(\wt\thetab_j)}_2^2} \le \frac{\Cvar}{q} B_F^2 B_\theta^2 + \frac{\Cvar}{q} B_F^2 B_\theta^2 (k-1) = \frac{\Cvar}{q} B_F^2 B_\theta^2 k.
    \end{align}

    \emph{Bounding the bias}:
    For all $k = 1,\cdots,K$, let $\wt\zetab_k = \thetab^*_0 + \delta t \sum_{j=0}^{k-1} \mub(\wt\thetab_j)$.
    We denote
    \begin{align}\label{eq:pf_bias_decomposition_2}
    \begin{split}
        \wb_{k-1} = &\mub(\wt\thetab_{k-1}) - \etab^*_{k-1} \\
        = &\rbr{\mub(\wt\thetab_{k-1}) - \Fb(\wt\thetab_{k-1})} + \rbr{\Fb(\wt\thetab_{k-1}) - \Fb(\wt\zetab_{k-1})} + \rbr{\Fb(\wt\zetab_{k-1}) - \Fb(\thetab^*_{k-1})}
    \end{split}
    \end{align}
    and decompose the bias as follows:
    \begin{align}\label{eq:pf_bias_decomposition}
    \begin{split}
        &\E_{\bigcup_{j=0}^{k-2} \Gcal_j}\sbr{\nbr{\sum_{j=0}^{k-1}\mub(\wt\thetab_j) - \etab^*_j}_2^2} \\
        = &\E_{\bigcup_{j=0}^{k-2} \Gcal_j}\sbr{\nbr{\wb_{k-1} + \rbr{\sum_{j=0}^{k-2}\mub(\wt\thetab_j) - \etab^*_j}}_2^2} \\
        \le &\E_{\bigcup_{j=0}^{k-2} \Gcal_j}\sbr{\nbr{\wb_{k-1}}_2^2} + 2 \E_{\bigcup_{j=0}^{k-2} \Gcal_j}\sbr{\nbr{\wb_{k-1}}_2} \E_{\bigcup_{j=0}^{k-3} \Gcal_j}\sbr{\nbr{\sum_{j=0}^{k-2}\mub(\wt\thetab_j) - \etab^*_j}_2} \\
        &+ \E_{\bigcup_{j=0}^{k-3} \Gcal_j}\sbr{\nbr{\sum_{j=0}^{k-2}\mub(\wt\thetab_j) - \etab^*_j}_2^2} \\
        \le &E_{\bigcup_{j=0}^{k-2} \Gcal_j}\sbr{\nbr{\wb_{k-1}}_2^2} + 2 \E_{\bigcup_{j=0}^{k-2} \Gcal_j}\sbr{\nbr{\wb_{k-1}}_2} \sqrt{\E_{\bigcup_{j=0}^{k-3} \Gcal_j}\sbr{\nbr{\sum_{j=0}^{k-2}\mub(\wt\thetab_j) - \etab^*_j}_2^2}} \\
        &+ \E_{\bigcup_{j=0}^{k-3} \Gcal_j}\sbr{\nbr{\sum_{j=0}^{k-2}\mub(\wt\thetab_j) - \etab^*_j}_2^2},
    \end{split}
    \end{align}
    where the second last inequality follows from the Cauchy-Schwarz inequality, and the last inequality follows from Jensen's inequality.
    Now, we first leverage the decomposition of $\wb_{k-1}$ in \eqref{eq:pf_bias_decomposition_2} to bound 
    \begin{align*}
        &\E_{\bigcup_{j=0}^{k-2} \Gcal_j}\sbr{\nbr{\wb_{k-1}}_2^2} \\
        \le &3 \E_{\bigcup_{j=0}^{k-2} \Gcal_j}\sbr{\nbr{\mub(\wt\thetab_{k-1}) - \Fb(\wt\thetab_{k-1})}_2^2 + \nbr{\Fb(\wt\thetab_{k-1}) - \Fb(\wt\zetab_{k-1})}_2^2 + \nbr{\Fb(\wt\zetab_{k-1}) - \Fb(\thetab^*_{k-1})}_2^2}.
    \end{align*}
    For the first term, the assumption on the bias in \Cref{asm:randomized_constraints_biased}, along with the boundedness of $\Fb$ (\Cref{asm:lipschitzness_Ainvf}(ii)) and the boundedness of $\wt\thetab_{k-1}$ as stated in the theorem, implies that
        \begin{align}\label{eq:pf_bias_bound_1}
            \E_{\bigcup_{j=0}^{k-2} \Gcal_j}\sbr{\nbr{\mub(\wt\thetab_{k-1}) - \Fb(\wt\thetab_{k-1})}_2^2} \le \Cbias \E_{\bigcup_{j=0}^{k-2} \Gcal_j}\sbr{\nbr{\Fb(\wt\thetab_{k-1})}_2^2} \le \Cbias B_F^2 B_\theta^2.
        \end{align}
       For the second term, we can apply the Lipschitzness of $\Fb$ in \Cref{asm:lipschitzness_Ainvf} to obtain
        \begin{align}\label{eq:pf_bias_bound_2}
        \begin{split}
            \E_{\bigcup_{j=0}^{k-2} \Gcal_j}\sbr{\nbr{\Fb(\wt\thetab_{k-1}) - \Fb(\wt\zetab_{k-1})}_2^2} 
            \le &\E_{\bigcup_{j=0}^{k-2} \Gcal_j}\sbr{L_F^2 \nbr{\wt\thetab_{k-1} - \wt\zetab_{k-1}}_2^2} \\
            = &\E_{\bigcup_{j=0}^{k-2} \Gcal_j}\sbr{L_F^2 \delta t^2 \nbr{\sum_{j=0}^{k-2} \wt\etab_j - \mub(\wt\thetab_j)}_2^2} \\
            \le &L_F^2 \delta t^2 \frac{\Cvar}{q} B_F^2 B_\theta^2 (k-1),
        \end{split}
        \end{align}
        where the last inequality follows from \eqref{eq:pf_variance_bound}.
        For the third term, we again apply the Lipschitzness of $\Fb$ in \Cref{asm:lipschitzness_Ainvf}:
        \begin{align}\label{eq:pf_bias_bound_3}
        \begin{split}
            \E_{\bigcup_{j=0}^{k-2} \Gcal_j}\sbr{\nbr{\Fb(\wt\zetab_{k-1}) - \Fb(\thetab^*_{k-1})}_2^2}
            \le &\E_{\bigcup_{j=0}^{k-2} \Gcal_j}\sbr{L_F^2 \nbr{\wt\zetab_{k-1} - \thetab^*_{k-1}}_2^2} \\
            = &L_F^2 \delta t^2 \E_{\bigcup_{j=0}^{k-3} \Gcal_j}\sbr{\nbr{\sum_{j=0}^{k-2} \mub(\wt\thetab_j) - \etab^*_j}_2^2}\,,
        \end{split}
        \end{align}
where the last equality shows that the expectation actually depends only $\bigcup_{j=0}^{k-3} \Gcal_j$ up to index $k - 3$ rather than $k - 2$. 
   Overall, we have 
    \begin{align}\label{eq:pf_bias_bound_4}
    \begin{split}
        \E_{\bigcup_{j=0}^{k-2} \Gcal_j}\sbr{\nbr{\wb_{k-1}}_2^2} 
        \le &3 B_F^2 B_\theta^2 \rbr{\Cbias + L_F^2 \delta t^2 \frac{\Cvar}{q} (k-1)} \\
        &+ 3 L_F^2 \delta t^2 \E_{\bigcup_{j=0}^{k-3} \Gcal_j}\sbr{\nbr{\sum_{j=0}^{k-2} \mub(\wt\thetab_j) - \etab^*_j}_2^2}\,.
    \end{split}
    \end{align}

    By Jensen's inequality, we can also bound $\E_{\bigcup_{j=0}^{k-2} \Gcal_j}\sbr{\nbr{\wb_{k-1}}_2}$ based on the above:
    \begin{align}\label{eq:pf_bias_bound_5}
    \begin{split}
        &\E_{\bigcup_{j=0}^{k-2} \Gcal_j}\sbr{\nbr{\wb_{k-1}}_2} \\
        \le &\E_{\bigcup_{j=0}^{k-2} \Gcal_j}\sbr{\nbr{\mub(\wt\thetab_{k-1}) - \Fb(\wt\thetab_{k-1})}_2 + \nbr{\Fb(\wt\thetab_{k-1}) - \Fb(\wt\zetab_{k-1})}_2 + \nbr{\Fb(\wt\zetab_{k-1}) - \Fb(\thetab^*_{k-1})}_2} \\
        \le &\sqrt{\Cbias} B_F B_\theta + L_F \delta t B_F B_\theta \sqrt{\frac{\Cvar}{q}(k-1)} + L_F \delta t \sqrt{\E_{\bigcup_{j=0}^{k-3} \Gcal_j}\sbr{\nbr{\sum_{j=0}^{k-2} \mub(\wt\thetab_j) - \etab^*_j}_2^2}}.
    \end{split}
    \end{align}
    Then, plugging in \eqref{eq:pf_bias_bound_4} and \eqref{eq:pf_bias_bound_5} to \eqref{eq:pf_bias_decomposition} gives
    \begin{align*}
        &\E_{\bigcup_{j=0}^{k-2} \Gcal_j}\sbr{\nbr{\sum_{j=0}^{k-1}\mub(\wt\thetab_j) - \etab^*_j}_2^2} 
        \le 3 B_F^2 B_\theta^2 \rbr{\Cbias + L_F^2 \delta t^2 \frac{\Cvar}{q} (k-1)} \\
        &+ 2 \rbr{\sqrt{\Cbias} B_F B_\theta + L_F \delta t B_F B_\theta \sqrt{\frac{\Cvar}{q}(k-1)}} \sqrt{\E_{\bigcup_{j=0}^{k-3} \Gcal_j}\sbr{\nbr{\sum_{j=0}^{k-2} \mub(\wt\thetab_j) - \etab^*_j}_2^2}} \\
        &+ \rbr{1 + 2 L_F \delta t + 3 L_F^2 \delta t^2} \E_{\bigcup_{j=0}^{k-3} \Gcal_j}\sbr{\nbr{\sum_{j=0}^{k-2} \mub(\wt\thetab_j) - \etab^*_j}_2^2}\,.
    \end{align*}
    We further obtain because $2ab \le \lambda a^2 + b^2/\lambda$ holds for any $a,b \in \R, \lambda > 0$ that
    \begin{align*}
        &2 \rbr{\sqrt{\Cbias} B_F B_\theta + L_F \delta t B_F B_\theta \sqrt{\frac{\Cvar}{q}(k-1)}} \sqrt{\E_{\bigcup_{j=0}^{k-3} \Gcal_j}\sbr{\nbr{\sum_{j=0}^{k-2} \mub(\wt\thetab_j) - \etab^*_j}_2^2}} \\
        \le &L_F \delta t \E_{\bigcup_{j=0}^{k-3} \Gcal_j}\sbr{\nbr{\sum_{j=0}^{k-2} \mub(\wt\thetab_j) - \etab^*_j}_2^2} + B_F^2 B_\theta^2 \rbr{\sqrt{\frac{\Cbias}{L_F \delta t}} + \sqrt{L_F \delta t \frac{\Cvar}{q}(k-1)}}^2
    \end{align*}
    holds for $\lambda = L_F \delta t$. 
    Let $E_{k} = \mathbb{E}_{\bigcup_{j=0}^{k-2} \Gcal_j}\sbr{\nbr{\sum_{j=0}^{k-1}\mub(\wt\thetab_j) - \etab^*_j}_2^2}$. When taking $0 < \delta t \le 1/4L_F$, the above then leads to $E_{k} \le \alpha + \beta E_{k-1}$ where
    \begin{align}
    \begin{split}
        \alpha = &B_F^2 B_\theta^2 \rbr{\frac{\Cbias}{L_F \delta t} + 2 L_F T \frac{\Cvar}{q} + 3 \Cbias + 2\sqrt{\Cbias \frac{\Cvar}{q} K}} \\
        \ge &B_F^2 B_\theta^2 \rbr{\frac{\Cbias}{L_F \delta t} + L_F \delta t (1 + 3 L_F \delta t) \frac{\Cvar}{q} (k-1) + 3 \Cbias + 2\sqrt{\Cbias \frac{\Cvar}{q} (k-1)}}, \\
        \beta = &1 + 4 L_F \delta t \ge 1 + 3 L_F \delta t (1 + L_F \delta t)\,.
    \end{split}
    \end{align}
    Since $E_1 = \E_{\Gcal_0}[\|\mub(\wt\thetab_0) - \etab^*_0\|_2^2] \le \Cbias B_F^2 B_\theta^2 \le \alpha$ by \Cref{asm:randomized_constraints_biased}, along with the boundedness of $F(\wt\thetab_0)$, we have by induction that 
    \begin{align}\label{eq:Proof:Kalphabetabound}
    \begin{split}
        E_k \le &\alpha + \beta E_{k-1} \le \alpha \sum_{j=0}^{k-1} \beta^j = \alpha \frac{\beta^k - 1}{\beta - 1} 
        \le \alpha \frac{k \beta^{k-1} (\beta-1)}{\beta -1}
        \le k \alpha \beta^{k-1}.
    \end{split}
    \end{align}
    Note that the constant $\beta$ depends on the time-step size $\delta t$ and thus the last inequality in \eqref{eq:Proof:Kalphabetabound} brings that dependence to the numerator. 
    Overall we obtain that when $0 < \delta t \le 1/4L_F$, the bias term in \eqref{eq:variance_bias_decomposition} can be bounded as
    \begin{align}\label{eq:pf_bias_bound}
    \begin{split}
        &E_{\bigcup_{j=0}^{k-2} \Gcal_j}\sbr{\nbr{\sum_{j=0}^{k-1}\mub(\wt\thetab_j) - \etab^*_j}_2^2} \le k \alpha \rbr{1 + 4 L_F \delta t}^{k-1} \le k \alpha e^{4 L_F k \delta t}.
    \end{split}
    \end{align}

    \emph{Bounding global error}: 
    For any $k=1, \cdots, K$, we observe that 
    \begin{align*}
        \wt\thetab_{k} - \thetab^*_{k} = \rbr{\thetab^*_0 + \delta t \sum_{j=0}^{k-1} \wt\etab_j} - \rbr{\thetab^*_0 + \delta t \sum_{j=0}^{k-1} \etab^*_j} = \delta t \sum_{j=0}^{k-1} \wt\etab_j - \etab^*_j.
    \end{align*}
    Then, plugging the variance and bias bounds, \eqref{eq:pf_variance_bound} and \eqref{eq:pf_bias_bound}, in \eqref{eq:variance_bias_decomposition} implies that
    \begin{align}\label{eq:pf_stepwise_error}
    \begin{split}
        &\E_{\bigcup_{j=0}^{k-1} \Gcal_j}\sbr{\nbr{\wt\thetab_{k} - \thetab^*_{k}}_2^2} 
        = \delta t^2 \E_{\bigcup_{j=0}^{k-1} \Gcal_j}\sbr{\nbr{\sum_{j=0}^{k-1}\wt\etab_j - \etab^*_j}_2^2} 
        \le 2 k \delta t^2 \rbr{\frac{\Cvar}{q} B_F^2 B_\theta^2 + \alpha e^{4 L_F k \delta t}} \\
        &\le 2 k \delta t^2 B_F^2 B_\theta^2 \rbr{\frac{\Cvar}{q} + e^{4 L_F k \delta t} \rbr{\frac{\Cbias}{L_F \delta t} + 2 L_F T \frac{\Cvar}{q} + 3 \Cbias + 2\sqrt{\Cbias \frac{\Cvar}{q} K}}} \\
        &\le 2 k \delta t^2 B_F^2 B_\theta^2 \rbr{\frac{\Cvar}{q} + e^{4 L_F k \delta t} \rbr{\frac{2 \Cbias}{L_F \delta t} + 3 L_F T \frac{\Cvar}{q} + 3 \Cbias}},
    \end{split}
    \end{align}
    where the last inequality holds because $2\sqrt{\Cbias \frac{\Cvar}{q} K} \le \frac{\Cbias}{L_F \delta t} + \frac{\Cvar}{q} L_F T$. 
\end{proof}

\begin{corollary}\label{cor:UnbiasedCase}
Let us consider the same setup as in \Cref{thm:randomized_constraints_convergence_biased} except that the increment function $\wt\Fb$ is unbiased in the sense that $C_b = 0$ in \Cref{asm:randomized_constraints_biased}. Then, the error is bounded as 
\[
\E_{\bigcup_{j=0}^{K-1} \Gcal_j}\sbr{\nbr{\wt\thetab_{K} - \thetab^*_{K}}_2} 
        \le \sqrt{2} B_F B_{\theta}\sqrt{\frac{C_v}{q}T(1 + 3L_F T \mathrm e^{4 L_F T})\delta t}\,.
\]
\end{corollary}
\Cref{cor:UnbiasedCase} shows that if an unbiased increment function $\wt\Fb$ is available, then the corresponding randomized time stepping is consistent in the sense that the global error can be made arbitrarily small by increasing the number of replicates $q$ and decreasing the time-step size $\delta t$.

\section{Randomized time stepping with least-squares increments}\label{sec:RTSInc}
We now consider the specific case that the increment function is given via a least-squares problem as in \eqref{eq:increment_function}. We propose randomized increment functions based on sketching and show that such randomized increment functions improve the conditioning of least-squares problems.

\subsection{Randomized least-squares increments via sketching}\label{sec:randomized_state_estimates}
We now develop a randomized increment function for increment functions given via a least-squares problem as in \eqref{eq:increment_function}.
We define the randomized increment function via sketching as 
\begin{align}\label{eq:least_square_randomized_state_estimation}
        \wt{\Fb}(\thetab; \Gammab) 
        = \argmin_{\etab \in \range(\Gammab)} \nbr{\Jb(\thetab) \etab - \fb(\thetab)}_2^2 
        = \Gammab \argmin_{\vb \in \R^m} \nbr{\Jb(\thetab) \Gammab \vb - \fb(\thetab)}_2^2,
    \end{align}
where $\Gammab$ is a random matrix of size $p \times m$ with $m \leq p$. Note that we will choose $\Gammab$ such that $\operatorname{rank}(\Jb(\thetab)\Gammab) = m$ and thus the least-squares problem in \eqref{eq:least_square_randomized_state_estimation} has a unique solution. 
Applying $\Gammab$ to $\Jb(\thetab)$ can be viewed as a random linear embedding of the update $\etab$ onto the column space of $\Jb(\thetab)$ from $\R^p$ into $\R^m$~\cite{halko2011finding,woodruff2014sketching}.
The randomized time stepping scheme associated with \eqref{eq:least_square_randomized_state_estimation} is then given by $\wt\thetab_{k+1} = \wt\thetab_k + \delta t \wt\etab_k$ for $k = 0, 1, \dots, K-1$ with
\begin{align}\label{eq:discrete_time_randomNG}
    \wt\etab_k = \frac{1}{q} \sum_{i=1}^{q} \wt{\Fb}(\wt\thetab_k; \Gammab_{k,i}), \quad 
    \wt{\Fb}(\wt\thetab_k; \Gammab_{k,i}) = \Gammab_{k,i} \argmin_{\vb \in \R^m} \nbr{\Jb(\wt\thetab_k) \Gammab_{k,i} \vb - \fb(\wt\thetab_k)}_2^2.
\end{align}

It is critical to note that we apply sketching from the right, which means that the unknown $\bfeta$ of the least-squares problem is sketched. This is different from many other sketching schemes for least-squares problems where the sketching is applied from the left \cite{sarlos2006improved,rokhlin2008fast,drineas2011faster}. To the best knowledge of the authors, the first randomized time stepping scheme of a similar nature has been proposed in \cite{berman2024randomized}. However, the work \cite{berman2024randomized} provides empirical results only. Our scheme given in \eqref{eq:discrete_time_randomNG} is a different formulation than the randomized time stepping proposed in \cite{lam2024randomizedlowrankrungekuttamethods,lindsey2025mne} that sketch the least-squares problem underlying the increment function from the left and so still solve for a $p$-dimensional parameter vector rather than an $m$-dimensional vector as \eqref{eq:discrete_time_randomNG}. 

Some common examples of random embeddings $\Gammab$ include Gaussian and random unitary embeddings~\cite[\S 3.3]{vershynin2018high}. A Gaussian embedding $\Gb \sim \Ncal(0,1/m)^{p \times m}$ is a $p \times m$ Gaussian random matrix with $\iid$ entries drawn from a normal distribution $\Ncal(0,1/m)$. 
A $\Gammab \sim \mathrm{Haar}(\mathrm{Stiefel}(p,m))$ with orthonormal columns is referred to as a random unitary embedding, which is drawn from the Haar measure of the Stiefel manifold of $p \times m$ matrices with orthonormal columns. It can be generated by orthonormalizing a Gaussian embedding $\Gb$.

In terms of efficiency, sketching reduces the least-squares problem dimension from $p$ to $m$ and leads to lower computational costs: computing the sketched matrix $\Jb(\thetab) \Gammab$ in \eqref{eq:least_square_randomized_state_estimation} scales as  $O(npm)$ and the costs of solving the randomized least-squares problem in \eqref{eq:discrete_time_randomNG} scale as $O(nm^2)$. In contrast, solving the original least-squares problem in \eqref{eq:increment_function} or its regularized versions in \eqref{eq:discrete_time_det_tikhonov} or \eqref{eq:discrete_time_det_tsvd} incurs computational costs that scale as $O(np^2)$. 
It is also worth highlighting that the $q$ local samples at each time step $k = 0, 1, \dots, K - 1$ can be computed fully in parallel. 

\subsection{Random sketching can improve conditioning} 
\label{sec:sketching_conditioning}
In this section, we quantify the improvement in conditioning achieved by randomization via sketching.
We show that for a matrix $\Jb(\thetab)\in \R^{n \times p}$ with full column rank, the condition number of the sketched matrix $\kappa(\Jb(\thetab)\Gammab)$ with a random embedding $\Gammab \in \R^{p \times m}$ can be lower than that of the original matrix $\kappa(\Jb(\thetab))$ when the sketching dimension $m$ is small enough and the spectrum of $\Jb(\thetab)$ decays. We show analogous results via matrix concentration and random matrix theory.

\begin{proposition}\label{prop:CondBound}
Let $\Jb(\thetab) \in \mathbb{R}^{n \times p}$ with $n \geq p$ and full rank $p$. Let further $\Gammab \sim \operatorname{Haar}(\operatorname{Stiefel}(p, m))$ with $1 \leq m \leq p$. Fix $0 \leq \epsilon < 1$ and $0 < \delta < 1$ and let $m \leq \ell \leq p$ be such that 
\begin{equation}\label{eq:PropCondBoundMEllCondition}
m \gtrsim \log\left(\frac{p}{\delta}\right)\,,\qquad \ell \gtrsim \frac{m}{\epsilon^2}\log\left(\frac{m}{\delta}\right)\,,
\end{equation}
holds, then the condition number can be bounded as
\begin{equation}\label{eq:CondAGammaStatement}
\kappa(\Jb(\thetab)\Gammab) \leq \frac{\sigma_{\text{max}}(\Jb(\thetab))}{\sigma_{\text{min}}(\Jb(\thetab)\Gammab)} \leq \sqrt{\frac{1}{1 - \epsilon}\frac{p}{\ell}} \frac{\sigma_1(\Jb(\thetab))}{\sigma_{\ell}(\Jb(\thetab))}\,,
\end{equation}
with probability at least $1 - \delta$. The constants hidden in \eqref{eq:PropCondBoundMEllCondition} are independent of $n, p, m, \ell, \epsilon, \delta$.
\end{proposition}

Before we provide a proof, let us remark on the difference of our result to randomized SVD. 
The random subspace constraints in \eqref{eq:least_square_randomized_state_estimation} can be viewed as a form of randomized regularization for the poorly conditioned least-squares problems in \eqref{eq:increment_function}. 
Borrowing the intuition from randomized SVD~\cite{halko2011finding}, the sketched matrix $\Jb(\thetab)\Gammab$ approximately preserves the top $m$ singular values of $\Jb(\thetab)$~\cite{halko2011finding,gu2015subspace}, as well as their corresponding singular vectors~\cite{saibaba2019randomized,dong2024efficient}. 
Therefore, the random subspace constraints in \eqref{eq:least_square_randomized_state_estimation} is effectively a form of randomized truncation, analogous to the truncated SVD in \eqref{eq:discrete_time_det_tsvd}. 
Despite the shared intuition between randomized SVD and random subspace constraints, the rich analyses for randomized SVD~\cite{halko2011finding} have no obvious extensions to \eqref{eq:least_square_randomized_state_estimation} due to a critical nuance---randomized SVD has at least one built-in power iteration, \ie the analyses are based on $\Jb(\thetab)^\top \Jb(\thetab) \Gammab$ or $(\Jb(\thetab) \Jb(\thetab)^\top)^q \Jb(\thetab) \Gammab$ for some $q \ge 1$, instead of $\Jb(\thetab) \Gammab$~\cite{halko2011finding,gu2015subspace,martinsson2020randomized}. 
Conditioning of the plain sketched matrix $\Jb(\thetab)\Gammab$ (without power iterations) is less studied. To the best of our knowledge, the most relevant result in this setting is \cite[Lemma 22]{derezinski2024fine} that establishes a tail bound for the smallest singular value of $\Jb(\thetab)\Gammab$ with a special sparse embedding $\Gammab$ named LESS embedding~\cite{derezinski2021sparse}.

We now show two lemmata, which will help with the proof of \Cref{prop:CondBound}. \nocite{doi:10.1142/S1793536911000787}

\begin{lemma}\label{lm:BetaConcentration}
Let $\ell, p \in \mathbb{N}$ with $1 \leq \ell < p$ and $Z \sim \operatorname{Beta}(\frac{\ell}{2}, \frac{p - \ell}{2})$. For $0 < \epsilon \leq \min\{1, \frac{p - \ell}{\ell}\}$, the bound
\begin{equation}\label{eq:LemmaBetaConcentration}
\mathbb{P}\left[\left|Z - \frac{\ell}{p}\right| > \epsilon \frac{\ell}{p}\right] \leq 4 \exp\left(-\frac{\epsilon^2 \ell}{64}\right)
\end{equation}
holds. 
\end{lemma}

One approach to prove \Cref{lm:BetaConcentration} is to use that ratios of independent random variables with $\chi^2$ distributions have Beta distributions; see, e.g., \cite[Section 2, p.~212]{JohnsonRV}. One can then apply standard concentration bounds for $\chi^2$ random variables \cite[Lemma~1]{10.1214/aos/1015957395} to derive the bound \eqref{eq:LemmaBetaConcentration}. We leave the proof for the appendix.

\begin{lemma}\label{lm:W1SVDBoundNEW}
Let $\Vb \in \mathbb{R}^{p \times \ell}$ be a matrix with orthonormal columns and $\Gammab \sim \operatorname{Haar}(\operatorname{Stiefel}(p,m))$ be a random unitary embedding of size $p \times m$ with $1 \leq m \leq \ell \leq p$. Fix $0 < \epsilon \leq 1$ and $0 < \delta < 1$. If
\begin{equation}\label{eq:W1SVDBoundNEW:CondMEll}
m \geq 64 \log \left(\frac{8p}{\delta}\right)\,,\qquad \ell \geq \frac{6}{\epsilon^2} m \log\left(\frac{4m}{\delta}\right)\,,
\end{equation}
then, with probability at least $1 - \delta$,
\begin{equation}\label{eq:W1SVDBoundResultNew}
\sqrt{\frac{\ell}{p}(1 - \epsilon)} \leq \sigma_{\text{min}}(\Vb^{\top}\Gammab) \leq \sigma_{\text{max}}(\Vb^{\top}\Gammab) \leq \sqrt{\frac{\ell}{p}(1 + \epsilon)}\,.
\end{equation}
\end{lemma}
\begin{proof}
First, notice that the singular values of $\Vb^{\top}\Gammab$ are the square roots of the eigenvalues of $\Gammab^{\top}\Vb\Vb^{\top}\Gammab$. Because the Haar distribution is invariant under rotations, we can introduce a matrix $\Qb$ of size $p \times p$ such that $\Qb^{\top}(\Vb\Vb^{\top})\Qb = \operatorname{diag}(\Ib_{\ell}, \boldsymbol{0}_{p - \ell})$ while
\[
\Gammab^{\top}\Qb^{\top}(\Vb\Vb^{\top})\Qb\Gammab \stackrel{d}{=} \Gammab^{\top}\Vb\Vb^{\top}\Gammab\,,
\]
holds, where $\stackrel{d}{=}$ denotes equality in distribution. Thus, we can assume without loss of generality that $\Vb\Vb^{\top} = \operatorname{diag}(\Ib_{\ell}, \boldsymbol{0}_{p - \ell})$, which means that
\[
\Gammab^{\top}\Vb\Vb^{\top}\Gammab = \sum_{i = 1}^{\ell} \rb_i\rb_i^{\top}\,,
\]
where the $m$-dimensional vectors $\rb_1, \dots, \rb_{p}$ denote the rows of $\Gammab = [\rb_1^{\top}; \dots; \rb_p^{\top}]$. Recall that $\Gammab$ has orthonormal columns so that
\[
\sum_{i = 1}^p \rb_i \rb_i^{\top} = \Gammab^{\top}\Gammab = \Ib_m\,.
\]
We now relate the rows $\rb_1, \dots, \rb_p$ to uniformly distributed samples on the sphere $\mathbb{S}^{p - 1}$ in $p$ dimensions: Extend $\Gammab$ to $\Ub \sim \operatorname{Haar}(p)$ of size $p \times p$ so that $\rb_i^{\top} = \eb_i^{\top}\Gammab = [\eb_i^{\top}\Ub]_{1:m}$, which means the $i$-th row of $\Gammab$ corresponds to the first $m$ elements of the $i$-th row of $\Ub$. Now recall that the rows of $\Ub$ are uniformly distributed on $\mathbb{S}^{p-1}$. Let the $i$-th row of $\Ub$ be denoted as $\ub_i$. Let now $\gb = [g_1, \dots, g_p]$ with \emph{i.i.d.} $g_1, \dots, g_p \sim \mathcal{N}(0, 1)$ so that $\ub_i \stackrel{d}{=} \gb/\|\gb\|_2$. Now notice that $\|\rb_i\|_2^2 \stackrel{d}{=} \frac{\sum_{i = 1}^m g_i^2}{\sum_{i = 1}^p g_i^2} \stackrel{d}{=} X/(X + Y)$ with independent $\chi^2$ random variables $X, Y$ with $m$ and $p-m$ degrees of freedom, respectively.  
Thus, for each $i = 1, \dots, p$, the squared norm $\|\rb_i\|_2^2$ is distributed as a Beta random variable with $\operatorname{Beta}(m/2, (p - m)/2)$ with mean $m/p$; see, e.g., \cite[Section 2, p.~212]{JohnsonRV}. 

We now define the event
\[
\mathcal{E} = \left\{\max_{1 \leq i \leq p}  \|\rb_i\|_2^2 - \frac{m}{p} \leq \alpha \frac{m}{p}\right\}\,,
\]
where 
\[
\alpha = \sqrt{\frac{64}{m}\log\left(\frac{8p}{\delta}\right)}\,.
\]
Now notice that $\alpha \leq 1$ because of the condition on $m$ given in \eqref{eq:W1SVDBoundNEW:CondMEll}. Let us now distinguish between two cases. First, $\alpha > (p - m)/m$, then $(1 + \alpha)m/p \geq 1$ and thus $\mathcal{E}$ holds with probability 1 because $\|\rb_i\|_2^2 \leq 1$ by definition. If $\alpha \leq (p-m)/m$, then $\alpha \leq \min\{1,(p-m)/m\}$ because $\alpha \leq 1$ always holds because of \eqref{eq:W1SVDBoundNEW:CondMEll}. 
Thus, we can invoke Lemma~\ref{lm:BetaConcentration} and obtain
\[
\mathbb{P}\left[\|\rb_i\|_2^2 - \frac{m}{p} > \alpha \frac{m}{p}\right] \leq 4 \exp\left(-\frac{\alpha^2 m}{64}\right)\,.
\]
With a union bound over $i = 1, \dots, p$ we obtain
\begin{equation}\label{eq:W1SVDBoundNEW:BoundOnMathcalE}
\mathbb{P}\left[\mathcal{E}\right] \geq 1- 4 p \exp\left(- \frac{\alpha^2 m}{64}\right) = 1 - 4p \exp\left(-\log\left(\frac{8p}{\delta}\right)\right) = 1 - \frac{\delta}{2}\,
\end{equation}
holds. Thus, as long as \eqref{eq:W1SVDBoundNEW:CondMEll} is satisfied, we have $\mathbb{P}[\mathcal{E}] \geq 1 - \delta/2$. 

We now work on the event $\mathcal{E}$ and condition on a fixed realization $\Gammab$ corresponding to an element in $\mathcal{E}$. For $i = 1, \dots, p$, we define $\Rb_i = r_i r_i^{\top}$ and recall that all $\Rb_i$ are symmetric positive-semidefinite matrices. Furthermore, because of $\mathcal{E}$, we have 
\begin{equation}\label{eq:W1SVDBoundNEW:UpperB}
\lambda_{\max}(\Rb_i) = \|\rb_i\|_2^2 \leq (1 + \alpha)\frac{m}{p}\,.
\end{equation}
Given $\Gammab$, we now sample $\Zb_1, \dots, \Zb_{\ell}$ from $\mathcal{R} = \{\Rb_1, \dots, \Rb_p\}$ uniformly without replacement. For a single sample $\Zb_1$ from $\mathcal{R}$, this means that $\Zb_1$ is uniformly distributed on $\mathcal{R}$ and thus 
\[
\mathbb{E}[\Zb_1 | \Gammab] = \frac{1}{p}\sum_{i = 1}^p \Rb_i = \frac{1}{p}\Gammab^{\top}\Gammab = \frac{1}{p}\Ib_m\,.
\]
We set 
\[
\mu_{\text{min}} = \ell \lambda_{\min}(\mathbb{E}[\Zb_1 | \Gammab]) = \ell/p\,,\qquad \mu_{\max} = \ell \lambda_{\max}(\mathbb{E}[\Zb_1 | \Gammab]) = \ell/p\,, 
\]
and are now ready to apply \cite[Theorem~2.2]{doi:10.1142/S1793536911000787}: For $\bar{\Zb} = \sum_{i = 1}^{\ell} \Zb_i$ and $0 \leq \epsilon < 1$, it holds
\begin{align}
\mathbb{P}\left[\lambda_{\min}(\bar{\Zb}) \leq (1 - \epsilon)\frac{\ell}{p}\, \big|\, \Gammab\right] & \leq m \left(\frac{\mathrm e^{-\epsilon}}{(1 - \epsilon)^{1 - \epsilon}}\right)^{\ell/((1 + \alpha) m)}\label{eq:W1SVDBoundNEW:UpperTailA}\,,\\
\mathbb{P}\left[\lambda_{\max}(\bar{\Zb}) \geq (1 + \epsilon)\frac{\ell}{p}\, \big|\, \Gammab\right] & \leq m\left(\frac{\mathrm e^{\epsilon}}{(1 + \epsilon)^{1 + \epsilon}}\right)^{\ell/((1 + \alpha)m)}\label{eq:W1SVDBoundNEW:UpperTailB}\,.
\end{align}
Because
\[
\max\left\{\frac{\mathrm e^{\epsilon}}{(1 + \epsilon)^{1 + \epsilon}}, \frac{\mathrm e^{-\epsilon}}{(1 - \epsilon)^{1-\epsilon}}\right\} \leq \mathrm \exp\left(-\frac{\epsilon^2}{3}\right)\,,\qquad \text{ for } 0 \leq \epsilon < 1\,,
\]
and $\bar{\Zb} - \frac{\ell}{p} \Ib_m$ is symmetric, we obtain with a union bound
\begin{equation}\label{eq:W1SVDBoundNEW:GammaVVGammaBoundConditioned}
\mathbb{P}\left[\|\bar{\Zb} - \frac{\ell}{p} \Ib_m \|_2 \geq \epsilon \frac{\ell}{p} \,\big|\, \Gammab\right] \leq 2m \exp\left(-\frac{\epsilon^2}{3} \frac{\ell}{(1 + \alpha)m}\right)\,.
\end{equation}
If we set $\ell$  as in \eqref{eq:W1SVDBoundNEW:CondMEll} and because the choice of $\alpha$ gives $1/(1 + \alpha) \geq 1/2$, the right-hand side of \eqref{eq:W1SVDBoundNEW:GammaVVGammaBoundConditioned} is at most $\delta/2$. 

For making the following tower rule argument work, it is important to note that $\mu_{\text{min}}, \mu_{\text{max}}, \mathbb{E}[\Zb_1 | \Gammab]$, and the upper bound \eqref{eq:W1SVDBoundNEW:UpperB}, and thus the right-hand side of \eqref{eq:W1SVDBoundNEW:UpperTailA}--\eqref{eq:W1SVDBoundNEW:UpperTailB}, are the same for all $\Gammab$ corresponding to the event $\mathcal{E}$. 
Thus, using the tower rule,  we combine \eqref{eq:W1SVDBoundNEW:GammaVVGammaBoundConditioned} with \eqref{eq:W1SVDBoundNEW:BoundOnMathcalE} to obtain
\[
\mathbb{P}\left[\|\bar{\Zb} - \frac{\ell}{p} \Ib_m \|_2 \geq \epsilon \frac{\ell}{p}\right] \leq \mathbb{P}\left[\|\bar{\Zb} - \frac{\ell}{p} \Ib_m \|_2 \geq \epsilon \frac{\ell}{p} \,\big|\, \Gammab\right] \mathbb{P}[\mathcal{E}] +  \mathbb{P}[\mathcal{E}^c] \leq \delta\,,
\]
for $\alpha$ and $\ell$ as given in \eqref{eq:W1SVDBoundNEW:CondMEll}and $\mathcal{E}^c$ denoting the complement of $\mathcal{E}$. Note that we used the trivial inequality $\mathbb{P}[\|\bar{\Zb} - \frac{\ell}{p} \Ib_m \|_2 \geq \epsilon \frac{\ell}{p} | \mathcal{E}^c] \leq 1$.

Now we use that $\bar{\Zb} \stackrel{d}{=} \Gammab^{\top}\Vb\Vb^{\top}\Gammab$ because the Haar law is left-invariant under permutation matrices and thus the rows are exchangeable (i.e., $\sum_{i = 1}^{\ell} \rb_i \rb_i^{\top} \stackrel{d}{=} \sum_{i \in S}\rb_i\rb_i^{\top}$ for any subset $S \subset \{1, \dots, p\}$ of size $|S| = \ell$; in particular, for the $\ell$ samples $\Zb_1, \dots, \Zb_{\ell}$ from $\mathcal{R}$). Furthermore, the square roots of the non-zero eigenvalues of $\Gammab^{\top} \Vb \Vb^{\top}\Gammab$ are the singular values of $\Vb^{\top}\Gammab$, which shows \eqref{eq:W1SVDBoundResultNew}.

\end{proof}

Building on Lemmata~\ref{lm:BetaConcentration} and \ref{lm:W1SVDBoundNEW}, we can now state the proof of Proposition~\ref{prop:CondBound}.

\begin{proof}[Proof of Proposition~\ref{prop:CondBound}]
Consider the SVD of $\Jb(\thetab) = \Jb = \Ub\Sigmab \Vb^{\top}$ with $\Sigmab = \operatorname{diag}(\sigma_1, \dots, \sigma_p)$ and partition $\Vb = [\Vb_1, \Vb_2]$ with $\Vb_1 \in \mathbb{R}^{p \times \ell}$ containing the first $\ell$ columns of $\Vb$ and $\Vb_2 \in \mathbb{R}^{p \times (p - \ell)}$ containing the other $p - \ell$ columns. Note that we dropped the $\thetab$ dependence in the notation for convenience. We first upper bound the largest singular value of $\Jb\Gammab$ by noting that 
\[
\sigma_{\text{max}}^2(\Jb\Gammab) = \lambda_{\text{max}}(\Gammab^{\top}\Jb^{\top}\Jb\Gammab) = \max_{\|\xb\|_2 = 1} \xb^{\top}\Gammab^{\top}\Jb^{\top}\Jb\Gammab \xb\,.
\]
Using that $\|\Jb\yb\|_2^2 \leq \sigma_1^2(\Jb)\|\yb\|_2^2$ for  $\yb \in \mathbb{R}^{p}$ and setting $\yb = \Gammab \xb$ gives
\[
\max_{\|\xb\|_2 = 1} \xb^{\top}\Gammab^{\top}\Jb^{\top}\Jb\Gammab \xb \leq \sigma_1^2(\Jb) \max_{\|\xb\|_2 = 1} \xb^{\top}\Gammab^{\top}\Gammab \xb = \sigma_1^2(\Jb)\lambda_{\text{max}}(\Gammab^{\top}\Gammab)
\]
and thus 
\begin{equation}\label{eq:UpperBoundSigmaCondProof}
\sigma_{\text{max}}(\Jb\Gammab) \leq \sigma_1(\Jb)\,,
\end{equation}
where we used that $\Gammab$ has orthogonal columns and thus $\lambda_{\text{max}}(\Gammab^{\top}\Gammab) = 1$.

Let us now turn to deriving a lower bound for the smallest singular value of $\Jb\Gammab$. For this, we consider the decomposition
\[
\Gammab^{\top}\Jb^{\top}\Jb\Gammab = \Gammab^{\top} \Vb \Sigmab^2 \Vb^{\top}\Gammab = \Gammab^{\top} [\Vb_1, \Vb_2] \begin{bmatrix}\Sigmab_1^2 & 0\\ 0 & \Sigmab_2^2\end{bmatrix} \begin{bmatrix}\Vb_1^{\top}\\\Vb_2^{\top}\end{bmatrix}\Gammab = \Gammab^{\top}\Vb_1\Sigmab_1^2\Vb_1^{\top}\Gammab + \Gammab^{\top}\Vb_2\Sigmab_2^2\Vb_2^{\top}\Gammab\,,
\]
where $\Sigmab_1^2 = \operatorname{diag}(\sigma_1^2, \dots, \sigma_{\ell}^2)$ and $\Sigmab_2^2 = \operatorname{diag}(\sigma_{\ell + 1}^2, \dots, \sigma_p^2)$ are diagonal matrices. We define the random matrix $\Gammab_1$ of size $\ell \times m$ and $\Gammab_2$ of size $(p - \ell) \times m$ as
\[
\Gammab_1 = \Vb_1^{\top}\Gammab\,,\qquad \Gammab_2 = \Vb_2^{\top}\Gammab\,,
\]
so that
\[
\Gammab^{\top}\Jb^{\top}\Jb\Gammab = \Gammab_1^{\top}\Sigmab_1^2\Gammab_1 + \Gammab_2^{\top}\Sigmab_2^2\Gammab_2\,.
\]
Because $\Vb_1$ and $\Vb_2$ have orthogonal columns, we obtain that for all $\xb \in \mathbb{R}^m$ we have 
\[
\|\Jb\Gammab \xb\|_2^2 = \|\Sigmab_1\Gammab_1 \xb\|_2^2 + \|\Sigmab_2\Gammab_2 \xb\|_2^2 \geq \|\Sigmab_1 \Gammab_1 \xb\|_2^2\,.
\]
Using that $\sigma_{\ell}(\Jb)$ is the smallest singular value of $\Sigmab_1$, we have $\|\Jb\Gammab \xb\|_2^2 \geq \sigma_{\ell}^2(\Jb)\|\Gammab_1 \xb\|_2^2$, which we can minimize over $\|\xb\|_2 = 1$ to obtain almost surely
\[
\sigma_{\text{min}}^2(\Jb\Gammab) = \min_{\|\xb\|_2 = 1} \|\Jb\Gammab \xb\|_2^2 \geq \sigma_{\ell}^2(\Jb) \min_{\|\xb\|_2 = 1}\|\Gammab_1 \xb\|_2^2 = \sigma_{\ell}^2(\Jb)\sigma_{\text{min}}^2(\Gammab_1)\,.
\]
Thus, we now need to bound the smallest singular value $\sigma_{\text{min}}(\Gammab_1)$ of $\Gammab_1$. Recall that $\Gammab_1 = \Vb_1^{\top}\Gammab$ and so Lemma~\ref{lm:W1SVDBoundNEW} is applicable and the proposition follows. 
\end{proof}

Random matrix theory offers a different approach to obtaining a result analogous to \eqref{eq:CondAGammaStatement}, which we state in the following proposition.

\begin{proposition}\label{prop:RMTBound}
Let $\Jb(\thetab) \in \mathbb{R}^{n \times p}$ with full rank $p$. Consider now integers $m, \ell \in \mathbb{N}$ with  $1 \leq m < \ell \leq p$ and fixed aspect ratios $\gamma = \frac{m}{p} \in (0, 1), \nu = \frac{\ell}{p} \in (\gamma, 1)$. For $\Gammab \sim \operatorname{Haar}(\operatorname{Stiefel}(p, m))$, there exists a random variable $\alpha(p)$ depending on $p$ and the randomness from $\Gammab$ such that
\begin{align}\label{eq:cond_asymp_haar}
    \kappa(\Jb(\thetab)\Gammab) \leq \frac{\sigma_{\text{max}}(\Jb(\thetab))}{\sigma_{\text{min}}(\Jb(\thetab)\Gammab)} \leq \alpha(p)\frac{\sigma_1(\Jb(\thetab))}{\sigma_{\ell}(\Jb(\thetab))}\,, \quad \alpha(p) \overset{\mathbb{P}}{\longrightarrow} \alpha = \frac{\sqrt{\nu(1-\gamma)} + \sqrt{\gamma(1-\nu)}}{\nu - \gamma} 
\end{align}
for $p \to \infty$. 
Furthermore, for $\epsilon > 0$, there exists $P \in \mathbb{N}$ and a constant $C > 0$ depending on $\epsilon, \gamma, \nu$ only such that for all $p \geq P$,
\begin{equation}\label{eq:RMTProofCond:FinalRateResult}
    \PP\sbr{\frac{\sigma_{\text{max}}(\Jb(\thetab))}{\sigma_{\text{min}}(\Jb(\thetab)\Gammab)} \geq \kappa(\Jb(\thetab)\Gammab) > \rbr{1 + C p^{-2/3}} \alpha \frac{\sigma_1(\Jb(\thetab))}{\sigma_\ell(\Jb(\thetab))}} < \epsilon\,.
\end{equation}
\end{proposition}

\begin{proof}[Proof of \Cref{prop:RMTBound}]
Building on the proof of \Cref{prop:CondBound}, we need to lower bound $\sigma_{\text{min}}(\Gammab_1) = \lambda_{\text{min}}(\Gammab_1^{\top}\Gammab_1)^{1/2}$. 

We first show that the law of the eigenvalues of $\Gammab_1^\top \Gammab_1$ is the same as the law of the eigenvalues of Jacobi matrices, which are well studied; see, e.g., \cite[Theorem 3.1]{wachter1980limiting} and \cite[Proposition 1.1]{dumitriu2012global}.
Recall that for two independent Gaussian random matrices, $\Gb_1 \sim \Ncal(0,1)^{\ell \times m}$ and $\Gb_2 \sim \Ncal(0,1)^{(p-\ell) \times m}$, the $m \times m$ matrix $\Wb_1 = \Gb_1^\top \Gb_1 \sim \mathrm{Wishart}(\Ib_m, \ell)$ and the $m \times m$ matrix $\Wb_2 = \Gb_2^\top \Gb_2 \sim \mathrm{Wishart}(\Ib_m, p-\ell)$ follow the Wishart distribution, and for $\ell > m$, the $m \times m$ matrix $\Wb_1(\Wb_1 + \Wb_2)^{-1}$ follows the Jacobi ensemble (more generally, $\beta$-Jacobi ensemble with $\beta=1$, corresponding to real matrices).
Now, we observe that $\Gb = [\Gb_1^\top,\Gb_2^\top]^\top \sim \Ncal(0,1)^{p \times m}$ and let $\Gb = \Qb \Rb$ be the reduced QR decomposition of $\Gb$ where $\Qb \in \mathrm{Stiefel}(p,m)$ and $\Rb \in \R^{m \times m}$ is an upper triangular matrix that is invertible almost surely. 
Let $\Qb = [\Qb_1^\top,\Qb_2^\top]^\top$ where $\Qb_1 \in \R^{\ell \times m}$ contains the first $\ell$ rows in $\Qb$. Because $\Qb$ is Haar-distributed on $\operatorname{Stiefel}(p, m)$ and independent of $\Rb$, the marginal law of $\Qb_1$ coincides with $\Gammab_1$ so that $\Qb_1 \stackrel{d}{=} \Gammab_1$. 
Because $\Gammab_1^{\top}\Gammab_1$ and $\Qb_1^{\top}\Qb_1$ are real and symmetric, the map to their sorted eigenvalues is measurable and thus preserved in equality in distribution. Furthermore,  $\Qb_1^\top \Qb_1$ and the Jacobi matrix $\Wb_1(\Wb_1 + \Wb_2)^{-1}$  are similar
\begin{align*}
        \Wb_1(\Wb_1 + \Wb_2)^{-1} = \Gb_1^\top \Gb_1 (\Gb^\top \Gb)^{-1} = \Rb^\top \Qb_1^\top \Qb_1 \Rb(\Rb^\top \Rb)^{-1} = \Rb^\top \Qb_1^\top \Qb_1 (\Rb^\top)^{-1}.
    \end{align*}
Thus, the eigenvalues of $\Gammab_1^{\top}\Gammab_1$ are in distribution equal to the eigenvalues of matrices that are distributed as $\Wb_1(\Wb_1 + \Wb_2)^{-1}$. 

To study the smallest eigenvalue 
$\lambda_{\text{min}}(\Gammab_1^{\top}\Gammab_1)$, we invoke \cite[Theorem 4]{holcomb2012edge}, which studies the eigenvalues of a matrix $\Bb = \Ib_m - \Wb_1(\Wb_1 + \Wb_2)^{-1}$. Thus, bounding the largest eigenvalue of $\Bb$ leads to a lower bound on the smallest eigenvalue of  $\Wb_1(\Wb_1 + \Wb_2)^{-1}$; see also \cite[Remark~5]{holcomb2012edge}. We summarize the result of \cite[Theorem~4, Remark~5]{holcomb2012edge} as follows: For matrices distributed as $\Wb_1(\Wb_1 + \Wb_2)^{-1}$ (and thus as $\Gammab_1^{\top}\Gammab_1$), there exists a constant $C_\tau(\gamma,\nu) > 0$ independent of $p$ so that $(\lambda_{\min}(\Gammab_1^\top \Gammab_1) - \lambda_-) / \tau_p$ converges in distribution to $\Lambda$ as $p \to \infty$, where $\lambda_{\pm} = (\sqrt{\nu(1-\gamma)} \pm \sqrt{\gamma(1-\nu)})^2$, $\tau_p = C_\tau(\gamma,\nu) p^{-2/3}$ and $\Lambda$ is a random variable that follows the normalized orthogonal Tracy-Widom law \citep{tracy1996orthogonal}, which is universal; only the centering $\lambda_-$ and scaling $C_{\tau}(\gamma, \nu)$ depend on the aspect ratios $\gamma$ and $\nu$. Notice that $\ell > m$, which is required for this statement to hold, and that the aspect ratios are fixed so that \cite[Remark~6]{holcomb2012edge} is applicable.

For $\epsilon > 0$, pick now the quantile $t_{\epsilon} > 0$ so that $\mathbb{P}[\Lambda \leq -t_{\epsilon}] < \epsilon/4$ and $\mathbb{P}[\Lambda \geq t_{\epsilon}] < \epsilon/4$. Such a $t_{\epsilon}$ exists because $\Lambda$ has a continuous distribution with tails decaying to zero. Because of the convergence in distribution to $\Lambda$, there exist $P_\epsilon \in \N$ so  that for all $p \ge P_\epsilon$
\[
|\mathbb{P}[(\lambda_{\min}(\Gammab_1^\top \Gammab_1) - \lambda_-) / \tau_p \leq -t_{\epsilon}] - \mathbb{P}[ \Lambda \leq -t_{\epsilon}]| < \epsilon/4\,,
\]
holds, as well as analogously $|\mathbb{P}[(\lambda_{\min}(\Gammab_1^\top \Gammab_1) - \lambda_-) / \tau_p \geq t_{\epsilon}] - \mathbb{P}[\Lambda \geq t_{\epsilon}]| < \epsilon/4$.  
Then, for $p \geq P_{\epsilon}$, we have $\mathbb{P}[|(\lambda_{\min}(\Gammab_1^\top \Gammab_1) - \lambda_-) / \tau_p| \geq t_{\epsilon}] \leq \epsilon$. 

To obtain a bound for the singular values, notice that
\[
|\sqrt{x} - \sqrt{\lambda_-}| = \frac{|x - \lambda_-|}{\sqrt{x} + \sqrt{\lambda_-}} \leq \frac{|x - \lambda_-|}{\sqrt{\lambda_-}}
\]
holds for $x \geq 0$. Thus, we can set $C(\epsilon,\gamma,\nu) = t_{\epsilon}C_{\tau}(\gamma, \nu)/\sqrt{\lambda_-}$ and obtain 
\begin{align}\label{eq:pf_tw_sval}
        \PP\sbr{|\sigma_{\min}(\Gammab_1) - \sqrt{\lambda_-}| \geq C(\epsilon,\gamma,\nu)\,p^{-2/3}} \leq \epsilon
    \end{align}
for all $p \ge P_\epsilon$.
In particular, for the fixed aspect ratios $0 < \gamma < \nu < 1$, equation \eqref{eq:pf_tw_sval} implies that $\sigma_{\min}(\Gammab_1)$ converges in probability, 
    \begin{align*}
        \sigma_{\min}(\Gammab_1) \overset{\PP}{\longrightarrow} \sqrt{\lambda_-} = \sqrt{\nu(1-\gamma)} - \sqrt{\gamma(1-\nu)},
    \end{align*}
   as $p \to \infty$. 
    
For the rate of convergence, we again leverage \eqref{eq:pf_tw_sval} and pick $P_{\epsilon}$ so large that additionally $C (\epsilon, \gamma, \nu) p^{-2/3}/\sqrt{\lambda_-} \in (0, 1/2]$. Such a $P_{\epsilon}$ has to exist because $\lambda_{-} > 0$. For $p \geq P_{\epsilon}$, we then have $\sqrt{\lambda_-} - C (\epsilon, \gamma, \nu) p^{-2/3} > 0$ and thus with probability at least $1 - \epsilon$ that
    \[
    \frac{1}{\sigma_{\text{min}}(\Gammab_1)} \leq \frac{1}{\sqrt{\lambda_{-}} - C(\epsilon, \gamma, \nu) p^{-2/3}} = \alpha\left(1 - \frac{C(\epsilon, \gamma, \nu)}{\sqrt{\lambda_-}}p^{-2/3}\right)^{-1}\,,
    \]
    where we set $\alpha = 1/\sqrt{\lambda_-}$. For $p \geq P_{\epsilon}$, we further have $C (\epsilon, \gamma, \nu) p^{-2/3}/\sqrt{\lambda_-} \leq 1/2$. Using $(1 - x)^{-1} \leq 1 + 2x$ for $x \in [0, 1/2]$, we obtain
    \[
    \alpha\left(1 - \frac{C(\epsilon, \gamma, \nu)}{\sqrt{\lambda_-}}p^{-2/3}\right)^{-1} \leq \alpha\left(1 + 2 \frac{C(\epsilon, \gamma, \nu)}{\sqrt{\lambda_-}}p^{-2/3}\right) = \alpha\left(1 + C_{\alpha}(\epsilon, \gamma, \nu)p^{-2/3}\right)\,,
    \]
    where $C_{\alpha}(\epsilon, \gamma, \nu) = 2 C(\epsilon, \gamma, \nu)/\sqrt{\lambda_-}$. 
    Thus, setting $\alpha(p) = 1/\sigma_{\text{min}}(\Gammab_1)$ we have
    \[
\mathbb{P}\left[\alpha(p) > \alpha(1 + C_{\alpha}(\epsilon, \gamma, \nu)p^{-2/3})\right] \leq \mathbb{P}\left[\sigma_{\min}(\Gammab_1) < \sqrt{\lambda_{-}} - C(\epsilon, \gamma, \nu)p^{-2/3}\right] < \epsilon\,,
    \]
    for any $p \geq P_{\epsilon}$, which implies \eqref{eq:RMTProofCond:FinalRateResult}.
\end{proof}

\begin{remark}\label{rm:GaussianVsUnitary}
Let us compare the scaling $\alpha$ for a unitary random embedding $\Gammab \sim \operatorname{Haar}(\operatorname{Stiefel}(p, m))$ derived in \Cref{prop:RMTBound} to the scaling for a Gaussian random embedding $\Gb \sim \mathcal{N}(0, 1/m)^{p \times m}$. Again building on the proof of \Cref{prop:CondBound}, we need to bound
\[
\kappa(\Jb(\thetab)\Gb)  \leq \frac{\sigma_{\max}(\Vb(\thetab)^\top \Gb)}{\sigma_{\min}(\Gb_1(\thetab))}\frac{\sigma_{1}({\Jb(\thetab))}}{\sigma_{\ell}(\Jb(\thetab))} \,,
\]
where $\Gb_1(\thetab)  = \Vb_1(\thetab)^{\top}\Gb$ with $\Vb_1(\thetab) \in \mathbb{R}^{p \times \ell}$ containing the first $\ell$ right-singular vectors of $\Jb(\thetab)$.   
By the rotation invariance of Gaussian random matrices, for any orthogonal matrix $\Vb \in \R^{p \times p}$, $\Vb^\top \Gb$ and $\Gb$ follow the same distribution. Therefore, $\sigma_{\max}(\Vb(\thetab)^\top \Gb) \overset{d}{=} \sigma_{\max}(\Gb)$. 
Then, \cite{geman1980limit} shows that for fixed $\gamma = \frac{m}{p} \in (0,1)$, as $p \to \infty$, $\sigma_{\max}(\Gb) \overset{a.s.}{\longrightarrow} 1 + \gamma^{-1/2}$ converges almost surely.
At the same time, the rotation invariance also suggests that $\Gb_1(\thetab) = \Vb_1(\thetab)^\top \Gb \sim \Ncal(0,1/m)^{\ell \times m}$, and thus \cite{silverstein1985smallest} implies that with the fixed aspect ratios $0 < \gamma < \nu < 1$, as $p \to \infty$, $\sigma_{\min}(\Gb_1(\thetab)) \overset{a.s.}{\longrightarrow} \sqrt{\nu/\gamma} - 1$ converges almost surely, and thus overall 
\begin{align}\label{eq:cond_asymp_gaussian}
    \frac{\sigma_{\max}(\Vb(\thetab)^\top \Gb)}{\sigma_{\min}(\Gb_1(\thetab))} \overset{a.s.}{\longrightarrow} \varrho = \frac{1 + \gamma^{-1/2}}{\sqrt{\nu/\gamma} - 1} = \frac{1 + \sqrt{\gamma}}{\sqrt{\nu} - \sqrt{\gamma}} \quad \t{as}\quad p \to \infty.
\end{align}
Now notice that for any fixed aspect ratios $0 < \gamma < \nu < 1$, we have
\begin{align}\label{eq:haar_lowerrho_gauss}
    \alpha = \frac{\sqrt{\nu(1-\gamma)} + \sqrt{\gamma(1-\nu)}}{\nu - \gamma} < \frac{\sqrt{\nu} + \sqrt{\gamma}}{\nu - \gamma} < \frac{1 + \sqrt{\gamma}}{\sqrt{\nu} - \sqrt{\gamma}} = \varrho.
\end{align}
The advantage of random unitary embeddings compared to Gaussian embeddings extends beyond the tighter upper bound of the condition number. In particular, \eqref{eq:haar_lowerrho_gauss} coincides with the numerical experiments in \Cref{sec:NumExp:Sketching}, which indicate that random unitary embeddings typically lead to smaller condition numbers than Gaussian embeddings in practice.\end{remark}

\subsection{Variance and bias of sketched increments}

We now discuss \Cref{asm:randomized_constraints_biased} for  
random embeddings $\Gammab \sim \operatorname{Haar}(\operatorname{Stiefel}(p, m))$.  

\begin{proposition}\label{prop:least_square_randomized_state_estimation}
    Let $\Gammab \sim \operatorname{Haar}(\operatorname{Stiefel}(p, m))$, then the randomized increment function \eqref{eq:least_square_randomized_state_estimation} satisfies \Cref{asm:randomized_constraints_biased}(i) with $\Cbias \le 1$.
\end{proposition}

\begin{proof}[Proof of \Cref{prop:least_square_randomized_state_estimation}]
    Because $\Gammab$ has orthonormal columns, we obtain  $\rank(\Jb(\thetab)\Gammab)=m$ for any matrix $\Jb(\thetab)$ with $\rank(\Jb(\thetab))=p$, so that the least-squares problem in \eqref{eq:least_square_randomized_state_estimation} has a unique solution, and therefore the randomized increment function is well-defined. Recall \Cref{asm:randomized_constraints_biased} and in particular note that we can write 
    $\Fb(\thetab) = \Jb(\thetab)^\dagger \fb(\thetab)$ and $\wt{\Fb}(\thetab; \Gammab) = \Gammab (\Jb(\thetab)\Gammab)^\dagger \fb(\thetab)$. For a given $\thetab \in \R^p$, we define the (reduced) SVD $\Jb(\thetab) = \Ub(\thetab)\Sigmab(\thetab)\Vb(\thetab)^\top$ where $\Ub(\thetab) \in \R^{n \times p}$ and $\Vb(\thetab) \in \R^{p \times p}$ contain orthonormal columns, and $\Sigmab(\thetab) \in \R^{p \times p}$ is a diagonal matrix consisting of positive singular values in descending order. We now drop the $\thetab$ dependence in the notation to ease exposition.
    
    Let $\Vb^\top \Gammab = \Qb\Rb$ be a (reduced) QR decomposition of $\Vb^\top \Gammab$ where $\Qb \in \R^{p \times m}$ contains orthonormal columns and the upper triangular matrix $\Rb \in \R^{m \times m}$ is invertible. Notice that since $\Gammab$ is rotation-invariant, $\Qb \stackrel{d}{=} \Gammab$ holds. Denoting $\bb = \Ub^\top \fb \in \R^p$, write $\Fb = \Vb \Sigmab^{-1} \bb$ and
    \begin{align*}
        \wt{\Fb}(\thetab;\Gammab) = \Vb (\Vb^\top \Gammab) (\Sigmab \Vb^\top \Gammab)^\dagger \bb = \Vb \Qb (\Sigmab \Qb)^\dagger \bb = \Vb \Qb (\Qb^\top \Sigmab^2 \Qb)^{-1} \Qb^\top \Sigmab \bb,
    \end{align*}
    where the first equality holds because $\Ub$ has full column rank and orthonormal columns, the second equality because $\Rb$ is invertible, and the third equality follows from the definition of the Moore-Penrose pseudo inverse.

    Denote with $\Pib_{\Qb} = \Sigmab \Qb (\Qb^\top \Sigmab^2 \Qb)^{-1} \Qb^\top \Sigmab$ the orthogonal projection onto the range of $\Sigmab \Qb$ and notice that $\wt{\Fb}(\thetab;\Gammab) = \Vb \Sigmab^{-1} \Pib_{\Qb} \bb$.
    Leveraging the definition of $\mub(\thetab)$ given in \eqref{eq:Rand:MuDef}, 
    we observe that
    \begin{align*}
        &\nbr{\mub(\thetab) - \Fb(\thetab)}_2^2
        = \nbr{\Vb \Sigmab^{-1} (\E_{\Qb}\sbr{\Pib_{\Qb}} \bb - \bb)}_2^2 
        = \nbr{\Sigmab^{-1} (\Ib_p - \E_{\Qb}\sbr{\Pib_{\Qb}}) \Sigmab \Sigmab^{-1} \bb}_2^2 \\
        \le &\nbr{\Sigmab^{-1} (\Ib_p - \E_{\Qb}\sbr{\Pib_{\Qb}}) \Sigmab}_2^2 \nbr{\Sigmab^{-1} \bb}_2^2
        = \nbr{\Sigmab^{-1} (\Ib_p - \E_{\Qb}\sbr{\Pib_{\Qb}}) \Sigmab}_2^2 \nbr{\Fb(\thetab)}_2^2,
    \end{align*}
    and therefore $\Cbias \le \|\Sigmab^{-1} (\Ib_p - \E_{\Qb}\sbr{\Pib_{\Qb}}) \Sigmab\|_2^2$.

    Now, we make a key claim that $\E_{\Qb}\sbr{\Pib_{\Qb}}$ is diagonal. To see this, first notice that for any diagonal matrix $\Db = \diag(d_1,\cdots,d_p)$ with each $d_i \in \{-1,1\}$, we have $\Db \Qb \overset{d}{=} \Qb$, since $\Qb \sim \mathrm{Haar}(\mathrm{Stiefel}(p,m))$ is rotation-invariant, and $\Db$ is orthogonal. Following the definition of $\Pib_{\Qb}$, we define $\Pib_{\Db\Qb}$ and observe that the diagonal matrices $\Db, \Sigmab, \Sigmab^2$ commute. Therefore
    \begin{align*}
        \Pib_{\Db\Qb} 
        = &\Sigmab \Db \Qb (\Qb^\top \Db \Sigmab^2 \Db \Qb)^{-1} \Qb^\top \Db \Sigmab
        = \Db \Sigmab \Qb (\Qb^\top \Sigmab^2 \Db^2 \Qb)^{-1} \Qb^\top \Sigmab \Db \\
        = &\Db (\Sigmab \Qb (\Qb^\top \Sigmab^2 \Qb)^{-1} \Qb^\top \Sigmab) \Db
        = \Db \Pib_{\Qb} \Db,
    \end{align*}
    where the third equality follows from $\Db^2 = \Ib_p$.
    Thus, with $\Db \Qb \overset{d}{=} \Qb$, we obtain
    \begin{align*}
        \E_{\Qb}[\Pib_{\Qb}] = \E_{\Qb} [\Pib_{\Db\Qb}] = \Db \E_{\Qb}[\Pib_{\Qb}] \Db.
    \end{align*}
    Then, $\E_{\Qb}[\Pib_{\Qb}] = \Db \E_{\Qb}[\Pib_{\Qb}] \Db$ for all $\Db$ enforces $\E_{\Qb}[\Pib_{\Qb}]$ to be diagonal. To see this, let $M_{ij}$ be the $(i,j)$-th entry of $\E_{\Qb}[\Pib_{\Qb}]$ such that $M_{ij} \neq 0$ and $i \neq j$. For any $\Db$ such that $d_i d_j = -1$, the $(i,j)$-th entry of $\Db \E_{\Qb}[\Pib_{\Qb}] \Db$ leads to a contradiction: $d_i M_{ij} d_j = - M_{ij} \neq M_{ij}$.

    Because $\E_{\Qb}\sbr{\Pib_{\Qb}}$ is diagonal, we also have that  $(\Ib_p - \E_{\Qb}\sbr{\Pib_{\Qb}})$ is diagonal. Therefore, $(\Ib_p - \E_{\Qb}\sbr{\Pib_{\Qb}})$ commutes with the diagonal matrix $\Sigmab$. Therefore, we have 
    \begin{align*}
        \Cbias \le \nbr{\Sigmab^{-1} (\Ib_p - \E_{\Qb}\sbr{\Pib_{\Qb}}) \Sigmab}_2^2 = \nbr{(\Ib_p - \E_{\Qb}\sbr{\Pib_{\Qb}}) \Sigmab^{-1} \Sigmab}_2^2 = \nbr{\Ib_p - \E_{\Qb}\sbr{\Pib_{\Qb}}}_2^2.
    \end{align*}
    Since $\Pib_{\Qb}$ is an orthogonal projection, we have $0 \preceq \Pib_{\Qb} \preceq \Ib_p$, which implies $0 \preceq \E_{\Qb}\sbr{\Pib_{\Qb}} \preceq \Ib_p$, and therefore, $\Cbias \le \nbr{\Ib_p - \E_{\Qb}\sbr{\Pib_{\Qb}}}_2^2 \le 1$.
\end{proof}

Let us now consider the constant $C_v$ controlling the variance in \Cref{asm:randomized_constraints_biased}(ii). Note that the variance corresponding to $C_v$ can be controlled with the number of replicates $q$ and the time-step size $\delta t$ (see \Cref{thm:randomized_constraints_convergence_biased}) so that the absolute value of the upper bounds for $C_v$ is less critical than for the constant $C_b$ corresponding to the bias, which cannot be directly controlled with $q$ and $\delta t$. In particular, as our numerical experiments in \Cref{sec:coeff_bias_var_powlaw} will indicate, the following bounds are loose. We now show bounds of $C_v$ that scale with the condition number of the sketched $\Jb(\thetab)\Gammab$ rather than the unsketched matrix $\Jb(\thetab)$. The bounds depend on the $\Jb(\thetab)$ at hand, which is different from the uniform bound assumed in \Cref{asm:randomized_constraints_biased}(ii). If the set of all possible $\thetab$ is compact, one can derive a uniform bound, albeit a crude one.

\begin{proposition}\label{prop:CvCbboundwrtCond}
Let $\Gammab \sim \operatorname{Haar}(\operatorname{Stiefel}(p, m))$,  then \eqref{asm:ii} holds with a constant $C_v(\Jb(\thetab)) > 0$ depending on $\Jb(\thetab)$ that satisfies:
\begin{itemize}
\item[(i)] Under the same assumptions as \Cref{prop:CondBound}, for $\epsilon, \delta \in (0, 1)$ and suitable $m \leq \ell$, the constant $C_v(\Jb(\thetab))$ is bounded as
\begin{equation}\label{eq:CVProof:Bound1}
C_v(\Jb(\thetab)) \leq (1 - \delta)\frac{1}{1 - \epsilon} \frac{p}{\ell} \frac{\sigma_1^2(\Jb(\thetab))}{\sigma_{\ell}^2(\Jb(\thetab))} + \delta \kappa^2(\Jb(\thetab))\,.
\end{equation}
\item [(ii)] Under the same assumptions as \Cref{prop:RMTBound}, for $\epsilon > 0$ with $C > 0$ and  $p$ sufficiently large, the constant $C_v(\Jb(\thetab))$ is bounded as
\begin{equation}\label{eq:CVProof:Bound2}
C_v(\Jb(\thetab)) \leq (1 - \epsilon) \left(1 + C p^{-2/3}\right)^2\alpha^2 \frac{\sigma_1^2(\Jb(\thetab))}{\sigma_{\ell}^2(\Jb(\thetab))} + \epsilon \kappa^2(\Jb(\thetab))\,.
\end{equation}
\end{itemize}
\end{proposition}
\begin{proof}
We drop $\thetab$ as argument in $\Jb(\thetab)$ and just write $\Jb$ for brevity in this proof. Furthermore, note that $\Jb$ having full column rank $p$ and $\Gammab$ having full column rank $m$ and orthonormal columns implies $\operatorname{rank}(\Jb\Gammab) = m$

Because $\rank(\Jb\Gammab) = m$, we can write $(\Jb\Gammab)^{\dagger} = (\Gammab^{\top}\Jb^{\top}\Jb\Gammab)^{-1}\Gammab^{\top}\Jb^{\top}$. Using $\Fb = \Fb(\thetab) =  \Jb^{\dagger}(\thetab)\fb(\thetab)$ we obtain $\Jb^{\top}\Jb \Fb = \Jb^{\top} \fb$ and thus we can write $\tilde{\Fb}$ defined in \eqref{eq:least_square_randomized_state_estimation} as 
\begin{equation}\label{eq:Prop:CvCbBoundsIntermed0}
\tilde{\Fb} =  \Gammab(\Jb\Gammab)^{\dagger}\fb = \Gammab (\Gammab^{\top}\Jb^{\top}\Jb\Gammab)^{-1}\Gammab^{\top}\Jb^{\top}\Jb \Fb\,.
\end{equation}
Now notice that $\Jb^{\top} (\fb - \Jb \Fb) = \boldsymbol 0$ because $\Jb^{\top}\Jb\Fb = \Jb^{\top}\fb$ and thus also $(\Jb\Gammab)^{\top}(\fb - \Jb \Fb) = 0$ because the column space of $\Jb\Gammab$ is a subspace of the column space of $\Jb$. 
With \eqref{eq:Prop:CvCbBoundsIntermed0}, we have
\begin{equation}\label{eq:Prop:CvCbBoundsIntermed3}
\|\tilde{\Fb}\|_2 = \|\Gammab (\Jb\Gammab)^{\dagger}(\Jb \Fb + (\fb - \Jb\Fb))\|_2 \leq \|\Gammab\|_2 \|(\Jb\Gammab)^{\dagger}\|_2 \|\Jb\|_2 \|\Fb\|_2 = \frac{\sigma_{\text{max}}(\Jb)}{\sigma_{\text{min}}(\Jb\Gammab)} \|\Fb\|_2\,,
\end{equation}
where we used that $\|\Gammab\|_2 = 1$ because $\Gammab$ has orthonormal columns. 
We now use that $\mathbb{E}[\|X - \mathbb{E}[X]\|^2] = \mathbb{E}[\|X\|^2] - \|\mathbb{E}[X]\|^2 \leq \mathbb{E}[\|X\|^2]$ to obtain
\[
\mathbb{E}_{\Gammab}\left[\|\tilde{\Fb} - \mathbb{E}_{\Gammab}[\tilde{\Fb}]\|_2^2\right] \leq \mathbb{E}_{\Gammab}\left[\|\tilde{\Fb}\|_2^2\right]\,.
\]
Building on \Cref{prop:CondBound}, we obtain \eqref{eq:CVProof:Bound1}, where we used $\|\tilde{\Fb}\|_2 \leq \kappa(\Jb) \|\Fb\|_2$ that also follows from \eqref{eq:Prop:CvCbBoundsIntermed3}. 
Analogously, building on \Cref{prop:RMTBound}, we obtain for $\epsilon > 0$ with $C > 0$ that \eqref{eq:CVProof:Bound2} holds for sufficiently large $p$. 
\end{proof}

We numerically study the behavior of $C_b$ and $C_v$ in \Cref{sec:coeff_bias_var_powlaw} and demonstrate that the constant associated with the bias $\Cbias$ decreases monotonically as the sketching dimension $m$ increases; while $\Cvar$ first increases and then decreases as $m$ grows. The variance associated with $\Cvar$ can be made arbitrarily small by increasing either the number of replicates $q$ or decreasing the time-step size $\delta t$, the bias remains proportional to $\Cbias$.
To reduce the cumulative bias in \eqref{eq:rand_convergence_step_biased} that does not decrease with increasing $q$ or decreasing $\delta t$, we want a larger $m$ that leads to lower $\Cbias$.
However, when $m$ is close to $p$, the sketched least-squares problem in \eqref{eq:least_square_randomized_state_estimation} tends to suffer a similar numerical issue as the original least-squares problem in \eqref{eq:increment_function} since the condition number of $\Jb(\thetab)\Gammab$ would be similar to that of $\Jb(\thetab)$.
Therefore, balancing the condition number of $\Jb(\thetab)\Gammab$ (by choosing a smaller $m$) and the bias associated with $\Cbias$ (by choosing a larger $m$) is the key for choosing a suitable sketching dimension $m$ in the randomized increment function \eqref{eq:least_square_randomized_state_estimation}.

\section{Numerical experiments}\label{sec:NumExp}
In this section, we first empirically demonstrate the effect of the sketching dimension $m$ on the conditioning of randomized least-squares problems and then on the variance in form of $\Cvar$ and the bias in form of $\Cbias$. 
Then, we apply randomized time stepping to the Neural Galerkin scheme \cite{bruna2024neural,BERMAN2024389} and compare to deterministic regularization on problems based on the Schr\"odinger and Allen-Cahn equation.

\subsection{Sketching and condition number}\label{sec:NumExp:Sketching}
In this section, we numerically demonstrate \Cref{prop:RMTBound} and investigate the condition number of $\Jb(\thetab)\Gammab$. 
We consider a matrix $\Jb(\thetab) = \Ab \in \R^{n \times p}$ with a power law spectrum, $\sigma_i(\Ab) = i^{-\omega}$ for all $i=1,\cdots,p$ where we take $\omega = 2$. 
For $n \ge p > m$, the proof of \Cref{prop:RMTBound} implies that the condition number $\kappa(\Ab\Gammab)$ is independent of $n$. Therefore, we take $n = p = 1000$ in the numerical experiments without loss of generality.
To generate the matrices, we first draw random unitary matrices $\Ub \in \R^{n \times p}$ and $\Vb \in \R^{p \times p}$ as the left and right singular vectors by applying reduced QR decomposition to $n \times p$ and $p \times p$ Gaussian random matrices, respectively. Then, with the given singular values $\Sigmab = \diag(\sigma_1(\Ab),\cdots,\sigma_p(\Ab))$, we form a matrix as $\Ab = \Ub \Sigmab \Vb^\top$. 

\begin{figure}
        \centering
        \resizebox{0.99\textwidth}{!}{\input{PlotSources/upper_bounds}}
    \caption{Sketching via either random unitary or Gaussian embeddings considerably reduce the condition number of $\Ab\Gammab$ compared to $\Ab$. In agreement with \Cref{prop:RMTBound}, $\kappa(\Ab\Gammab)$ decreases as $m$ decreases and lies close to $\sigma_{1}(\Ab)/\sigma_{m}(\Ab)$. The blue lines show the averaged condition numbers $\kappa(\Ab\Gammab)$ over $10$ $\iid$ drawn $\Gammab$, with the shading indicating the maxima and minima among the $10$ trials.}
    \label{fig:cond}
\end{figure}

\begin{figure}
    \centering
        \resizebox{0.99\textwidth}{!}{\input{PlotSources/rho_estimates}}
    \caption{The random variables $1/\sigma_{\min}(\Gammab_1)^{-1}$ for $\Gammab \sim \mathrm{Haar}(\mathrm{Stiefel}(p,m))$ and $\sigma_{\max}(\Gammab)/\sigma_{\min}(\Gammab_1)$ for $\Gammab \sim \Ncal(0,1/m)^{p \times m}$ in \Cref{prop:RMTBound} and \Cref{rm:GaussianVsUnitary} that characterize the randomness of $\Gammab$ in the upper bounds concentrate tightly around their respective asymptotic limits. Here, the blue lines show averages over $10$ $\iid$ drawn $\Gammab$, with the shading indicating the maxima and minima among the $10$ trials.}
    \label{fig:rho_concentrate}
\end{figure}

In \Cref{fig:cond}, we plot the averaged condition numbers $\kappa(\Ab\Gammab)$ over $10$ $\iid$ drawn Gaussian and unitary embeddings for the embedding dimensions $m \in \{100, \dots, 500\}$ and compare to the ratio of singular values $\sigma_1(\Ab)/\sigma_m(\Ab)$ at the respective $m$. 
\Cref{prop:RMTBound} also shows the asymptotic upper bounds in \eqref{eq:cond_asymp_haar} and \eqref{eq:cond_asymp_gaussian} for unitary and Gaussian embeddings, respectively, at two different aspect ratios $\nu/\gamma = 1.2, 2.0$. 
We observe that $\kappa(\Ab\Gammab)$ is slightly larger than but close to $\sigma_1(\Ab)/\sigma_m(\Ab)$, which agrees with \Cref{prop:RMTBound}. 
For the same $\gamma = m/p$, random unitary embeddings tend to enjoy smaller $\kappa(\Ab\Gammab)$ compared to Gaussian embeddings. 
Meanwhile, $\kappa(\Ab\Gammab)$ is highly concentrated in practice with both random unitary and Gaussian embeddings.

\Cref{fig:rho_concentrate} further investigates the concentration of random variables $1/\sigma_{\text{min}}(\Gammab_1)$ for $\Gammab \sim \mathrm{Haar}(\mathrm{Stiefel}(p,m))$ and  $\sigma_{\text{max}}(\Gammab)/\sigma_{\text{min}}(\Gammab_1)$ for $\Gammab \sim \Ncal(0,1/m)^{p \times m}$ that characterize the randomness of $\Gammab$ in the upper bounds in \Cref{prop:RMTBound}; see the proof of \Cref{prop:CondBound} for a definition of $\Gammab_1$ with respect to $\Ab$ and $\Gammab$.  
We observe that both concentrate tightly around their asymptotic limits $\alpha$ and $\varrho$, respectively, while the random unitary embeddings tend to lead to stronger concentration.
Furthermore, the ratio $\sigma_{\text{max}}(\Gammab)/\sigma_{\text{min}}(\Gammab_1)$ tends to be slightly larger for Gaussian embeddings than $1/\sigma_{\text{min}}(\Gammab_1)$ for unitary embeddings at the same ratios $\gamma,\nu$, echoing the message in \eqref{eq:haar_lowerrho_gauss} and Remark~\ref{rm:GaussianVsUnitary}.

\subsection{Bias and variance of randomized least-squares increments}\label{sec:coeff_bias_var_powlaw}

For the sake of demonstration, we consider a specific scenario where $\Jb(\thetab) = \Ab \in \R^{n \times p}$ has a power law singular value decay, $\sigma_i(\Ab) = i^{-\omega}$, for some exponent $\omega > 0$ and set $n=1200$ and  $p=1000$. For $\Gammab \sim \operatorname{Haar}(\operatorname{Stiefel}(p, m))$ and $\fb \in \R^{n}$, we then consider 
\[
\nbr{(\E[\Gammab (\Ab \Gammab)^\dagger] - \Ab^\dagger) \fb}_2^2 \le \Cbias \nbr{\Ab^\dagger \fb}_2^2\,, \E[\|(\Gammab (\Ab \Gammab)^\dagger - \E[\Gammab (\Ab \Gammab)^\dagger]) \fb\|_2^2] \le \Cvar \nbr{\Ab^\dagger \fb}_2^2\,.
\]
We numerically estimate $\Cbias$ and $\Cvar$ by 
approximating the expectation over $\Gammab$ via the average over $100$ evaluations with $\iid$ random unitary embeddings $\Gammab$ and reporting the maximum ratio $\nbr{(\E[\Gammab (\Ab \Gammab)^\dagger] - \Ab^\dagger) \fb}_2^2/\nbr{\Ab^\dagger \fb}_2^2$ for an estimate of $C_b$ and $\E[\|(\Gammab (\Ab \Gammab)^\dagger - \E[\Gammab (\Ab \Gammab)^\dagger]) \fb\|_2^2]/\nbr{\Ab^\dagger \fb}_2^2$ for an estimate of $C_v$ over $4n$ independently drawn right-hand sides $\fb \sim \Ncal(\b0_n, \Ib_n)$.

The scaling of the maximum ratios estimating $\Cbias$ and $\Cvar$, respectively, with respect to $m$ is visualized in \Cref{fig:bias_var_powlaw}.
\begin{figure}
    \centering
    \resizebox{0.99\textwidth}{!}{\input{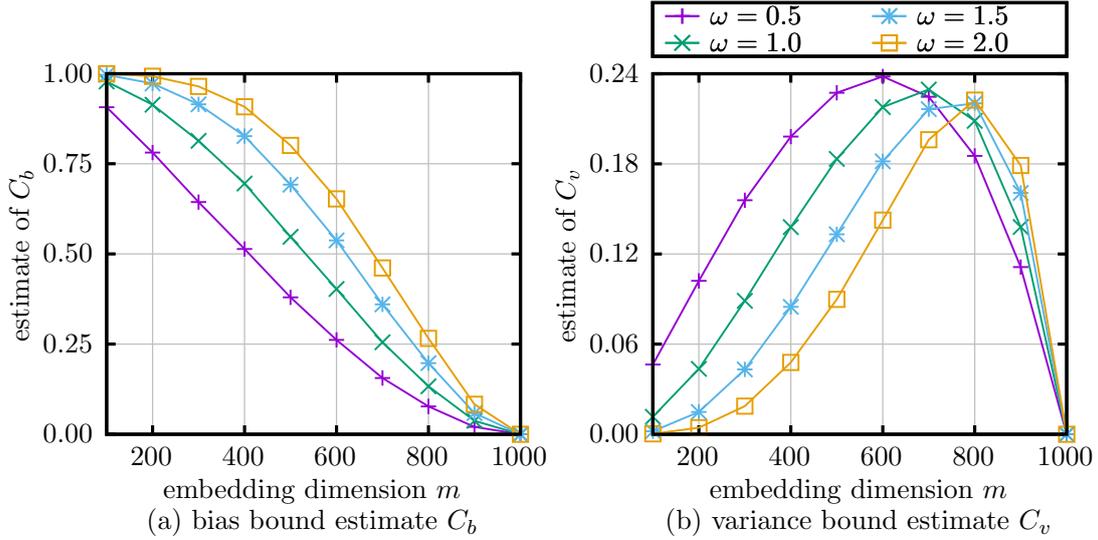}}
    \caption{Scaling of the estimates of $\Cbias$ and $\Cvar$ for random unitary embeddings $\Gammab$ of dimension $p \times m$ with respect to the sketching dimension $m$, where singular values of $\Ab \in \R^{n \times p}$ admits a power law decay, $\sigma_i(\Ab) = i^{-\omega}$ for all $i \in [p]$.}\label{fig:bias_var_powlaw}
\end{figure}
We observe that the bias characterized by the estimate of $\Cbias$ decreases monotonically as $m$ increases; whereas the estimate of $\Cvar$ corresponding to the variance first increases and then decreases as $m$ grows.
Intuitively, when $m$ is much smaller than $p$, the mean $\E[\Gammab (\Ab \Gammab)^\dagger \fb]$ contains little information. Therefore, the bias is large with $\Cbias$ close to $1$, and the variance $\Cvar$ is close to zero due to the lack of information. 
Toward the other extreme case where $m$ is close to $p$, $\E[\Gammab (\Ab \Gammab)^\dagger \fb]$ provides a good approximation of $\Ab^\dagger \fb$, which pushes both $\Cbias$ and $\Cvar$ to zero.
In addition, $\Ab$ with the faster singular value decay (larger $\omega$) tends to require a larger sketching dimension $m$ to achieve a desired decrease in $\Cbias$.

\subsection{Double-well quantum dynamics}
We demonstrate the proposed randomized time stepping on a double-well quantum dynamics problem described by the Schr\"odinger equation. 

\begin{figure}
    \centering
    \resizebox{\textwidth}{!}{\input{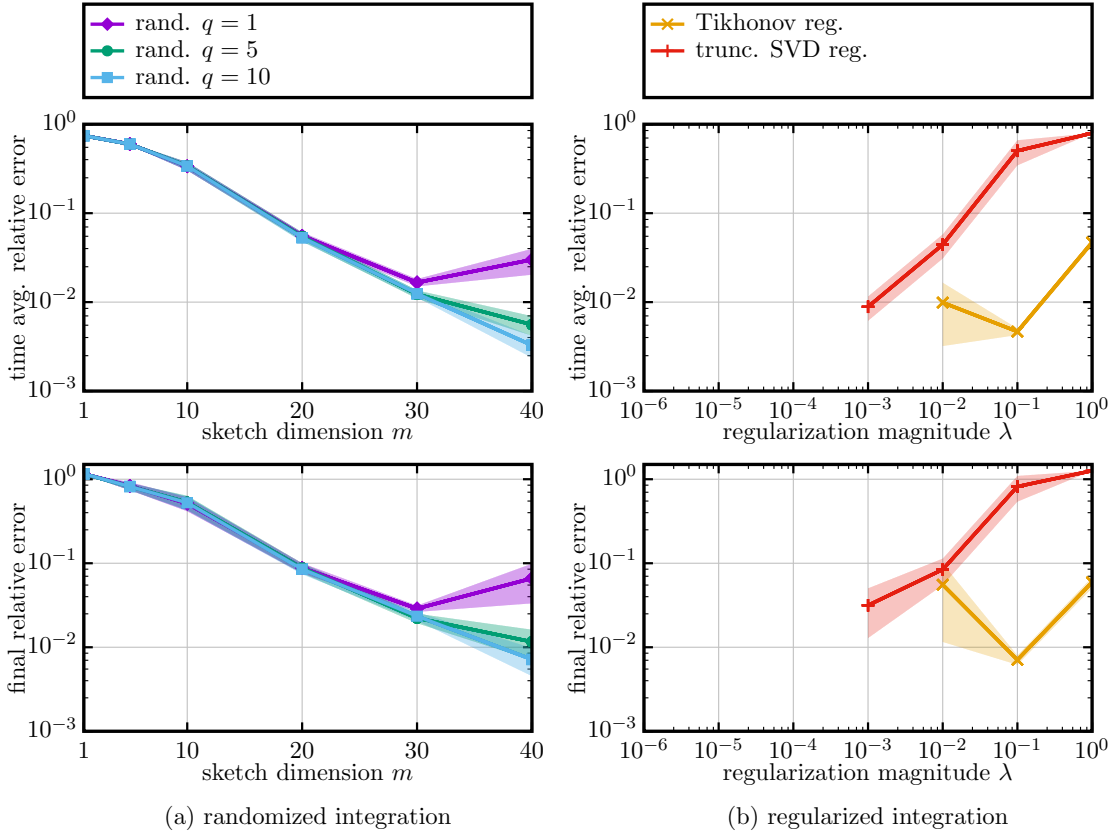}}
    \caption{Double-well quantum dynamics: Randomized time stepping with sketched least-squares problem lead to comparable errors as regularization in this example. This is in agreement with the results presented in Section~\ref{sec:RTSInc} that randomized time stepping with sketching has a regularization effect on Neural Galerkin schemes.}
    \label{fig:schroedinger_collection_rel_err}
\end{figure}

\begin{figure}
    \centering
    \resizebox{\textwidth}{!}{\input{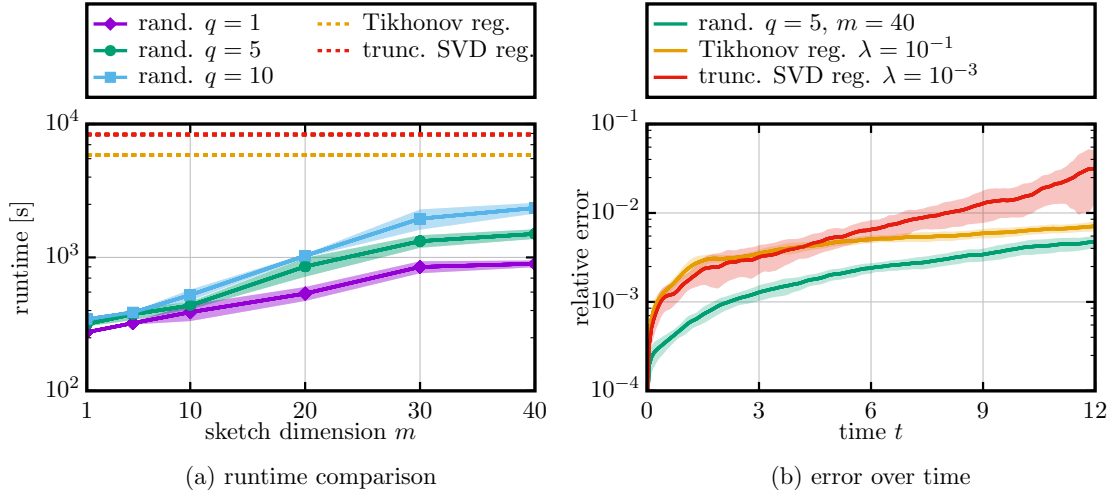}}
    \caption{Double-well quantum dynamics: Randomized time stepping can be faster than regularization because sketching the least-squares problem means that only $m \ll p$ out of all $p$ parameters of the neural network are updated at each time step.}
    \label{fig:schroedinger_collection_runtime}
\end{figure}

\begin{figure}
    \resizebox{\textwidth}{!}{\input{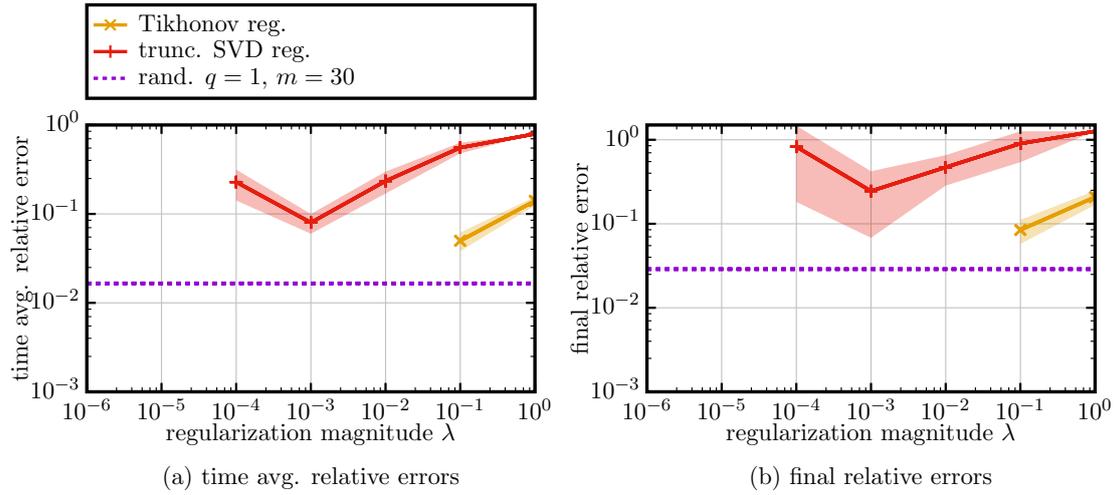}}
\caption{Double-well quantum dynamics: We fix the sketching dimension $m = 30$ so that randomized time stepping updates only $m = 30 \ll p$ out of all $p = 1363$ parameters of the network. We compare to regularization on a smaller network of size $p^{\prime} = 39$ so that about the same number of parameters are updated at each time step. Randomized time stepping achieves up to one order of magnitude lower errors here.}
\label{fig:schroedinger_relerr_smallnetwork}
\end{figure}

\subsubsection{Setup}
Consider the one-dimensional Schr\"odinger equation, 
   \begin{align} 
   \label{eq:schroed}
    \mathrm{i}\dot \psi(t, x) = H\psi(t, x),
    \end{align}
with wave function $\psi$, operator $H=-\frac{1}{2}\partial_{x}^2+\alpha_2 x+\alpha_4 x^4$, and initial condition
$\psi_0(x)=\pi^{-1/4}\mathrm{e}^{-(x+2)^2/2}$. Following~\cite{feischl2024regularized}, we set $\alpha_2 = -0.125$ and $\alpha_4 = 0.125^2$ to create a double-well potential. 
We impose periodic boundary conditions in the large spatial interval $[-6,6)$, which  mimics solving over the real line because the wave function decays fast enough. The time interval is $[0,12]$. To obtain a benchmark solution, we use a spectral method of order 500 in space. For integration in time, we use fourth-order Runge-Kutta with adaptive time stepping guided by a relative and absolute error tolerance of $10^{-5}$. 
We compare the benchmark solution to solutions obtained with  Neural Galerkin \cite{bruna2024neural,BERMAN2024389} when using various regularization schemes as well as the randomized time stepping. In all experiments, we parameterize $\psi$ using a deep neural network, which takes as input the spatial coordinate and outputs the real part and imaginary part of $\psi$ at that spatial coordinate. The neural network parameters are evolved in time using an explicit Euler scheme with fixed time-step size $\delta t = 10^{-3}$. We take a total of $K = 12\times 10^{3}$ time steps. 
We use feedforward neural networks. The input layer is always a  periodic embedding layer as in~\cite{BERMAN2024389}. As nonlinear activation function, we use the swish-function in each layer, which is defined as $\sigma(x)=x/(1+\exp(-x))$. If not otherwise stated, then there are four hidden layers with 20 nodes each, so that the total number of components of $\bftheta(t)$ is $p = 1363$. 
The weight vector $\bftheta_0$ corresponding to the initial condition is fitted via the mean squared error loss to $\psi_0$ at the 2000 equidistant points in the spatial domain. For fitting the initial condition, we use $10^5$ ADAM iterations~\cite{KingmaBa2014} with a learning rate ranging from $10^{-2}$ to $10^{-12}$ according to a cosine decay schedule, which smoothly decreases the learning rate without restarts~\cite{LoshchilovHutter2016}.

Each experiment is repeat 20 times with a different random draw of the initial weights of the network parametrization. The initial weights are drawn i.i.d.~from a standard normal random variable.  
Because several experiments will lead to unstable behavior, we use the best ten out of these 20 replicates to be able to report meaningful statistics. We will also provide the number of unstable replicates that numerically lead to large errors or not-a-number.
We compare a Neural Galerkin solution $\psi_{\text{ng}}$ to the reference solution $\psi_{\text{ref}}$ at the 500 equidistant test points $x_1^{\text{test}}, \dots, x_{500}^{\text{test}}$ at time $t$ via the relative error 
\begin{align}
\newcommand{\refsol}{\psi_{\text{ref}}}
\newcommand{\ngsol}{\psi_{\text{ng}}}
\label{eq:rel_err_t}
    \mathcal{E}^{\text{test}}(t) = \frac{\sum_{i = 1}^{500}\|\refsol(t, x_i^{\text{test}})-\ngsol(t, x_i^{\text{test}})\|_2}{\sum_{i = 1}^{500}\|\refsol(t, x_i^{\text{test}})\|_2}\,.
\end{align}
We also report the mean relative error of \eqref{eq:rel_err_t} over the 200 equidistant test points in the time domain.

\subsubsection{Results} Figure~\ref{fig:schroedinger_collection_rel_err} plots the relative error \eqref{eq:rel_err_t} at end time as well as the mean relative error over the time interval for Neural Galerkin solutions with randomized time stepping, Tikhonov regularization, and regularization via the truncated SVD. We show results for sketching dimensions $m \in \{1, 5, 10, 20, 30, 40\}$ and regularization parameter $\lambda \in [10^{-6}, 10^0]$. The shaded area shows the standard deviation over the ten best replicates in terms of error and the solid line shows the mean over the ten best replicates. The results show that randomized time stepping with sketching has a regularization effect on the Neural Galerkin solution and achieves comparable accuracy as regularization over a wide range of regularization parameters. The number of local samples $q$ has to be increased to $q > 1$ for $m > 30$ to counter the variance that is introduced by the larger sketching dimension $m$.

Let us now consider the runtime of Neural Galerkin schemes with randomized time stepping via sketching versus regularization; see Figure~\ref{fig:schroedinger_collection_runtime}. Because sketching means that at each time step only $m \ll p$ out of all $p$ parameters of the neural network are updated, randomized time stepping achieves lower runtimes for comparable errors than regularization. In fact, the cost complexity of one time step of Neural Galerkin schemes with randomized time stepping via sketching scales $\mathcal{O}(n m^2)$ versus the costs scaling as $\mathcal{O}(n p^2)$ for regularization. In this particular example, this leads to about one order of magnitude runtime speed up with $m = 40$ and $q = 5$ to achieve a comparable error as regularization. We can go a step further and correspondingly compare to the accuracy achieved with regularization when the number of neural-network parameters is comparable to the sketching dimension $m$ so that for about the same number of unknowns is solved at each time step. In Figure~\ref{fig:schroedinger_relerr_smallnetwork}, we show that regularized Neural Galerkin schemes with neural networks with $p^{\prime} = 39$ parameters lead to an error that is about one order of magnitude higher than Neural Galerkin schemes with randomized time stepping and sketching dimension $m = 30$ and $q = 1$, which means a comparable number of parameters that is updated per time step.

\subsection{Reaction-diffusion processes}
In this section, we demonstrate Neural Galerkin schemes with random time stepping on the Allen-Cahn equation to model reaction-diffusion processes. 

\begin{figure}
    \centering
    \resizebox{\textwidth}{!}{\input{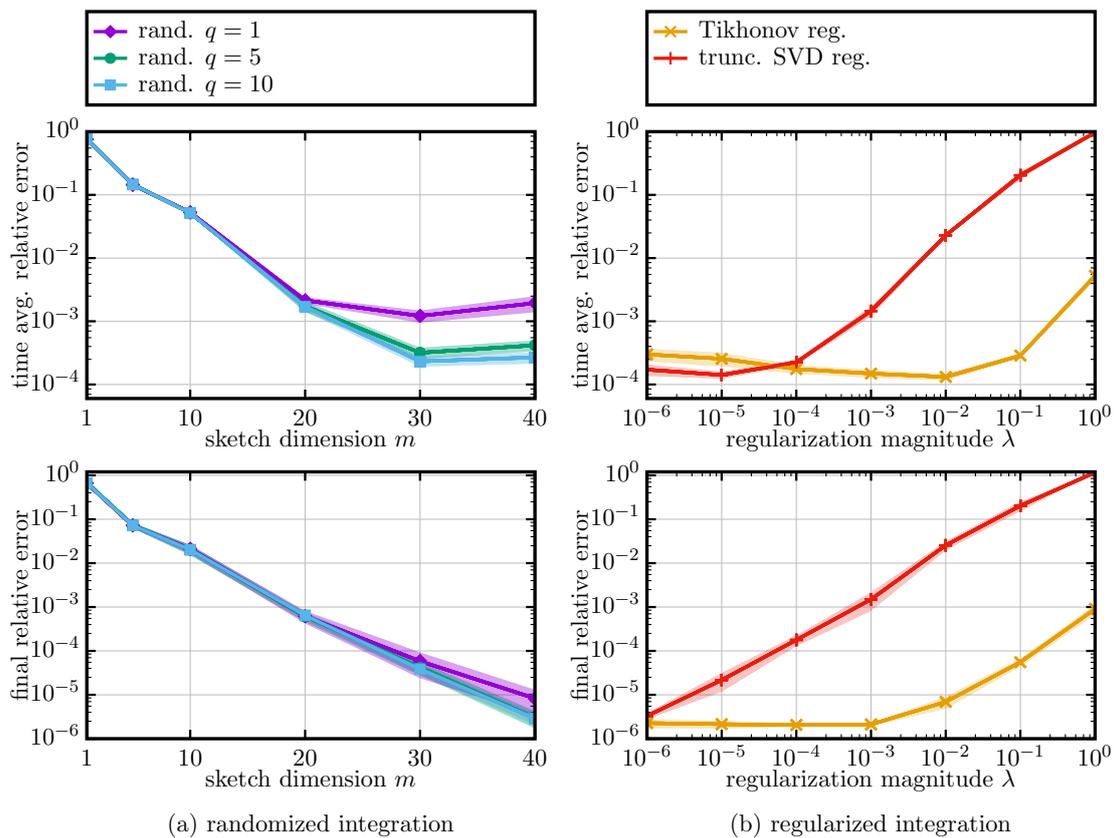}}
    \caption{Reaction-diffusion process: Randomized time stepping with sketching leads to Neural Galerkin solutions that achieve comparable accuracy as solutions obtained with Tikhonov and truncated SVD regularization. 
    }
   
    \label{fig:allen_cahn_1d_collection_rel_err}
\end{figure}

\begin{figure}
    \centering
    \resizebox{\textwidth}{!}{\input{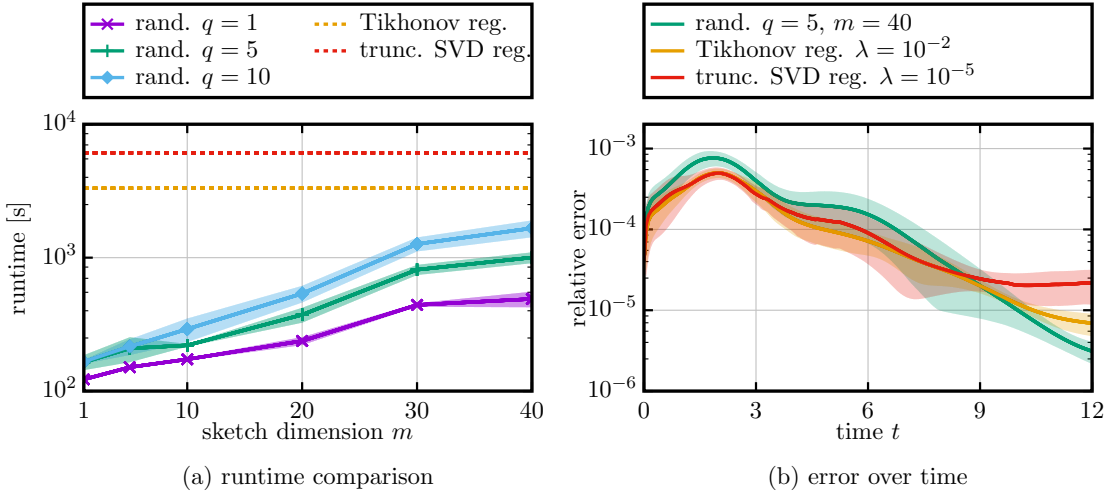}}
    \caption{Reaction-diffusion process: The runtime of the Neural Galerkin scheme with randomized time stepping via sketching is lower than with regularization because only $m \ll p$ parameters out of all $p$ network parameters are updated at each time step. At the same time, a similar accuracy is achieved over the time interval. }
    \label{fig:allen_cahn_1d_collection_runtime}
\end{figure}

\begin{figure}
    \resizebox{\textwidth}{!}{\input{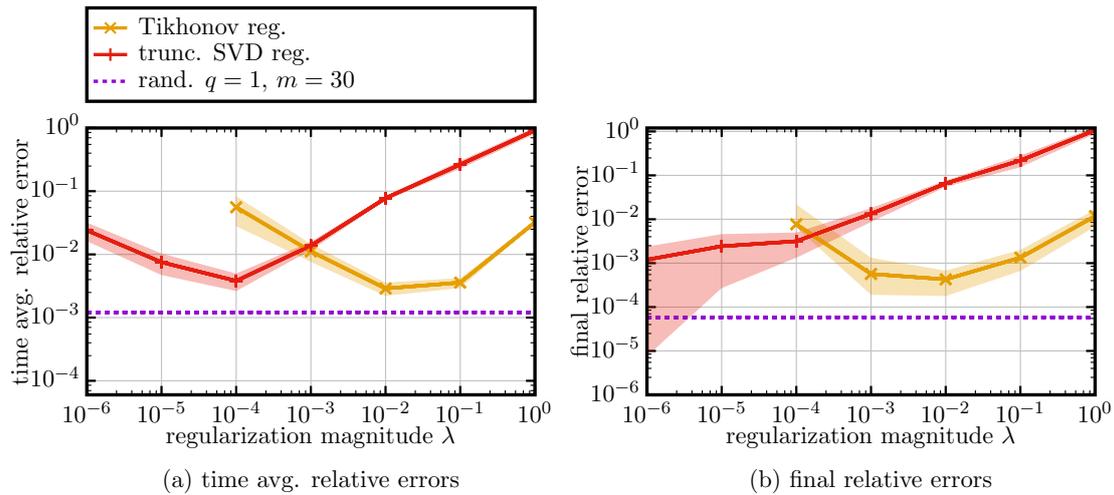}}
\caption{Reaction-diffusion processes: Comparing randomized time stepping with sketching updates only $m = 30$ parameters out of all $p$ parameters of the network at each time step but still achieves one order of magnitude lower errors than the Neural Galerkin solutions obtained with regularization on a smaller network of $p^{\prime} =  34$ parameters that updates all $p^{\prime}$ parameters at each time step.}
\label{fig:allen_cahn_1d_collection_runtime_smallnetwork}
\end{figure}

\subsubsection{Setup}
We consider the Allen-Cahn equation of the form 
    \begin{align}
        \label{eq:allen_cahn_1d}
        \partial_t u(t, x)=\varepsilon \partial_x^2u(t, x)-a(t, x)(u(t, x)-u(t, x)^3),
    \end{align}
which describes reaction-diffusion processes in the one-dimensional spatial domain $[0,1)$ with periodic boundary conditions. We set $\varepsilon=5\times 10^{-4}$ and the coefficient function to $a(t, x)= 1.05 + t  \sin(2 \pi x)$, which is similar to the setup in~\cite{bruna2024neural,BERMAN2024389}. 
The initial condition is 
$u_0(x) = \phi_g(x, 0.03)-\phi_g(x, 0.7)$, where
$\phi_g(x, b) =\exp(-20 \sin(\pi (x - b)))^2)$. 
A benchmark solution is computed with a spectral method of order 500 and fourth-order Runge-Kutta with adaptive time stepping. The time-step adaptation is guided by a relative and absolute error tolerance of $10^{-5}$. For applying Neural Galerkin, we parameterize the solution function $u$ using a deep neural network with the same architecture as in the previous example; the only difference is that we only have one output quantity, as the solution $u$ of \eqref{eq:allen_cahn_1d} has no imaginary part. This means that $\bftheta(t)$ has $p = 1342$ components.

\subsubsection{Results}
Figure~\ref{fig:allen_cahn_1d_collection_rel_err} shows in the first row the mean relative error \eqref{eq:rel_err_t} over the time interval and in the second row the relative error at end time for Neural Galerkin solutions obtained with randomized time stepping, Tikhonov regularization, and regularization via the truncated SVD. Over a wide range of sketching dimensions $m$ and regularization parameters $\lambda$, Neural Galerkin with randomized time stepping achieves comparable accuracy as with regularization. Because only $m \ll p$ parameters are updated at each time step, randomized time stepping leads to a speed up of almost one order of magnitude compared to regularization that updates all $p$ parameters; see Figure~\ref{fig:allen_cahn_1d_collection_runtime}. In Figure~\ref{fig:allen_cahn_1d_collection_runtime_smallnetwork}, we show that using a smaller network with $p^{\prime} = 34$ so that about the same number of parameters are updated at each time step as with randomized time stepping and sketching dimension $m = 30$ and $q = 1$ leads to higher errors in this example. Overall, the results are in close agreement with the results obtained for the double-well quantum dynamics given by the Schr\"odinger equation in the previous section.

\section{Conclusions}\label{sec:Conc}
Random projections via sketching help to improve the conditioning of time-dependent least-squares problems. At the same time, at each time step of the randomized time stepping, only a low-dimensional projection of the parameter vector is updated, which means that the number of unknowns in the least-squares problems scales with the sketching dimension. Thus, the costs of solving the least-squares problem at each time step is lower than when applying regularization that updates the parameter vector in all directions, e.g., in all components. 
The observation that updating only a low-dimensional projection of the parameter is sufficient indicates that locally in time the problems are low dimensional, even though globally in time a rich, nonlinear parametrization can be necessary. This insight is further supported by dynamic low-rank methods with adaptive ranks \cite{Ceruti2022a,hesthaven_pagliantini_rozza_2022,doi:10.1137/22M1534948} as well as the work \cite{berman2024randomized} that updates only a random subset of all components of the parameter vectors. Furthermore, it shows that nonlinear parametrizations that depend on time can be beneficial in this sense compared to approaches that aim to approximate the solution fields globally in time and space, which cannot leverage such local-in-time properties. 
The randomized time stepping scheme proposed in this work explicitly takes the local-in-time low dimensionality into account, which is different from standard regularization schemes that update parameters in all directions at every time step.

\section*{Acknowledgments}
This work was partially supported by the National Science Foundation grant 2046521 and by the Office of Naval Research award N00014-22-1-2728. 

\bibliographystyle{my-sn-mathphys-num}
\bibliography{ref}


\begin{thebibliography}{78}
\ifx \bisbn   \undefined \def \bisbn  #1{ISBN #1}\fi
\ifx \binits  \undefined \def \binits#1{#1}\fi
\ifx \bauthor  \undefined \def \bauthor#1{#1}\fi
\ifx \batitle  \undefined \def \batitle#1{#1}\fi
\ifx \bjtitle  \undefined \def \bjtitle#1{#1}\fi
\ifx \bvolume  \undefined \def \bvolume#1{\textbf{#1}}\fi
\ifx \byear  \undefined \def \byear#1{#1}\fi
\ifx \bissue  \undefined \def \bissue#1{#1}\fi
\ifx \bfpage  \undefined \def \bfpage#1{#1}\fi
\ifx \blpage  \undefined \def \blpage #1{#1}\fi
\ifx \burl  \undefined \def \burl#1{\textsf{#1}}\fi
\ifx \doiurl  \undefined \def \doiurl#1{\url{https://doi.org/#1}}\fi
\ifx \betal  \undefined \def \betal{\textit{et al.}}\fi
\ifx \binstitute  \undefined \def \binstitute#1{#1}\fi
\ifx \binstitutionaled  \undefined \def \binstitutionaled#1{#1}\fi
\ifx \bctitle  \undefined \def \bctitle#1{#1}\fi
\ifx \beditor  \undefined \def \beditor#1{#1}\fi
\ifx \bpublisher  \undefined \def \bpublisher#1{#1}\fi
\ifx \bbtitle  \undefined \def \bbtitle#1{#1}\fi
\ifx \bedition  \undefined \def \bedition#1{#1}\fi
\ifx \bseriesno  \undefined \def \bseriesno#1{#1}\fi
\ifx \blocation  \undefined \def \blocation#1{#1}\fi
\ifx \bsertitle  \undefined \def \bsertitle#1{#1}\fi
\ifx \bsnm \undefined \def \bsnm#1{#1}\fi
\ifx \bsuffix \undefined \def \bsuffix#1{#1}\fi
\ifx \bparticle \undefined \def \bparticle#1{#1}\fi
\ifx \barticle \undefined \def \barticle#1{#1}\fi
\ifx \bconfdate \undefined \def \bconfdate #1{#1}\fi
\ifx \botherref \undefined \def \botherref #1{#1}\fi
\ifx \url \undefined \def \url#1{\textsf{#1}}\fi
\ifx \bchapter \undefined \def \bchapter#1{#1}\fi
\ifx \bbook \undefined \def \bbook#1{#1}\fi
\ifx \bcomment \undefined \def \bcomment#1{#1}\fi
\ifx \oauthor \undefined \def \oauthor#1{#1}\fi
\ifx \citeauthoryear \undefined \def \citeauthoryear#1{#1}\fi
\ifx \endbibitem  \undefined \def \endbibitem {}\fi
\ifx \bconflocation  \undefined \def \bconflocation#1{#1}\fi
\ifx \arxivurl  \undefined \def \arxivurl#1{\textsf{#1}}\fi
\csname PreBibitemsHook\endcsname

\bibitem[\protect\citeauthoryear{Ern and Guermond}{2004}]{FEMTheory}
\begin{bbook}
\bauthor{\bsnm{Ern}, \binits{A.}},
\bauthor{\bsnm{Guermond}, \binits{J.-L.}}:
\bbtitle{Theory and Practice of Finite Elements}.
\bpublisher{Springer},
\blocation{New York}
(\byear{2004})
\end{bbook}
\endbibitem

\bibitem[\protect\citeauthoryear{Goodfellow
  et~al.}{2016}]{Goodfellow-et-al-2016}
\begin{bbook}
\bauthor{\bsnm{Goodfellow}, \binits{I.}},
\bauthor{\bsnm{Bengio}, \binits{Y.}},
\bauthor{\bsnm{Courville}, \binits{A.}}:
\bbtitle{Deep Learning}.
\bpublisher{MIT Press},
\blocation{Cambridge, MA}
(\byear{2016}).
\bcomment{http://www.deeplearningbook.org}
\end{bbook}
\endbibitem

\bibitem[\protect\citeauthoryear{Cichocki et~al.}{2016}]{MAL-059}
\begin{barticle}
\bauthor{\bsnm{Cichocki}, \binits{A.}},
\bauthor{\bsnm{Lee}, \binits{N.}},
\bauthor{\bsnm{Oseledets}, \binits{I.}},
\bauthor{\bsnm{Phan}, \binits{A.-H.}},
\bauthor{\bsnm{Zhao}, \binits{Q.}},
\bauthor{\bsnm{Mandic}, \binits{D.P.}}:
\batitle{Tensor networks for dimensionality reduction and large-scale
  optimization: Part 1 low-rank tensor decompositions}.
\bjtitle{Foundations and Trends® in Machine Learning}
\bvolume{9}(\bissue{4-5}),
\bfpage{249}--\blpage{429}
(\byear{2016})
\end{barticle}
\endbibitem

\bibitem[\protect\citeauthoryear{Beck et~al.}{2000}]{BECK20001}
\begin{barticle}
\bauthor{\bsnm{Beck}, \binits{M.H.}},
\bauthor{\bsnm{Jäckle}, \binits{A.}},
\bauthor{\bsnm{Worth}, \binits{G.A.}},
\bauthor{\bsnm{Meyer}, \binits{H.-D.}}:
\batitle{The multiconfiguration time-dependent hartree ({MCTDH}) method: a
  highly efficient algorithm for propagating wavepackets}.
\bjtitle{Physics Reports}
\bvolume{324}(\bissue{1}),
\bfpage{1}--\blpage{105}
(\byear{2000})
\end{barticle}
\endbibitem

\bibitem[\protect\citeauthoryear{Lubich}{2008}]{lubich2008quantum}
\begin{bbook}
\bauthor{\bsnm{Lubich}, \binits{C.}}:
\bbtitle{From Quantum to Classical Molecular Dynamics: Reduced Models and
  Numerical Analysis}
vol. \bseriesno{12}.
\bpublisher{European Mathematical Society},
\blocation{Z\"urich, Switzerland}
(\byear{2008})
\end{bbook}
\endbibitem

\bibitem[\protect\citeauthoryear{Han
  et~al.}{2018}]{doi:10.1073/pnas.1718942115}
\begin{barticle}
\bauthor{\bsnm{Han}, \binits{J.}},
\bauthor{\bsnm{Jentzen}, \binits{A.}},
\bauthor{\bsnm{E}, \binits{W.}}:
\batitle{Solving high-dimensional partial differential equations using deep
  learning}.
\bjtitle{Proceedings of the National Academy of Sciences}
\bvolume{115}(\bissue{34}),
\bfpage{8505}--\blpage{8510}
(\byear{2018})
\end{barticle}
\endbibitem

\bibitem[\protect\citeauthoryear{Peherstorfer}{2022}]{P22AMS}
\begin{barticle}
\bauthor{\bsnm{Peherstorfer}, \binits{B.}}:
\batitle{Breaking the {Kolmogorov} barrier with nonlinear model reduction}.
\bjtitle{Notices of the American Mathematical Society}
\bvolume{69},
\bfpage{725}--\blpage{733}
(\byear{2022})
\end{barticle}
\endbibitem

\bibitem[\protect\citeauthoryear{Koch and Lubich}{2007}]{doi:10.1137/050639703}
\begin{barticle}
\bauthor{\bsnm{Koch}, \binits{O.}},
\bauthor{\bsnm{Lubich}, \binits{C.}}:
\batitle{Dynamical low‐rank approximation}.
\bjtitle{SIAM Journal on Matrix Analysis and Applications}
\bvolume{29}(\bissue{2}),
\bfpage{434}--\blpage{454}
(\byear{2007})
\end{barticle}
\endbibitem

\bibitem[\protect\citeauthoryear{Einkemmer et~al.}{2021}]{EINKEMMER2021110353}
\begin{barticle}
\bauthor{\bsnm{Einkemmer}, \binits{L.}},
\bauthor{\bsnm{Hu}, \binits{J.}},
\bauthor{\bsnm{Wang}, \binits{Y.}}:
\batitle{An asymptotic-preserving dynamical low-rank method for the multi-scale
  multi-dimensional linear transport equation}.
\bjtitle{Journal of Computational Physics}
\bvolume{439},
\bfpage{110353}
(\byear{2021})
\end{barticle}
\endbibitem

\bibitem[\protect\citeauthoryear{Einkemmer
  et~al.}{2024}]{doi:10.1137/23M1547603}
\begin{barticle}
\bauthor{\bsnm{Einkemmer}, \binits{L.}},
\bauthor{\bsnm{Hu}, \binits{J.}},
\bauthor{\bsnm{Kusch}, \binits{J.}}:
\batitle{Asymptotic-preserving and energy stable dynamical low-rank
  approximation}.
\bjtitle{SIAM Journal on Numerical Analysis}
\bvolume{62}(\bissue{1}),
\bfpage{73}--\blpage{92}
(\byear{2024})
\end{barticle}
\endbibitem

\bibitem[\protect\citeauthoryear{Peherstorfer and
  Willcox}{2015}]{Peherstorfer15aDEIM}
\begin{barticle}
\bauthor{\bsnm{Peherstorfer}, \binits{B.}},
\bauthor{\bsnm{Willcox}, \binits{K.}}:
\batitle{Online adaptive model reduction for nonlinear systems via low-rank
  updates}.
\bjtitle{SIAM Journal on Scientific Computing}
\bvolume{37}(\bissue{4}),
\bfpage{2123}--\blpage{2150}
(\byear{2015})
\end{barticle}
\endbibitem

\bibitem[\protect\citeauthoryear{Hesthaven et~al.}{2022a}]{CPagAdapt}
\begin{barticle}
\bauthor{\bsnm{Hesthaven}, \binits{J.S.}},
\bauthor{\bsnm{Pagliantini}, \binits{C.}},
\bauthor{\bsnm{Ripamonti}, \binits{N.}}:
\batitle{Rank-adaptive structure-preserving model order reduction of
  hamiltonian systems}.
\bjtitle{ESAIM: M2AN}
\bvolume{56}(\bissue{2}),
\bfpage{617}--\blpage{650}
(\byear{2022})
\end{barticle}
\endbibitem

\bibitem[\protect\citeauthoryear{Hesthaven
  et~al.}{2022b}]{hesthaven_pagliantini_rozza_2022}
\begin{barticle}
\bauthor{\bsnm{Hesthaven}, \binits{J.S.}},
\bauthor{\bsnm{Pagliantini}, \binits{C.}},
\bauthor{\bsnm{Rozza}, \binits{G.}}:
\batitle{Reduced basis methods for time-dependent problems}.
\bjtitle{Acta Numerica}
\bvolume{31},
\bfpage{265}--\blpage{345}
(\byear{2022})
\end{barticle}
\endbibitem

\bibitem[\protect\citeauthoryear{Koch and Lubich}{2010}]{doi:10.1137/09076578X}
\begin{barticle}
\bauthor{\bsnm{Koch}, \binits{O.}},
\bauthor{\bsnm{Lubich}, \binits{C.}}:
\batitle{Dynamical tensor approximation}.
\bjtitle{SIAM Journal on Matrix Analysis and Applications}
\bvolume{31}(\bissue{5}),
\bfpage{2360}--\blpage{2375}
(\byear{2010})
\end{barticle}
\endbibitem

\bibitem[\protect\citeauthoryear{Arnold and Jahnke}{2014}]{Arnold2014}
\begin{barticle}
\bauthor{\bsnm{Arnold}, \binits{A.}},
\bauthor{\bsnm{Jahnke}, \binits{T.}}:
\batitle{On the approximation of high-dimensional differential equations in the
  hierarchical {Tucker} format}.
\bjtitle{BIT Numerical Mathematics}
\bvolume{54}(\bissue{2}),
\bfpage{305}--\blpage{341}
(\byear{2014})
\end{barticle}
\endbibitem

\bibitem[\protect\citeauthoryear{Sapsis and Lermusiaux}{2009}]{SAPSIS20092347}
\begin{barticle}
\bauthor{\bsnm{Sapsis}, \binits{T.P.}},
\bauthor{\bsnm{Lermusiaux}, \binits{P.F.J.}}:
\batitle{Dynamically orthogonal field equations for continuous stochastic
  dynamical systems}.
\bjtitle{Physica D: Nonlinear Phenomena}
\bvolume{238}(\bissue{23}),
\bfpage{2347}--\blpage{2360}
(\byear{2009})
\end{barticle}
\endbibitem

\bibitem[\protect\citeauthoryear{Musharbash et~al.}{2015}]{NobileDO}
\begin{barticle}
\bauthor{\bsnm{Musharbash}, \binits{E.}},
\bauthor{\bsnm{Nobile}, \binits{F.}},
\bauthor{\bsnm{Zhou}, \binits{T.}}:
\batitle{Error analysis of the dynamically orthogonal approximation of time
  dependent random {PDEs}}.
\bjtitle{SIAM Journal on Scientific Computing}
\bvolume{37}(\bissue{2}),
\bfpage{776}--\blpage{810}
(\byear{2015})
\end{barticle}
\endbibitem

\bibitem[\protect\citeauthoryear{Musharbash and
  Nobile}{2018}]{MUSHARBASH2018135}
\begin{barticle}
\bauthor{\bsnm{Musharbash}, \binits{E.}},
\bauthor{\bsnm{Nobile}, \binits{F.}}:
\batitle{Dual dynamically orthogonal approximation of incompressible {Navier
  Stokes} equations with random boundary conditions}.
\bjtitle{Journal of Computational Physics}
\bvolume{354},
\bfpage{135}--\blpage{162}
(\byear{2018})
\end{barticle}
\endbibitem

\bibitem[\protect\citeauthoryear{Feppon and
  Lermusiaux}{2018}]{doi:10.1137/16M1109394}
\begin{barticle}
\bauthor{\bsnm{Feppon}, \binits{F.}},
\bauthor{\bsnm{Lermusiaux}, \binits{P.F.J.}}:
\batitle{Dynamically orthogonal numerical schemes for efficient stochastic
  advection and {Lagrangian} transport}.
\bjtitle{SIAM Review}
\bvolume{60}(\bissue{3}),
\bfpage{595}--\blpage{625}
(\byear{2018})
\end{barticle}
\endbibitem

\bibitem[\protect\citeauthoryear{Charous and
  Lermusiaux}{2023}]{doi:10.1137/21M1431229}
\begin{barticle}
\bauthor{\bsnm{Charous}, \binits{A.}},
\bauthor{\bsnm{Lermusiaux}, \binits{P.F.J.}}:
\batitle{Dynamically orthogonal {Runge}–{Kutta} schemes with perturbative
  retractions for the dynamical low-rank approximation}.
\bjtitle{SIAM Journal on Scientific Computing}
\bvolume{45}(\bissue{2}),
\bfpage{872}--\blpage{897}
(\byear{2023})
\end{barticle}
\endbibitem

\bibitem[\protect\citeauthoryear{Charous and
  Lermusiaux}{2024}]{doi:10.1137/22M1534948}
\begin{barticle}
\bauthor{\bsnm{Charous}, \binits{A.}},
\bauthor{\bsnm{Lermusiaux}, \binits{P.F.J.}}:
\batitle{Stable rank-adaptive dynamically orthogonal {Runge}-{Kutta} schemes}.
\bjtitle{SIAM Journal on Scientific Computing}
\bvolume{46}(\bissue{1}),
\bfpage{529}--\blpage{560}
(\byear{2024})
\end{barticle}
\endbibitem

\bibitem[\protect\citeauthoryear{Bruna et~al.}{2024}]{bruna2024neural}
\begin{barticle}
\bauthor{\bsnm{Bruna}, \binits{J.}},
\bauthor{\bsnm{Peherstorfer}, \binits{B.}},
\bauthor{\bsnm{Vanden-Eijnden}, \binits{E.}}:
\batitle{Neural {Galerkin} schemes with active learning for high-dimensional
  evolution equations}.
\bjtitle{Journal of Computational Physics}
\bvolume{496},
\bfpage{112588}
(\byear{2024})
\end{barticle}
\endbibitem

\bibitem[\protect\citeauthoryear{Berman et~al.}{2024}]{BERMAN2024389}
\begin{bchapter}
\bauthor{\bsnm{Berman}, \binits{J.}},
\bauthor{\bsnm{Schwerdtner}, \binits{P.}},
\bauthor{\bsnm{Peherstorfer}, \binits{B.}}:
\bctitle{Chapter 8 - {Neural} {Galerkin} schemes for sequential-in-time solving
  of partial differential equations with deep networks}.
In: \beditor{\bsnm{Mishra}, \binits{S.}},
\beditor{\bsnm{Townsend}, \binits{A.}} (eds.)
\bbtitle{Numerical Analysis Meets Machine Learning}.
\bsertitle{Handbook of Numerical Analysis},
vol. \bseriesno{25},
pp. \bfpage{389}--\blpage{418}.
\bpublisher{Elsevier},
\blocation{Amsterdam, The Netherlands}
(\byear{2024})
\end{bchapter}
\endbibitem

\bibitem[\protect\citeauthoryear{Du and Zaki}{2021}]{Du_2021}
\begin{botherref}
\oauthor{\bsnm{Du}, \binits{Y.}},
\oauthor{\bsnm{Zaki}, \binits{T.A.}}:
Evolutional deep neural network.
Physical Review E
\textbf{104}(4)
(2021)
\end{botherref}
\endbibitem

\bibitem[\protect\citeauthoryear{Anderson and
  Farazmand}{2022}]{doi:10.1137/21M1415972}
\begin{barticle}
\bauthor{\bsnm{Anderson}, \binits{W.}},
\bauthor{\bsnm{Farazmand}, \binits{M.}}:
\batitle{Evolution of nonlinear reduced-order solutions for {PDEs} with
  conserved quantities}.
\bjtitle{SIAM Journal on Scientific Computing}
\bvolume{44}(\bissue{1}),
\bfpage{176}--\blpage{197}
(\byear{2022})
\end{barticle}
\endbibitem

\bibitem[\protect\citeauthoryear{{Black, Felix}
  et~al.}{2020}]{UngerTransformModes2020}
\begin{barticle}
\bauthor{\bsnm{{Black, Felix}}},
\bauthor{\bsnm{{Schulze, Philipp}}},
\bauthor{\bsnm{{Unger, Benjamin}}}:
\batitle{Projection-based model reduction with dynamically transformed modes}.
\bjtitle{ESAIM: M2AN}
\bvolume{54}(\bissue{6}),
\bfpage{2011}--\blpage{2043}
(\byear{2020})
\end{barticle}
\endbibitem

\bibitem[\protect\citeauthoryear{Finzi et~al.}{2023}]{finzi2023a}
\begin{bchapter}
\bauthor{\bsnm{Finzi}, \binits{M.A.}},
\bauthor{\bsnm{Potapczynski}, \binits{A.}},
\bauthor{\bsnm{Choptuik}, \binits{M.}},
\bauthor{\bsnm{Wilson}, \binits{A.G.}}:
\bctitle{A stable and scalable method for solving initial value {PDE}s with
  neural networks}.
In: \bbtitle{The Eleventh International Conference on Learning Representations}
(\byear{2023})
\end{bchapter}
\endbibitem

\bibitem[\protect\citeauthoryear{Kast and Hesthaven}{2024}]{KAST2024112986}
\begin{barticle}
\bauthor{\bsnm{Kast}, \binits{M.}},
\bauthor{\bsnm{Hesthaven}, \binits{J.S.}}:
\batitle{Positional embeddings for solving {PDEs} with evolutional deep neural
  networks}.
\bjtitle{Journal of Computational Physics}
\bvolume{508},
\bfpage{112986}
(\byear{2024})
\end{barticle}
\endbibitem

\bibitem[\protect\citeauthoryear{Zhang
  et~al.}{2025}]{zhang2024sequentialintimetrainingnonlinearparametrizations}
\begin{botherref}
\oauthor{\bsnm{Zhang}, \binits{H.}},
\oauthor{\bsnm{Chen}, \binits{Y.}},
\oauthor{\bsnm{Vanden-Eijnden}, \binits{E.}},
\oauthor{\bsnm{Peherstorfer}, \binits{B.}}:
Sequential-in-time training of nonlinear parametrizations for solving
  time-dependent partial differential equations.
SIAM Review
(2025).
(to appear)
\end{botherref}
\endbibitem

\bibitem[\protect\citeauthoryear{Dirac}{1930}]{dirac1930note}
\begin{barticle}
\bauthor{\bsnm{Dirac}, \binits{P.A.}}:
\batitle{Note on exchange phenomena in the thomas atom}.
\bjtitle{Mathematical proceedings of the Cambridge philosophical society}
\bvolume{26}(\bissue{3}),
\bfpage{376}--\blpage{385}
(\byear{1930}).
\bcomment{Cambridge University Press}
\end{barticle}
\endbibitem

\bibitem[\protect\citeauthoryear{Frenkel}{1934}]{frenkel1934wave}
\begin{bbook}
\bauthor{\bsnm{Frenkel}, \binits{J.}}:
\bbtitle{Wave Mechanics, Advanced General Theory}
vol. \bseriesno{436}.
\bpublisher{Oxford},
\blocation{Oxford, UK}
(\byear{1934})
\end{bbook}
\endbibitem

\bibitem[\protect\citeauthoryear{Kramer and Saraceno}{1981}]{Kramer1981}
\begin{bchapter}
\bauthor{\bsnm{Kramer}, \binits{P.}},
\bauthor{\bsnm{Saraceno}, \binits{M.}}:
\bctitle{Geometry of the time-dependent variational principle in quantum
  mechanics}.
In: \bbtitle{Lecture Notes in Physics},
vol. \bseriesno{140}.
\bpublisher{Springer},
\blocation{Berlin, Germany}
(\byear{1981})
\end{bchapter}
\endbibitem

\bibitem[\protect\citeauthoryear{Wen et~al.}{2024}]{WEN2024134129}
\begin{barticle}
\bauthor{\bsnm{Wen}, \binits{Y.}},
\bauthor{\bsnm{Vanden-Eijnden}, \binits{E.}},
\bauthor{\bsnm{Peherstorfer}, \binits{B.}}:
\batitle{Coupling parameter and particle dynamics for adaptive sampling in
  {Neural} {Galerkin} schemes}.
\bjtitle{Physica D: Nonlinear Phenomena}
\bvolume{462},
\bfpage{134129}
(\byear{2024})
\end{barticle}
\endbibitem

\bibitem[\protect\citeauthoryear{Lee and Carlberg}{2020}]{LEE2020108973}
\begin{barticle}
\bauthor{\bsnm{Lee}, \binits{K.}},
\bauthor{\bsnm{Carlberg}, \binits{K.T.}}:
\batitle{Model reduction of dynamical systems on nonlinear manifolds using deep
  convolutional autoencoders}.
\bjtitle{Journal of Computational Physics}
\bvolume{404},
\bfpage{108973}
(\byear{2020})
\end{barticle}
\endbibitem

\bibitem[\protect\citeauthoryear{Berman and
  Peherstorfer}{2024}]{berman2024colora}
\begin{bchapter}
\bauthor{\bsnm{Berman}, \binits{J.}},
\bauthor{\bsnm{Peherstorfer}, \binits{B.}}:
\bctitle{{C}o{L}o{RA}: Continuous low-rank adaptation for reduced implicit
  neural modeling of parameterized partial differential equations}.
In: \beditor{\bsnm{Salakhutdinov}, \binits{R.}},
\beditor{\bsnm{Kolter}, \binits{Z.}},
\beditor{\bsnm{Heller}, \binits{K.}},
\beditor{\bsnm{Weller}, \binits{A.}},
\beditor{\bsnm{Oliver}, \binits{N.}},
\beditor{\bsnm{Scarlett}, \binits{J.}},
\beditor{\bsnm{Berkenkamp}, \binits{F.}} (eds.)
\bbtitle{Proceedings of the 41st International Conference on Machine Learning}.
\bsertitle{Proceedings of Machine Learning Research},
vol. \bseriesno{235},
pp. \bfpage{3565}--\blpage{3583}.
\bpublisher{PMLR},
\blocation{Cambridge, MA}
(\byear{2024})
\end{bchapter}
\endbibitem

\bibitem[\protect\citeauthoryear{Weder
  et~al.}{2025}]{weder2024nonlinearmodelreductionneural}
\begin{barticle}
\bauthor{\bsnm{Weder}, \binits{P.}},
\bauthor{\bsnm{Schwerdtner}, \binits{P.}},
\bauthor{\bsnm{Peherstorfer}, \binits{B.}}:
\batitle{Nonlinear model reduction with {Neural} {Galerkin} schemes on
  quadratic manifolds}.
\bjtitle{Journal of Computational Physics}
\bvolume{539},
\bfpage{114249}
(\byear{2025})
\end{barticle}
\endbibitem

\bibitem[\protect\citeauthoryear{Kay}{1989}]{KAY1989165}
\begin{barticle}
\bauthor{\bsnm{Kay}, \binits{K.G.}}:
\batitle{The matrix singularity problem in the time-dependent variational
  method}.
\bjtitle{Chemical Physics}
\bvolume{137}(\bissue{1}),
\bfpage{165}--\blpage{175}
(\byear{1989})
\end{barticle}
\endbibitem

\bibitem[\protect\citeauthoryear{Sawada et~al.}{1985}]{10.1063/1.449204}
\begin{barticle}
\bauthor{\bsnm{Sawada}, \binits{S.}},
\bauthor{\bsnm{Heather}, \binits{R.}},
\bauthor{\bsnm{Jackson}, \binits{B.}},
\bauthor{\bsnm{Metiu}, \binits{H.}}:
\batitle{{A strategy for time dependent quantum mechanical calculations using a
  Gaussian wave packet representation of the wave function}}.
\bjtitle{The Journal of Chemical Physics}
\bvolume{83}(\bissue{6}),
\bfpage{3009}--\blpage{3027}
(\byear{1985})
\end{barticle}
\endbibitem

\bibitem[\protect\citeauthoryear{Rowan et~al.}{2020}]{PhysRevE.101.023313}
\begin{barticle}
\bauthor{\bsnm{Rowan}, \binits{K.}},
\bauthor{\bsnm{Schatzki}, \binits{L.}},
\bauthor{\bsnm{Zaklama}, \binits{T.}},
\bauthor{\bsnm{Suzuki}, \binits{Y.}},
\bauthor{\bsnm{Watanabe}, \binits{K.}},
\bauthor{\bsnm{Varga}, \binits{K.}}:
\batitle{Simulation of a hydrogen atom in a laser field using the
  time-dependent variational principle}.
\bjtitle{Phys. Rev. E}
\bvolume{101},
\bfpage{023313}
(\byear{2020})
\end{barticle}
\endbibitem

\bibitem[\protect\citeauthoryear{Feischl et~al.}{2024}]{feischl2024regularized}
\begin{botherref}
\oauthor{\bsnm{Feischl}, \binits{M.}},
\oauthor{\bsnm{Lasser}, \binits{C.}},
\oauthor{\bsnm{Lubich}, \binits{C.}},
\oauthor{\bsnm{Nick}, \binits{J.}}:
Regularized dynamical parametric approximation.
arXiv preprint arXiv:2403.19234
(2024)
\end{botherref}
\endbibitem

\bibitem[\protect\citeauthoryear{Kvaal et~al.}{2023}]{kvaal2023need}
\begin{botherref}
\oauthor{\bsnm{Kvaal}, \binits{S.}},
\oauthor{\bsnm{Lasser}, \binits{C.}},
\oauthor{\bsnm{Pedersen}, \binits{T.B.}},
\oauthor{\bsnm{Adamowicz}, \binits{L.}}:
No need for a grid: Adaptive fully-flexible gaussians for the time-dependent
  Schr\"odinger equation
(2023)
\end{botherref}
\endbibitem

\bibitem[\protect\citeauthoryear{Lubich and
  Oseledets}{2014}]{lubich2014projector}
\begin{barticle}
\bauthor{\bsnm{Lubich}, \binits{C.}},
\bauthor{\bsnm{Oseledets}, \binits{I.V.}}:
\batitle{A projector-splitting integrator for dynamical low-rank
  approximation}.
\bjtitle{BIT Numerical Mathematics}
\bvolume{54}(\bissue{1}),
\bfpage{171}--\blpage{188}
(\byear{2014})
\end{barticle}
\endbibitem

\bibitem[\protect\citeauthoryear{Ceruti and Lubich}{2022}]{Ceruti2022}
\begin{barticle}
\bauthor{\bsnm{Ceruti}, \binits{G.}},
\bauthor{\bsnm{Lubich}, \binits{C.}}:
\batitle{An unconventional robust integrator for dynamical low-rank
  approximation}.
\bjtitle{BIT Numerical Mathematics}
\bvolume{62}(\bissue{1}),
\bfpage{23}--\blpage{44}
(\byear{2022})
\end{barticle}
\endbibitem

\bibitem[\protect\citeauthoryear{Lubich and
  Nick}{2025}]{lubich2025regularizeddynamicalparametricapproximation}
\begin{botherref}
\oauthor{\bsnm{Lubich}, \binits{C.}},
\oauthor{\bsnm{Nick}, \binits{J.}}:
Regularized dynamical parametric approximation of stiff evolution problems.
arXiv
\textbf{2501.12118}
(2025)
\end{botherref}
\endbibitem

\bibitem[\protect\citeauthoryear{Berman and
  Peherstorfer}{2024}]{berman2024randomized}
\begin{botherref}
\oauthor{\bsnm{Berman}, \binits{J.}},
\oauthor{\bsnm{Peherstorfer}, \binits{B.}}:
Randomized sparse {Neural} {Galerkin} schemes for solving evolution equations
  with deep networks.
Advances in Neural Information Processing Systems
\textbf{36}
(2024)
\end{botherref}
\endbibitem

\bibitem[\protect\citeauthoryear{Lam
  et~al.}{2024}]{lam2024randomizedlowrankrungekuttamethods}
\begin{botherref}
\oauthor{\bsnm{Lam}, \binits{H.Y.}},
\oauthor{\bsnm{Ceruti}, \binits{G.}},
\oauthor{\bsnm{Kressner}, \binits{D.}}:
Randomized low-rank {Runge}-{Kutta} methods.
arXiv
\textbf{2409.06384}
(2024)
\end{botherref}
\endbibitem

\bibitem[\protect\citeauthoryear{Lindsey}{2025}]{lindsey2025mne}
\begin{botherref}
\oauthor{\bsnm{Lindsey}, \binits{M.}}:
Mne: overparametrized neural evolution with applications to diffusion processes
  and sampling.
arXiv preprint arXiv:2502.03645
(2025)
\end{botherref}
\endbibitem

\bibitem[\protect\citeauthoryear{Lie et~al.}{2022}]{Lie2022}
\begin{barticle}
\bauthor{\bsnm{Lie}, \binits{H.C.}},
\bauthor{\bsnm{Stahn}, \binits{M.}},
\bauthor{\bsnm{Sullivan}, \binits{T.J.}}:
\batitle{Randomised one-step time integration methods for deterministic
  operator differential equations}.
\bjtitle{Calcolo}
\bvolume{59}(\bissue{1}),
\bfpage{13}
(\byear{2022})
\end{barticle}
\endbibitem

\bibitem[\protect\citeauthoryear{Abdulle and Garegnani}{2020}]{Abdulle2020}
\begin{barticle}
\bauthor{\bsnm{Abdulle}, \binits{A.}},
\bauthor{\bsnm{Garegnani}, \binits{G.}}:
\batitle{Random time step probabilistic methods for uncertainty quantification
  in chaotic and geometric numerical integration}.
\bjtitle{Statistics and Computing}
\bvolume{30}(\bissue{4}),
\bfpage{907}--\blpage{932}
(\byear{2020})
\end{barticle}
\endbibitem

\bibitem[\protect\citeauthoryear{Gower et~al.}{2019}]{gower2019rsn}
\begin{botherref}
\oauthor{\bsnm{Gower}, \binits{R.}},
\oauthor{\bsnm{Kovalev}, \binits{D.}},
\oauthor{\bsnm{Lieder}, \binits{F.}},
\oauthor{\bsnm{Richt{\'a}rik}, \binits{P.}}:
Rsn: randomized subspace newton.
Advances in Neural Information Processing Systems
\textbf{32}
(2019)
\end{botherref}
\endbibitem

\bibitem[\protect\citeauthoryear{LeJeune et~al.}{2024}]{doi:10.1137/22M1530264}
\begin{barticle}
\bauthor{\bsnm{LeJeune}, \binits{D.}},
\bauthor{\bsnm{Patil}, \binits{P.}},
\bauthor{\bsnm{Javadi}, \binits{H.}},
\bauthor{\bsnm{Baraniuk}, \binits{R.G.}},
\bauthor{\bsnm{Tibshirani}, \binits{R.J.}}:
\batitle{Asymptotics of the sketched pseudoinverse}.
\bjtitle{SIAM Journal on Mathematics of Data Science}
\bvolume{6}(\bissue{1}),
\bfpage{199}--\blpage{225}
(\byear{2024})
\end{barticle}
\endbibitem

\bibitem[\protect\citeauthoryear{Dzahini and
  Wild}{2024}]{doi:10.1137/22M1524072}
\begin{barticle}
\bauthor{\bsnm{Dzahini}, \binits{K.J.}},
\bauthor{\bsnm{Wild}, \binits{S.M.}}:
\batitle{Stochastic trust-region algorithm in random subspaces with convergence
  and expected complexity analyses}.
\bjtitle{SIAM Journal on Optimization}
\bvolume{34}(\bissue{3}),
\bfpage{2671}--\blpage{2699}
(\byear{2024})
\end{barticle}
\endbibitem

\bibitem[\protect\citeauthoryear{Romanov et~al.}{2025}]{romanov2025newton}
\begin{bchapter}
\bauthor{\bsnm{Romanov}, \binits{E.}},
\bauthor{\bsnm{Zhang}, \binits{F.}},
\bauthor{\bsnm{Pilanci}, \binits{M.}}:
\bctitle{Newton meets marchenko-pastur: Massively parallel second-order
  optimization with hessian sketching and debiasing}.
In: \bbtitle{The Thirteenth International Conference on Learning
  Representations}
(\byear{2025})
\end{bchapter}
\endbibitem

\bibitem[\protect\citeauthoryear{Quarteroni et~al.}{2007}]{Quarterioni}
\begin{bbook}
\bauthor{\bsnm{Quarteroni}, \binits{A.}},
\bauthor{\bsnm{Sacco}, \binits{R.}},
\bauthor{\bsnm{Saleri}, \binits{F.}}:
\bbtitle{Numerical Mathematics}.
\bpublisher{Springer},
\blocation{New York}
(\byear{2007})
\end{bbook}
\endbibitem

\bibitem[\protect\citeauthoryear{Halko et~al.}{2011}]{halko2011finding}
\begin{barticle}
\bauthor{\bsnm{Halko}, \binits{N.}},
\bauthor{\bsnm{Martinsson}, \binits{P.}},
\bauthor{\bsnm{Tropp}, \binits{J.}}:
\batitle{Finding structure with randomness: Probabilistic algorithms for
  constructing approximate matrix decompositions}.
\bjtitle{SIAM Review}
\bvolume{53}(\bissue{2}),
\bfpage{217}--\blpage{288}
(\byear{2011})
\end{barticle}
\endbibitem

\bibitem[\protect\citeauthoryear{Woodruff}{2014}]{woodruff2014sketching}
\begin{barticle}
\bauthor{\bsnm{Woodruff}, \binits{D.P.}}:
\batitle{Sketching as a tool for numerical linear algebra}.
\bjtitle{Foundations and Trends{\textregistered} in Theoretical Computer
  Science}
\bvolume{10}(\bissue{1--2}),
\bfpage{1}--\blpage{157}
(\byear{2014})
\end{barticle}
\endbibitem

\bibitem[\protect\citeauthoryear{Sarlos}{2006}]{sarlos2006improved}
\begin{bchapter}
\bauthor{\bsnm{Sarlos}, \binits{T.}}:
\bctitle{Improved approximation algorithms for large matrices via random
  projections}.
In: \bbtitle{2006 47th Annual IEEE Symposium on Foundations of Computer Science
  (FOCS'06)},
pp. \bfpage{143}--\blpage{152}
(\byear{2006}).
\bcomment{IEEE}
\end{bchapter}
\endbibitem

\bibitem[\protect\citeauthoryear{Rokhlin and Tygert}{2008}]{rokhlin2008fast}
\begin{barticle}
\bauthor{\bsnm{Rokhlin}, \binits{V.}},
\bauthor{\bsnm{Tygert}, \binits{M.}}:
\batitle{A fast randomized algorithm for overdetermined linear least-squares
  regression}.
\bjtitle{Proceedings of the National Academy of Sciences}
\bvolume{105}(\bissue{36}),
\bfpage{13212}--\blpage{13217}
(\byear{2008})
\end{barticle}
\endbibitem

\bibitem[\protect\citeauthoryear{Drineas et~al.}{2011}]{drineas2011faster}
\begin{barticle}
\bauthor{\bsnm{Drineas}, \binits{P.}},
\bauthor{\bsnm{Mahoney}, \binits{M.W.}},
\bauthor{\bsnm{Muthukrishnan}, \binits{S.}},
\bauthor{\bsnm{Sarl{\'o}s}, \binits{T.}}:
\batitle{Faster least squares approximation}.
\bjtitle{Numerische mathematik}
\bvolume{117}(\bissue{2}),
\bfpage{219}--\blpage{249}
(\byear{2011})
\end{barticle}
\endbibitem

\bibitem[\protect\citeauthoryear{Vershynin}{2018}]{vershynin2018high}
\begin{bbook}
\bauthor{\bsnm{Vershynin}, \binits{R.}}:
\bbtitle{High-dimensional Probability: An Introduction with Applications in
  Data Science}
vol. \bseriesno{47}.
\bpublisher{Cambridge University Press},
\blocation{Cambridge, UK}
(\byear{2018})
\end{bbook}
\endbibitem

\bibitem[\protect\citeauthoryear{Gu}{2015}]{gu2015subspace}
\begin{barticle}
\bauthor{\bsnm{Gu}, \binits{M.}}:
\batitle{Subspace iteration randomization and singular value problems}.
\bjtitle{SIAM Journal on Scientific Computing}
\bvolume{37}(\bissue{3}),
\bfpage{1139}--\blpage{1173}
(\byear{2015})
\end{barticle}
\endbibitem

\bibitem[\protect\citeauthoryear{Saibaba}{2019}]{saibaba2019randomized}
\begin{barticle}
\bauthor{\bsnm{Saibaba}, \binits{A.K.}}:
\batitle{Randomized subspace iteration: Analysis of canonical angles and
  unitarily invariant norms}.
\bjtitle{SIAM Journal on Matrix Analysis and Applications}
\bvolume{40}(\bissue{1}),
\bfpage{23}--\blpage{48}
(\byear{2019})
\end{barticle}
\endbibitem

\bibitem[\protect\citeauthoryear{Dong et~al.}{2024}]{dong2024efficient}
\begin{barticle}
\bauthor{\bsnm{Dong}, \binits{Y.}},
\bauthor{\bsnm{Martinsson}, \binits{P.-G.}},
\bauthor{\bsnm{Nakatsukasa}, \binits{Y.}}:
\batitle{Efficient bounds and estimates for canonical angles in randomized
  subspace approximations}.
\bjtitle{SIAM Journal on Matrix Analysis and Applications}
\bvolume{45}(\bissue{4}),
\bfpage{1978}--\blpage{2006}
(\byear{2024})
\end{barticle}
\endbibitem

\bibitem[\protect\citeauthoryear{Martinsson and
  Tropp}{2020}]{martinsson2020randomized}
\begin{barticle}
\bauthor{\bsnm{Martinsson}, \binits{P.-G.}},
\bauthor{\bsnm{Tropp}, \binits{J.A.}}:
\batitle{Randomized numerical linear algebra: Foundations and algorithms}.
\bjtitle{Acta Numerica}
\bvolume{29},
\bfpage{403}--\blpage{572}
(\byear{2020})
\end{barticle}
\endbibitem

\bibitem[\protect\citeauthoryear{Derezi{\'n}ski
  et~al.}{2024}]{derezinski2024fine}
\begin{botherref}
\oauthor{\bsnm{Derezi{\'n}ski}, \binits{M.}},
\oauthor{\bsnm{LeJeune}, \binits{D.}},
\oauthor{\bsnm{Needell}, \binits{D.}},
\oauthor{\bsnm{Rebrova}, \binits{E.}}:
Fine-grained analysis and faster algorithms for iteratively solving linear
  systems.
arXiv preprint arXiv:2405.05818
(2024)
\end{botherref}
\endbibitem

\bibitem[\protect\citeauthoryear{Derezinski
  et~al.}{2021}]{derezinski2021sparse}
\begin{bchapter}
\bauthor{\bsnm{Derezinski}, \binits{M.}},
\bauthor{\bsnm{Liao}, \binits{Z.}},
\bauthor{\bsnm{Dobriban}, \binits{E.}},
\bauthor{\bsnm{Mahoney}, \binits{M.}}:
\bctitle{Sparse sketches with small inversion bias}.
In: \bbtitle{Conference on Learning Theory},
pp. \bfpage{1467}--\blpage{1510}
(\byear{2021}).
\bcomment{PMLR}
\end{bchapter}
\endbibitem

\bibitem[\protect\citeauthoryear{TROPP}{2011}]{doi:10.1142/S1793536911000787}
\begin{barticle}
\bauthor{\bsnm{TROPP}, \binits{J.A.}}:
\batitle{Improved analysis of the subsampled randomized hadamard transform}.
\bjtitle{Advances in Adaptive Data Analysis}
\bvolume{03}(\bissue{01n02}),
\bfpage{115}--\blpage{126}
(\byear{2011})
\end{barticle}
\endbibitem

\bibitem[\protect\citeauthoryear{Johnson et~al.}{1995}]{JohnsonRV}
\begin{bbook}
\bauthor{\bsnm{Johnson}, \binits{N.L.}},
\bauthor{\bsnm{Kotz}, \binits{S.}},
\bauthor{\bsnm{Balakrishnan}, \binits{N.}}:
\bbtitle{Continuous Univariate Distributions}.
\bpublisher{Wiley},
\blocation{New Jersey, USA}
(\byear{1995})
\end{bbook}
\endbibitem

\bibitem[\protect\citeauthoryear{Laurent and
  Massart}{2000}]{10.1214/aos/1015957395}
\begin{barticle}
\bauthor{\bsnm{Laurent}, \binits{B.}},
\bauthor{\bsnm{Massart}, \binits{P.}}:
\batitle{{Adaptive estimation of a quadratic functional by model selection}}.
\bjtitle{The Annals of Statistics}
\bvolume{28}(\bissue{5}),
\bfpage{1302}--\blpage{1338}
(\byear{2000})
\end{barticle}
\endbibitem

\bibitem[\protect\citeauthoryear{Wachter}{1980}]{wachter1980limiting}
\begin{botherref}
\oauthor{\bsnm{Wachter}, \binits{K.W.}}:
The limiting empirical measure of multiple discriminant ratios.
The Annals of Statistics,
937--957
(1980)
\end{botherref}
\endbibitem

\bibitem[\protect\citeauthoryear{Dumitriu and
  Paquette}{2012}]{dumitriu2012global}
\begin{barticle}
\bauthor{\bsnm{Dumitriu}, \binits{I.}},
\bauthor{\bsnm{Paquette}, \binits{E.}}:
\batitle{Global fluctuations for linear statistics of $\beta$-jacobi
  ensembles}.
\bjtitle{Random Matrices: Theory and Applications}
\bvolume{1}(\bissue{04}),
\bfpage{1250013}
(\byear{2012})
\end{barticle}
\endbibitem

\bibitem[\protect\citeauthoryear{Holcomb and
  Moreno~Flores}{2012}]{holcomb2012edge}
\begin{barticle}
\bauthor{\bsnm{Holcomb}, \binits{D.}},
\bauthor{\bsnm{Moreno~Flores}, \binits{G.R.}}:
\batitle{Edge scaling of the $\beta$-jacobi ensemble}.
\bjtitle{Journal of Statistical Physics}
\bvolume{149}(\bissue{6}),
\bfpage{1136}--\blpage{1160}
(\byear{2012})
\end{barticle}
\endbibitem

\bibitem[\protect\citeauthoryear{Tracy and Widom}{1996}]{tracy1996orthogonal}
\begin{barticle}
\bauthor{\bsnm{Tracy}, \binits{C.A.}},
\bauthor{\bsnm{Widom}, \binits{H.}}:
\batitle{On orthogonal and symplectic matrix ensembles}.
\bjtitle{Communications in Mathematical Physics}
\bvolume{177}(\bissue{3}),
\bfpage{727}--\blpage{754}
(\byear{1996})
\end{barticle}
\endbibitem

\bibitem[\protect\citeauthoryear{Geman}{1980}]{geman1980limit}
\begin{botherref}
\oauthor{\bsnm{Geman}, \binits{S.}}:
A limit theorem for the norm of random matrices.
The Annals of Probability,
252--261
(1980)
\end{botherref}
\endbibitem

\bibitem[\protect\citeauthoryear{Silverstein}{1985}]{silverstein1985smallest}
\begin{botherref}
\oauthor{\bsnm{Silverstein}, \binits{J.W.}}:
The smallest eigenvalue of a large dimensional {Wishart} matrix.
The Annals of Probability,
1364--1368
(1985)
\end{botherref}
\endbibitem

\bibitem[\protect\citeauthoryear{Kingma and Ba}{2015}]{KingmaBa2014}
\begin{bchapter}
\bauthor{\bsnm{Kingma}, \binits{D.P.}},
\bauthor{\bsnm{Ba}, \binits{J.}}:
\bctitle{Adam: A method for stochastic optimization}.
In: \bbtitle{3rd International Conference on Learning Representations (ICLR)}
(\byear{2015})
\end{bchapter}
\endbibitem

\bibitem[\protect\citeauthoryear{Loshchilov and
  Hutter}{2017}]{LoshchilovHutter2016}
\begin{bchapter}
\bauthor{\bsnm{Loshchilov}, \binits{I.}},
\bauthor{\bsnm{Hutter}, \binits{F.}}:
\bctitle{{SGDR}: Stochastic gradient descent with warm restarts}.
In: \bbtitle{Proceedings of the 5th International Conference on Learning
  Representations (ICLR)}
(\byear{2017})
\end{bchapter}
\endbibitem

\bibitem[\protect\citeauthoryear{Ceruti et~al.}{2022}]{Ceruti2022a}
\begin{barticle}
\bauthor{\bsnm{Ceruti}, \binits{G.}},
\bauthor{\bsnm{Kusch}, \binits{J.}},
\bauthor{\bsnm{Lubich}, \binits{C.}}:
\batitle{A rank-adaptive robust integrator for dynamical low-rank
  approximation}.
\bjtitle{BIT Numerical Mathematics}
\bvolume{62}(\bissue{4}),
\bfpage{1149}--\blpage{1174}
(\byear{2022})
\end{barticle}
\endbibitem

\end{thebibliography}

\appendix

\section{Beta concentration}

\begin{proof}[Proof of Lemma~\ref{lm:BetaConcentration}]
The random variable $Z$ has the same distribution as the ratio $X/(X + Y)$ with independent random variables $X \sim \chi^2(\ell)$ and $Y \sim \chi^2(p - \ell)$; see, e.g., \cite[Section 2, p.~212]{JohnsonRV}. %
Representing $Z$ in distribution as the ratio of $\chi^2$ random variables allows us to derive a bound on $Z$ from bounds on $\chi^2$ random variables. Let us therefore recall \cite[Lemma~1]{10.1214/aos/1015957395}, which states that for 
$X$ and $Y$ we obtain for $\tau > 0$
\begin{align}
\mathbb{P}\left[X \geq \ell + 2\sqrt{\ell \tau} + 2\tau\right] \leq \exp(-\tau)\,, & \mathbb{P}\left[Y \leq (p - \ell) - 2\sqrt{(p - \ell)\tau}\right] \leq \exp(-\tau)\,,\label{eq:ProofCondAuxLemmaCB1}\\
\mathbb{P}\left[X \leq \ell - 2\sqrt{\ell\tau}\right] \leq \exp(-\tau)\,, & \mathbb{P}\left[Y \geq (p - \ell) + 2\sqrt{(p - \ell)\tau} + 2\tau\right] \leq \exp(-\tau)\label{eq:ProofCondAuxLemmaCB2}\,.
\end{align}
We first consider $\ell \leq p/2$ so that the mean of $Z$ is $\frac{\ell/2}{\ell/2 + (p - \ell)/2} \leq 1/2$. To upper bound $|Z - \frac{\ell}{p}| > \epsilon\frac{\ell}{p}$ in probability, we lower bound the probability of $|Z - \frac{\ell}{p}| \leq \epsilon\frac{\ell}{p}$ and first consider the direction $Z - \frac{\ell}{p} \leq \epsilon \frac{\ell}{p}$. We invoke \eqref{eq:ProofCondAuxLemmaCB1} to obtain the combined bound
\[
\mathbb{P}\left[\{X \leq \ell + 2\sqrt{\ell \tau} + 2\tau\} \cap \{Y \geq (p - \ell) - 2\sqrt{(p - \ell)\tau}\}\right] \geq 1 - 2\exp(-\tau)
\]
Recall that $X, Y \geq 0$ so that the ratio $X/(X + Y)$ increases with increasing $X$ and decreasing $Y$ so that it is sufficient to have an upper bound on $X$ and a lower bound on $Y$ to upper bound the ratio $X/(X + Y)$. Thus, we obtain that
\begin{equation}\label{eq:CondProofIntermediateFirstXXYBound}
\frac{X}{X + Y} \leq \frac{\ell + 2\sqrt{\ell \tau} + 2\tau}{\ell + 2\sqrt{\ell \tau} + 2\tau + (p - \ell) - 2\sqrt{(p - \ell)\tau}}
\end{equation}
holds with probability at least $1 - 2\exp(-\tau)$. Recall that $0 < \epsilon \leq \min\{1, \frac{p - \ell}{\ell}\}$. Because we are in the case of $\ell \leq p/2$, we have $(p - \ell)/\ell \geq 1$ and thus $0 < \epsilon \leq 1$. We set $\tau = \epsilon^2 \ell/64$ and obtain
\begin{equation}\label{eq:CondProofIntermediateEpsDefs}
2 \sqrt{\ell \tau} = \frac{\epsilon \ell}{4}\,, \quad 2\tau = \frac{\epsilon^2 \ell}{32}\,,\quad 2\sqrt{(p - \ell)\tau} = \frac{\epsilon}{4}\sqrt{\ell(p - \ell)}\,.
\end{equation}
Building on \eqref{eq:CondProofIntermediateFirstXXYBound} and plugging in \eqref{eq:CondProofIntermediateEpsDefs}, we obtain
\begin{equation}\label{eq:CondProofIntermediateFirstXXYBound2}
\frac{X}{X + Y} \leq \frac{1 + \epsilon/4 + \epsilon^2/32}{\frac{p}{\ell} + \frac{\epsilon}{4}(1 - \sqrt{\frac{p}{\ell} - 1}) + \epsilon^2/32}
\end{equation}
with probability at least $1 - 2\exp(-\tau)$. Because we currently consider the case $\ell \leq p/2$, we have
\begin{equation}\label{eq:CondProofIntermediateLP2Condition}
1  \leq \sqrt{\frac{p}{\ell} - 1} \leq \frac{p}{2\ell}\,,
\end{equation}
which allows lower bounding the denominator of \eqref{eq:CondProofIntermediateFirstXXYBound2} as
\[
\frac{p}{\ell} + \frac{\epsilon}{4}\left(1 - \sqrt{\frac{p}{\ell} - 1}\right) + \frac{\epsilon^2}{32} \geq \frac{p}{\ell} + \frac{\epsilon}{4}\left(1 - \frac{p}{2\ell}\right) + \frac{\epsilon^2}{32} = \frac{p}{\ell}\left(1 - \frac{\epsilon}{8}\right) + \frac{\epsilon}{4} + \frac{\epsilon^2}{32}\,.
\]
It is now left to show that
\begin{equation}\label{eq:CondProofIntermediateFirstXXYBound3}
\frac{1 + \frac{\epsilon}{4} + \frac{\epsilon^2}{32}}{\frac{p}{\ell}\left(1 - \frac{\epsilon}{8}\right) + \frac{\epsilon}{4} + \frac{\epsilon^2}{32}} \leq (1 + \epsilon)\frac{\ell}{p}
\end{equation}
holds. To see that \eqref{eq:CondProofIntermediateFirstXXYBound3} is indeed true, drop $\epsilon/4 + \epsilon^2/32 \geq 0$ in the denominator on the left-hand side, which increases the fraction and thus proving the resulting stronger inequality is sufficient. Next, cancel $p/\ell$ on the left- and right-hand side, and notice that $0 < \epsilon \leq 1$ implies that the denominator on the left-hand side is positive. Thus, we obtain that \eqref{eq:CondProofIntermediateFirstXXYBound3} holds because
\[
1 + \frac{\epsilon}{4} + \frac{\epsilon^2}{32} \leq (1 + \epsilon)\left(1 - \frac{\epsilon}{8}\right) = 1 + \frac{7}{8}\epsilon - \frac{1}{8}\epsilon^2\,.
\]
holds, which holds because $0 < \epsilon \leq 1$. Thus, we have shown for $\ell \leq p/2$ that
\begin{equation}\label{eq:CondProofIntermediateResult1}
\mathbb{P}\left[Z - \frac{\ell}{p} \leq \epsilon \frac{\ell}{p}\right] \geq 1 - 2\exp\left(-\frac{\epsilon^2 \ell}{64}\right)
\end{equation}
for $0 < \epsilon \leq 1$.

Still for $\ell \leq p/2$, we now consider $-(Z - \frac{\ell}{p}) \leq \epsilon\frac{\ell}{p} \iff Z \geq (1 - \epsilon)\frac{\ell}{p}$. For this, we now lower bound the ratio $X/(X + Y)$. Because $X, Y \geq 0$, the ratio $X/(X + Y)$ decreases with decreasing $X$ and increasing $Y$. Thus, it is sufficient to have a lower bound on $X$ and an upper bound on $Y$ to lower bound the ratio. With  \eqref{eq:ProofCondAuxLemmaCB2}, we obtain that
\begin{equation}\label{eq:ProofCondIntermediateLP2Bound2}
\frac{X}{X + Y} \geq \frac{\ell - 2\sqrt{\ell \tau}}{p - 2\sqrt{\ell\tau}  + 2\sqrt{(p-\ell)\tau} + 2\tau} = 
\frac{1 - \frac{\epsilon}{4}}{\frac{p}{\ell} + \frac{\epsilon}{4}(\sqrt{\frac{p}{\ell} - 1} - 1) + \frac{\epsilon^2}{32}}\,,
\end{equation}
holds with probability at least $1 - 2\exp(-\tau)$, where we set $\tau = \epsilon^2 \ell/64$ again. We now need to show that \eqref{eq:ProofCondIntermediateLP2Bound2} is lower bounded by $(1 - \epsilon)\frac{\ell}{p}$. First notice that the denominator of the last equality in \eqref{eq:ProofCondIntermediateLP2Bound2} is positive because of \eqref{eq:CondProofIntermediateLP2Condition}, and thus 
\[
\frac{1 - \frac{\epsilon}{4}}{\frac{p}{\ell} + \frac{\epsilon}{4}(\sqrt{\frac{p}{\ell} - 1} - 1) + \frac{\epsilon^2}{32}} \geq (1 - \epsilon)\frac{\ell}{p} \iff \frac{3}{4}\frac{p}{\ell} - \frac{1}{4}(1 - \epsilon)\left(\sqrt{\frac{p}{\ell} - 1} - 1\right) - (1 - \epsilon)\frac{\epsilon}{32} \geq 0\,,
\]
where we used that $\epsilon > 0$. We invoke \eqref{eq:CondProofIntermediateLP2Condition} and $0 < \epsilon \leq 1$ again to obtain
\[
\frac{1}{4}(1 - \epsilon)\left(\sqrt{\frac{p}{\ell} - 1} - 1\right) \leq \frac{1}{4}\left(\sqrt{\frac{p}{\ell} - 1} - 1\right) \leq \frac{1}{8}\frac{p}{\ell}\,,\qquad (1 - \epsilon)\frac{\epsilon}{32} \leq \frac{1}{32}\,.
\]
Thus, it is sufficient that
\[
\frac{3}{4}\frac{p}{\ell} - \frac{1}{8}\frac{p}{\ell} - \frac{1}{32} \geq 0
\]
holds, which is true because $\ell \leq p/2$. Together with \eqref{eq:CondProofIntermediateResult1} and union bound, we obtain that
\[
\mathbb{P}\left[\left|Z - \frac{\ell}{p}\right| \leq \epsilon \frac{\ell}{p}\right] \geq 1 - 4\exp\left(-\frac{\epsilon^2 \ell}{64}\right)
\]
holds for $0 < \epsilon \leq 1$ if $\ell \leq p/2$. 

Let us now consider the case $\ell > p/2$, which means that the Beta random variable $Z$ has mean $\ell/p > 1/2$. We introduce the random variable $Z' = 1 - Z$ with mean $\mu' = 1 - \frac{\ell}{p} = (p - \ell)/p \in (0, \frac{1}{2})$. Notice that $|Z - \frac{\ell}{p}| = |Z' - \mu'|$. Thus, if we bound the probability of the event $\{|Z' - \mu'| > \epsilon' \mu'\}$ with $\epsilon' = \frac{\epsilon\ell}{p\mu'} = \epsilon \ell /(p - \ell)$, then we also obtain a bound on the desired probability of $\{|Z - \frac{\ell}{p}| > \epsilon \frac{\ell}{p}\}$. Furthermore, the random variable $Z' = 1 - Z$ has the same distribution as the ratio $Y/(X + Y)$, which is a Beta random variable $\operatorname{Beta}(\frac{p - \ell}{2}, \frac{\ell}{2})$. Set $\ell' = p - \ell$ and $p' = p$ to obtain that $Z'$ is a Beta variable $\operatorname{Beta}(\frac{\ell'}{2}, \frac{p' - \ell'}{2})$ with $\ell' \leq p'/2$ and $\ell'/p' = (p - \ell)/p = \mu'$. Thus, analogous to the case $\ell \leq p/2$, we obtain the bound
\begin{equation}\label{eq:CondProofIntermediateResult3}
\mathbb{P}\left[\left| Z' - \mu' \right| > \epsilon'\mu'\right] \leq 4\exp\left(-\frac{(\epsilon')^2 (p - \ell)}{64}\right)
\end{equation}
for $0 < \epsilon' \leq 1$. We now invoke that $0 < \epsilon \leq \min\{1, (p - \ell)/\ell\}$ and that we are interested in $\epsilon' = \epsilon \ell/(p - \ell)$, which implies that $0 < \epsilon' \leq 1$. Furthermore, because $\epsilon' = \epsilon \ell/(p - \ell)$ we have that \eqref{eq:CondProofIntermediateResult3} implies
\[
\mathbb{P}\left[|Z - \frac{\ell}{p}| > \epsilon \frac{\ell}{p}\right] = \mathbb{P}\left[\left| Z' - \mu' \right| > \epsilon'\mu'\right] \leq 4 \exp\left(- \frac{\epsilon^2 \ell^2}{64 (p - \ell)}\right) \leq 4\exp\left(-\frac{\epsilon^2\ell}{64}\right)
\]
for $0 < \epsilon \leq \min\{1, (p - \ell)/\ell\}$ and $\ell > p/2$, which shows the lemma. 
\end{proof}

\end{document}